\documentclass[a4paper,11pt]{article}
\usepackage[lmargin=30mm,rmargin=30mm,tmargin=26mm,bmargin=26mm]{geometry}
\usepackage{amsmath,amssymb,amsthm,mathtools}
\usepackage{booktabs}
\usepackage{array}
\usepackage[dvipsnames]{xcolor}
\usepackage[colorlinks=true,linkcolor=blue!50!black,citecolor=blue!50!black,urlcolor=blue!50!black]{hyperref}
\usepackage{enumitem}
\usepackage{tikz}
\usepackage{alltt}
\usepackage{sectsty}
\newcommand{\defn}[1]{\textcolor{Maroon}{\emph{#1}}}
\allsectionsfont{\boldmath}

\newtheorem{theorem}{Theorem}[section]
\newtheorem{lemma}[theorem]{Lemma}
\newtheorem{proposition}[theorem]{Proposition}
\newtheorem{corollary}[theorem]{Corollary}

\theoremstyle{definition}

\newcommand{\had}{\operatorname{had}}
\newcommand{\eps}{\varepsilon}
\newcommand{\cL}{\omega_{\ref*{thm:refinedlin}}}
\newcommand{\PL}{\tau_{\ref*{thm:refinedlin}}}
\newcommand{\Ck}{C_{\ref*{lem:Ck}}}
\newcommand{\Cdens}{C_{\ref*{prop:DP}}}
\newcommand{\Sst}{c_{\ref*{lem:level}}}
\newcommand{\cS}{c_{\ref*{thm:sepexplicit}}}
\newcommand{\cst}{c_{\ref*{prop:explicitbootstrap}}}
\newcommand{\LamE}{\Lambda_{\mathrm E}}
\newcommand{\LamU}{\Lambda_{\mathrm U}}
\newcommand{\LamF}{\Lambda_{\ref*{lem:weavefinal}}}
\newcommand{\KF}{\Lambda'_{\ref*{lem:weavefinal}}}
\newcommand{\CNP}{C_{\ref*{thm:NP}}}
\newcommand{\CLin}{C_{\mathrm{Lin}}}
\newcommand{\CDP}{C_{\mathrm{DP}}}
\newcommand{\cfin}{c_{\ref*{thm:all}}}

\title{\bfseries Tweaking the constant in the\\ Linear Hadwiger Theorem}
\author{David R. Wood\footnote{School of Mathematics, Monash University, Melbourne, Australia (\texttt{david.wood@monash.edu}). Research supported by the Australian Research Council.}}
\date{}

\newcommand{\doi}[1]{\href{https://doi.org/#1}{\nolinkurl{doi:#1}}}
\begin{document}
\maketitle
{\renewcommand{\thefootnote}{}\footnotetext{7th October 2026 (updated 15th October 2026)}}

\begin{abstract}
Norin and Steiner recently proved the Linear Hadwiger Conjecture. That is, there is an absolute constant $C$ such that every graph $G$ satisfies $\chi(G)\leqslant C\,\had(G)$, where $\chi(G)$ is the chromatic number and $\had(G)$ is the Hadwiger number of $G$. They estimate that their proof gives a non-optimised constant $C$ of order $10^{100}$. This paper uses AI assistance to optimise their argument, and shows that every graph $G$ satisfies $\chi(G)\,\leqslant\,1{,}118{,}430\,\had(G)$.
\end{abstract}

{\footnotesize\setlength{\parskip}{0pt}\tableofcontents}
\newpage

%=====================================================================
\section{Introduction}\label{sec:intro}
%=====================================================================

Hadwiger's Conjecture~\cite{Had43} suggests a deep relationship between complete graph minors and graph colourings, and is widely considered one of the most important open problems in graph theory. The conjecture asserts that $\chi(G)\leqslant\had(G)$ for every graph $G$. Here, the \defn{chromatic number} \defn{$\chi(G)$} is the minimum number of colours in a colouring of $G$, where a \defn{colouring} of $G$ assigns one colour to each vertex so that adjacent vertices receive distinct colours. And the \defn{Hadwiger number} \defn{$\had(G)$} is the largest integer $t$ such that the complete graph $K_t$ is a minor of $G$, where a graph $H$ is a \defn{minor} of a graph $G$ if a graph isomorphic to $H$ can be obtained from a subgraph of $G$ by contracting edges.

The conjecture is trivial for $\had(G)\leqslant2$. It was proved for $\had(G)=3$ by Hadwiger~\cite{Had43} and Dirac~\cite{Dir52}. For $\had(G)=4$ it is equivalent to the Four Colour Theorem for planar graphs~\cite{Wag37,AH77a,AHK77b,RSST97,IKMMTT26}. For $\had(G)=5$ it was reduced to the Four Colour Theorem by Robertson, Seymour and Thomas~\cite{RST93}. It is open for $\had(G)\geqslant6$. See the surveys of Seymour~\cite{Sey16} and Norin~\cite{Nor23} for background.

Much effort has gone into upper bounds on $\chi(G)$ as a function of $t=\had(G)$. Table~\ref{tab:history} summarises the history.

\begin{table}[ht]
\centering\small
\begin{tabular}{@{}lll@{}}\toprule
Date & Upper bound on $\chi(G)$ when $\had(G)=t$ & Reference\\\midrule
1964 & $2^{t-1}$ & Wagner~\cite{Wag64}\\
1968 & $O(t\log t)$ & Mader~\cite{Mad68}\\
1982/84 & $O(t\sqrt{\log t})$ & Kostochka~\cite{Kos82,Kos84}, Thomason~\cite{Tho84}\\
2001 & $(2\alpha+o(1))\,t\sqrt{\log t}$ & Thomason~\cite{Tho01}\\
2019 & $O(t(\log t)^{\beta})$ for every $\beta>\frac14$ & Norin, Postle and Song~\cite{NPS23}\\
2021 & $O(t\log\log t)$ & Delcourt and Postle~\cite{DP}\\
Sep 2026 & $O(t\sqrt{\log\log t})$ & Liu and Luo~\cite{LL26a}, Lin~\cite{Lin26b}\\
Sep 2026 & $O(t\log\log\log t)$ & Liu and Luo~\cite{LL26}\\
Oct 2026 & $O(t)$, with constant of order $10^{100}$ & Norin and Steiner~\cite{NS}\\
Oct 2026 & $1{,}118{,}430\,t$ & this paper\\\bottomrule
\end{tabular}
\caption{Upper bounds on the chromatic number of graphs with Hadwiger number $t$. Dates refer to the first appearance of each result. Here $\alpha=0.3190863\ldots$ is Thomason's constant~\cite{Tho01}. The order $10^{100}$ is the estimate of Norin and Steiner.}\label{tab:history}
\end{table}

The bounds up to 2001 came from degeneracy. A graph is \defn{$d$-degenerate} if every non-null subgraph has a vertex of degree at most $d$. A $d$-degenerate graph is $(d+1)$-colourable, by a greedy argument. Since $\had(H)\leqslant\had(G)$ for every subgraph $H$ of $G$, an upper bound on the average degree of a graph in terms of its Hadwiger number gives an upper bound on its degeneracy, and hence on its chromatic number. Mader~\cite{Mad68} showed that, for $t\geqslant2$, every graph with no $K_t$ minor has average degree $O(t\log t)$. Kostochka~\cite{Kos82,Kos84} and Thomason~\cite{Tho84} independently improved this to $O(t\sqrt{\log t})$. Thomason~\cite{Tho01} later determined the extremal function asymptotically. Thomason defined an explicit constant $\alpha=0.3190863\ldots$ and showed that every graph $G$ with no $K_t$ minor has at most $(\alpha+o(1))\,t\sqrt{\log t}\,v(G)$ edges, implying $G$ has average degree and thus chromatic number at most $(2\alpha+o(1))\,t\sqrt{\log t}$.

The bound $O(t\sqrt{\log t})$ on the average degree is tight. Bollob\'as, Catlin and Erd\H{o}s~\cite{BCE80} proved that the Hadwiger number of the random graph $G(n,\frac12)$ is almost surely $(1+o(1))\,n/\sqrt{\log_2n}$. So for each $t$ there is a graph with no $K_t$ minor and average degree $\Omega(t\sqrt{\log t})$. Hence no argument based only on degeneracy gives a bound on $\chi(G)$ better than $\Theta(t\sqrt{\log t})$. This limitation is called the \defn{degeneracy barrier}. Thomason~\cite{Tho01} showed that the constant $\alpha$ is best possible, with random graphs as extremal examples, and observed that the extremal graphs are essentially disjoint unions of dense pseudo-random graphs. Myers~\cite{Mye02} made this precise.

However, these random graphs are far from counterexamples to Hadwiger's Conjecture. By a theorem of Bollob\'as~\cite{Bol88}, $G(n,\frac12)$ almost surely has chromatic number $(1+o(1))\,n/(2\log_2n)$, which is much smaller than its Hadwiger number. Evidence for a linear bound comes from fractional colouring. Reed and Seymour~\cite{RS98} proved that $\chi_f(G)\leqslant2\had(G)$ for every graph $G$, where $\chi_f$ is the fractional chromatic number. So the fractional relaxation of Hadwiger's Conjecture holds up to a factor of~$2$. In the same paper, Reed and Seymour asked whether there is an absolute constant $C$ such that $\chi(G)\leqslant C\had(G)$ for every graph $G$. This question became known as the \defn{Linear Hadwiger Conjecture}.

The degeneracy barrier stood for more than thirty years, until Norin, Postle and Song~\cite{NPS23} broke it. They proved that graphs with no $K_t$ minor are $O(t(\log t)^\beta)$-colourable for every $\beta>\frac14$. Their approach shows that a graph with large chromatic number contains a small highly connected subgraph, and that many such subgraphs can be linked together to form a $K_t$ minor. A key tool is their density increment theorem, which finds in a graph either a denser minor or a small subgraph of nearly the same density. A density bound of Norin and Postle~\cite{NP23} for unbalanced bipartite graphs with no $K_t$ minor was used later by Delcourt and Postle~\cite{DP} to find small dense subgraphs.

The next breakthrough came from Delcourt and Postle~\cite{DP}, who proved the following reduction. If every graph with no $K_t$ minor and $O(t\log^4t)$ vertices is $O(t)$-colourable, then the Linear Hadwiger Conjecture holds. So only graphs of order $t(\log t)^{O(1)}$ need to be considered. For such a graph $G$ with $n$ vertices, the Duchet--Meyniel theorem~\cite{DM82} gives an independent set of at least $n/(2t)$ vertices, and greedy colouring then gives $\chi(G)\in O(t\log(n/t))\subseteq O(t\log\log t)$. Together with a more general form of their reduction, this gives the bound $O(t\log\log t)$ of Delcourt and Postle. Lin~\cite{Lin} later reduced the order of the graphs in this reduction to $O(t\log^2t)$.

After the reduction of Delcourt and Postle, the problem was to colour small graphs with no $K_t$ minor more efficiently. Graphs with $n\in O(t)$ vertices are trivially $O(t)$-colourable. So the difficulty lies with graphs whose order is more than linear in $t$ but at most $t(\log t)^{O(1)}$. Liu and Luo~\cite{LL26a} and, independently, Lin~\cite{Lin26b} improved the colouring of such graphs, and obtained the bound $O(t\sqrt{\log\log t})$. Liu and Luo~\cite{LL26} then improved this to $O(t\log\log\log t)$.

Shortly afterwards, Norin and Steiner~\cite{NS} proved the Linear Hadwiger Conjecture. They colour graphs with at most $t(\log t)^{1/2-\eps}$ vertices directly, using an argument based on the fractional colouring theorem of Reed and Seymour, a quantitative Pippenger--Frankl--R\"odl theorem, and linear programming (see Section~\ref{sec:NSoverview}). They then bootstrap up to graphs with $t\log^5t$ vertices, and finish with the reduction of Delcourt and Postle.

\begin{theorem}[Norin and Steiner~\cite{NS}]\label{thm:NSmain}
There is an absolute constant $C$ such that every graph $G$ satisfies $\chi(G)\leqslant C\,\had(G)$.
\end{theorem}

Because of Theorem~\ref{thm:NSmain}, all that remains in the study of Hadwiger's Conjecture is to determine the best possible constant. Norin and Steiner~\cite[Section~4]{NS} estimate that their proof gives a non-optimised constant $C$ of order $10^{100}$. They write:
\begin{quote}
``There is a lot of room for improvement, and we see no clear barriers to reducing the value of $C$ below $10^{10}$. Proving Theorem~2 with, say, $C=100$ would likely require substantial new ideas.''
\end{quote}
Here their Theorem~2 is Theorem~\ref{thm:NSmain}. Our main result confirms the first sentence of this quotation.

\begin{theorem}\label{thm:all}
Every graph $G$ satisfies $\chi(G)\leqslant1{,}118{,}430\,\had(G)$.
\end{theorem}

The author does not claim that Theorem~\ref{thm:all} is a significant improvement on the result of Norin and Steiner (which is a real breakthrough). The present paper simply reports the outcomes of an experiment in using AI to optimise the constants.

\paragraph{Statement about AI use.} I started by asking Claude Opus 5.5 to optimise the constant in the proof of Norin and Steiner~\cite{NS}. Claude identified several areas for optimisation, worked out all the details, and produced a first draft of the paper. GPT-6 checked the proof and suggested further optimisations. I also identified optimisations, and instructed Claude how to improve the exposition. This process went through several iterations. I have read and revised the whole paper so that it is human-readable. I understand the proof at a high level, but have not checked all the technical and numerical details by hand. For these details I rely on automated verification.

\paragraph{Verification.} Theorem~\ref{thm:all} has been formally verified in Lean~4~\cite{Lean4} with the Mathlib library~\cite{Mathlib}, building on the Lean library \texttt{LinearHadwiger}, which formalises the proof of Norin and Steiner (see Appendix~\ref{app:lean}). The final Lean statement is exactly Theorem~\ref{thm:all}:
\begin{alltt}\small
theorem HadwigerConstant.chromaticNumber_le_mul_hadwigerNumber
    (V : Type u) [Fintype V] (G : SimpleGraph V) :
    G.chromaticNumber \(\le\) 1118430 * hadwigerNumber G
\end{alltt}
Here \texttt{G.chromaticNumber} is the chromatic number of $G$ as defined in Mathlib, and \texttt{hadwigerNumber G} is defined from standard Mathlib notions. The verification depends only on the standard axioms of Lean. All non-trivial numerical inequalities were checked by computer. The scripts and Lean code are available as ancillary files on arXiv.

\paragraph{Organisation.} Section~\ref{sec:overview} explains how the present proof builds on the work of Delcourt and Postle and of Norin and Steiner, and where the constant in Theorem~\ref{thm:all} comes from. Section~\ref{sec:prelim} gives definitions and the results we use. Sections~\ref{sec:sep}, \ref{sec:bootstrap} and~\ref{sec:insep} contain the three main ingredients: the chromatic-separable reduction, the colouring of graphs of order $O(a\log^2a)$ by an additive bootstrap, and the chromatic-inseparable case. Section~\ref{sec:wollan} shows, using Wollan's work, that dense highly connected graphs are woven (Lemma~\ref{lem:Ck}), and Section~\ref{sec:lin} proves the refined form of Lin's theorem. Both results are used in Section~\ref{sec:insep}. Section~\ref{sec:proof} combines the three ingredients in a final separation step, and proves Theorem~\ref{thm:all}. Appendix~\ref{app:further} discusses how the constant might be reduced further. Appendix~\ref{app:lean} gives details of the computer verification.

\paragraph{Constant naming conventions.} Each constant or coefficient that is used outside the result introducing it is typically denoted by a letter with the number of that result in subscript. For example, $\Ck$ is introduced in Lemma~\ref{lem:Ck}. Constants from other papers are named after the result of this paper in which they are restated. Exceptions include Thomason's constant $\alpha$ of Section~\ref{sec:intro}, the constants $C_\beta$, $C_{\mathrm{sep}}$, $C_{B.1}$, $C_{B.8}$ and $C_{\mathrm{ins}}$ of Section~\ref{sec:overview}, which are those of Norin and Steiner and of Delcourt and Postle, and the constants ${\CDP}$ and ${\CLin}$ of Delcourt and Postle and of Lin (Sections~\ref{sec:extremal} and~\ref{sec:lin}). We use $C$ for constants in edge-density bounds, $\Lambda$ for coefficients in lower bounds on the chromatic number that force a $K_a$ minor, $c$ for coefficients in upper bounds on the chromatic number, $\tau$ for a bound on the order of a subgraph, and $\omega$ for a proportion of density. 

%=====================================================================
\section{From Delcourt--Postle and Norin--Steiner to our bound}\label{sec:overview}
%=====================================================================

This section explains how the proof of Theorem~\ref{thm:all} builds on the work of Delcourt and Postle~\cite{DP} and of Norin and Steiner~\cite{NS}. Section~\ref{sec:DPoverview} describes the reduction of Delcourt and Postle, Section~\ref{sec:NSoverview} describes the proof of Norin and Steiner, Section~\ref{sec:present} summarises how the present paper improves it, and Section~\ref{sec:constants} explains where the constant in Theorem~\ref{thm:all} comes from.

\subsection{The reduction of Delcourt and Postle}\label{sec:DPoverview}

Delcourt and Postle~\cite{DP} reduced the Linear Hadwiger Conjecture to colouring small graphs. Their ideas underlie both the proof of Norin and Steiner and the present paper. The following definition is a key ingredient. Let $s\geqslant0$. A graph $G$ is \emph{$s$-chromatic-separable} if it has two vertex-disjoint subgraphs, each with chromatic number at least $\chi(G)-s$. Otherwise $G$ is \emph{$s$-chromatic-inseparable}. (Section~\ref{sec:prelim} gives the precise definition.) To find a $K_t$ minor in a graph of large chromatic number, Delcourt and Postle use three ingredients.

\medskip\noindent\emph{Small dense pieces.} For every integer $k\geqslant t$, every graph with no $K_t$ minor and with at least ${\CDP}k$ edges per vertex has a $k$-connected subgraph with $O(t\log^3t)$ vertices~\cite[Theorem~2.3]{DP}. This is the small-dense-subgraph theorem of Delcourt and Postle, stated in Section~\ref{sec:extremal}. Lin~\cite{Lin} later reduced the order to $O(t\log t)$.

\medskip\noindent\emph{The chromatic-inseparable case.} Let $G$ be a graph. Let $\Gamma$ be the maximum chromatic number of a subgraph of $G$ that has no $K_t$ minor and $O(t\log^4t)$ vertices. There is an absolute constant $C_{\mathrm{ins}}$ such that, if $G$ is $C_{\mathrm{ins}}(t+\Gamma)$-chromatic-inseparable and $\chi(G)\geqslant2C_{\mathrm{ins}}(t+\Gamma)$, then $G$ has a $K_t$ minor~\cite[Lemma~2.5]{DP}. (In their notation, $\Gamma=tg$.) The minor is built in stages. At each stage, small dense pieces supply new branch sets. Inseparability guarantees that, after the vertices already used are removed, a highly connected subgraph of large chromatic number remains for the next stage.

\medskip\noindent\emph{The chromatic-separable case.} Here a $K_a$ minor, for $a\leqslant t$, is built by a ternary recursion on $a$. A highly connected subgraph of large chromatic number is separated twice, giving three vertex-disjoint subgraphs of large chromatic number. A rooted $K_{2a/3}$ minor is built recursively in each of them. Each branch set of the $K_a$ minor is then formed from branch sets of two of the three minors, so that any two of the new branch sets are adjacent. Delcourt and Postle combine the two cases in a single induction~\cite[Theorem~7.1]{DP}, which gives their reduction~\cite[Theorem~1.6]{DP}.

\medskip
The present paper makes both cases explicit, and recombines them so that the constants add rather than multiply.

\subsection{The proof of Norin and Steiner}\label{sec:NSoverview}

We now describe the structure of the proof of Norin and Steiner~\cite{NS}. Let $\mathrm{LH}(\beta)$ denote the following statement. There is a constant $C_\beta$ such that, for all sufficiently large real $t$, every graph $G$ with $\had(G)<t$ and $v(G)\leqslant t(\log t)^\beta$ satisfies $\chi(G)\leqslant C_\beta t$.
\begin{enumerate}[label=(\Roman*),itemsep=1ex]
\item \emph{Small graphs.} 
Norin and Steiner~\cite[Theorem~4]{NS} prove that $\mathrm{LH}(\frac12-\eps)$ holds with constant $4+\eps$ for every $\eps\in(0,\frac12)$. The proof uses three tools: the fractional colouring theorem of Reed and Seymour~\cite{RS98}, a quantitative form of the Pippenger--Frankl--R\"odl theorem on almost perfect matchings in hypergraphs, and a linear programming argument. Norin and Steiner~\cite[Theorem~9 and Appendix~A]{NS} prove the quantitative form. An earlier draft~\cite{NSdraft} of their paper states the explicit bound $\chi(G)\leqslant4\had(G)+\eps v(G)+(100/\eps)^{2000/\eps^2}$, from which we derive our base case (Theorem~\ref{thm:NSquant} and Lemma~\ref{lem:base}).
\item \emph{Bootstrapping.} 
Norin and Steiner~\cite[Theorem~5]{NS} prove that $\mathrm{LH}(\beta)$ implies $\mathrm{LH}(\frac43\beta)$, with new constant $C_{\frac43\beta}=4C_\beta+2C_{\mathrm{sep}}(1+C_\beta)$. Here $C_{\mathrm{sep}}$ is the constant in the chromatic-separable reduction of Norin and Steiner~\cite[Theorem~15]{NS}, which adapts the chromatic-separable case of Delcourt and Postle described above.
\item \emph{Large graphs.} By the reduction of Delcourt and Postle described above~\cite[Theorem~1.6]{DP}, $\mathrm{LH}(5)$ implies Theorem~\ref{thm:NSmain}. This step multiplies the constant $C_5$ by a further constant factor.
\end{enumerate}

Norin and Steiner start from exponent $\frac13$, and apply Part~(II) ten times to reach exponent $5$. In the notation of Norin and Steiner, $C_{\mathrm{sep}}=39\cdot35\cdot C_{B.1}^2\cdot C_{B.8}$ in~\cite[Section~B.3 and the proof of Lemma~B.12]{NS}. The constant $C_{B.1}$ is the density constant of the small-dense-subgraph theorem of Delcourt and Postle (Section~\ref{sec:extremal}, where it is ${\CDP}=3{,}072{,}000$). Norin and Steiner need $C_{B.1}\geqslant480$, and they need~\cite[Theorem~B.1]{NS}, which is this theorem, to hold with $C_{B.1}$ in place of ${\CDP}$. This is known for $C_{B.1}={\CDP}$, but not for $C_{B.1}=480$. Also $C_{B.8}$ is the linkage constant of Bollob\'as and Thomason~\cite{BT96} in the form of~\cite[Theorem~B.8]{NS}, and Norin and Steiner only require $C_{B.8}\geqslant1$. By the theorem of Thomas and Wollan~\cite[Corollary~1.3]{TW05}, we may take $C_{B.8}=10$. Even with $C_{B.1}=480$ and $C_{B.8}=10$, we get $C_{\mathrm{sep}}=39\cdot35\cdot480^2\cdot10\approx3.1\cdot10^9$. Then ten iterations of $C\mapsto4C+2C_{\mathrm{sep}}(1+C)$ starting from $4$ give a constant of at least $4.8\cdot10^{98}$. Here $4$ is less than their starting constant $4+\frac16$, which is the constant of the small-graph bound of Norin and Steiner~\cite[Theorem~4]{NS} (Section~\ref{sec:NSres}) with $\eps=\frac16$. With $C_{B.1}={\CDP}$ in place of $480$, the same computation gives a constant of at least $6.4\cdot10^{174}$. So the constant of Norin and Steiner is large for three reasons: $C_{\mathrm{sep}}$ is large, the constant is multiplied by more than $2C_{\mathrm{sep}}$ at each of the ten iterations of Part~(II), and Part~(III) multiplies the result by a further constant. 

\subsection{The present paper}\label{sec:present}

We now summarise how the present paper improves the proof of Norin and Steiner~\cite{NS}. Figure~\ref{fig:structure} shows how the main results depend on each other.

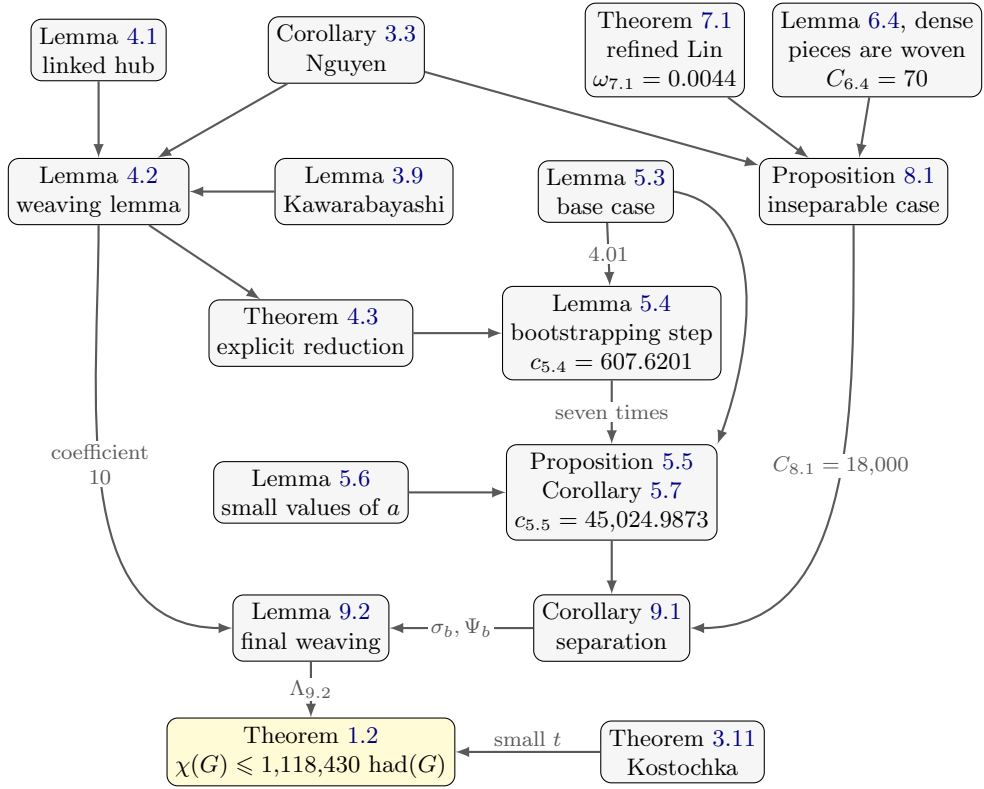
\begin{figure}[ht]
\centering
\begin{tikzpicture}[scale=0.78,
  box/.style={draw,rounded corners,fill=gray!7,font=\footnotesize,align=center,inner sep=3pt},
  fin/.style={box,fill=yellow!20},
  arr/.style={-latex,thick,gray!70!black},
  cst/.style={font=\scriptsize,fill=white,inner sep=1pt}]
% inputs
\node[box] (hub) at (-1,0) {Lemma~\ref{lem:hub}\\ linked hub};
\node[box] (ngu) at (3.2,0) {Corollary~\ref{cor:nguyen}\\ Nguyen};
\node[box] (lin) at (8.6,0) {Theorem~\ref{thm:refinedlin}\\ refined Lin\\ $\cL=0.0044$};
\node[box] (wov) at (12.2,0) {Lemma~\ref{lem:Ck}, dense\\ pieces are woven\\ ${\Ck}=70$};
\node[box] (kaw) at (3.5,-2.4) {Lemma~\ref{lem:kawa}\\ Kawarabayashi};
% second row
\node[box] (weave) at (-1,-2.4) {Lemma~\ref{lem:weave}\\ weaving lemma};
\node[box] (base) at (7.6,-2.4) {Lemma~\ref{lem:base}\\ base case};
\node[box] (dp) at (11.8,-2.4) {Proposition~\ref{prop:DP}\\ inseparable case};
% third row
\node[box] (sep) at (2.6,-4.8) {Theorem~\ref{thm:sepexplicit}\\ explicit reduction};
\node[box] (lev) at (7.7,-4.8) {Lemma~\ref{lem:level}\\ bootstrapping step\\ ${\Sst}=607.6201$};
% fourth row
\node[box] (bs) at (7.7,-7.5) {Proposition~\ref{prop:explicitbootstrap}\\ Corollary~\ref{cor:allscales}\\ ${\cst}=45{,}024.9873$};
\node[box] (ind) at (2.6,-7.5) {Lemma~\ref{lem:indep}\\ small values of $a$};
% fifth row
\node[box] (wf) at (2.6,-9.8) {Lemma~\ref{lem:weavefinal}\\ final weaving};
\node[box] (sepP) at (7.7,-9.8) {Corollary~\ref{cor:sepP}\\ separation};
% final
\node[fin] (thm) at (2.6,-11.9) {Theorem~\ref{thm:all}\\ $\chi(G)\leqslant1{,}118{,}430\,\had(G)$};
\node[box] (kos) at (8.9,-11.9) {Theorem~\ref{thm:kos}\\ Kostochka};
% arrows
\draw[arr] (hub) -- (weave);
\draw[arr] (ngu) -- (weave);
\draw[arr] (kaw) -- (weave);
\draw[arr] (ind) -- (bs);
\draw[arr] (ngu) -- (dp);
\draw[arr] (lin) -- (dp);
\draw[arr] (wov) -- (dp);
\draw[arr] (weave) -- (sep);
\draw[arr] (sep) -- (lev);
\draw[arr] (base) -- node[cst] {$4.01$} (lev);
\draw[arr] (base.east) to[out=-10,in=70] (bs.north east);
\draw[arr] (lev) -- node[cst] {seven times} (bs);
\draw[arr] (weave.south) to[out=-90,in=180] node[cst,pos=0.45,align=center] {coefficient\\ $10$} (wf.west);
\draw[arr] (dp.south) to[out=-90,in=0] node[cst,pos=0.45] {$\Cdens=18{,}000$} (sepP.east);
\draw[arr] (bs) -- (sepP);
\draw[arr] (sepP) -- node[cst] {$\sigma_b,\Psi_b$} (wf);
\draw[arr] (wf) -- node[cst] {${\LamF}$} (thm);
\draw[arr] (kos) -- node[cst,above=1pt,fill=none] {small $t$} (thm);
\end{tikzpicture}
\caption{The structure of the proof of Theorem~\ref{thm:all}. An arrow means that the second result uses the first, or uses a constant from it. Most labels are the main constants passed along the arrows (see Table~\ref{tab:constants}). Here $4.01$ is the constant of Lemma~\ref{lem:base}, and $10$ is the coefficient of $\Phi$ in ${\LamU}$ and of $78+2{\cst}$ in ${\LamF}$.}\label{fig:structure}
\end{figure}

\paragraph{1. An explicit chromatic-separable reduction (Section~\ref{sec:sep}).} We reprove, with explicit constants, the chromatic-separable reduction of Norin and Steiner~\cite[Theorem~15]{NS}, which is used in each step of Part~(II). We use a theorem of Nguyen~\cite{Ngu24} on highly connected subgraphs with large chromatic number (Theorem~\ref{thm:nguyen}) in place of the theorem of Gir\~ao and Narayanan~\cite{GN22} used by Norin and Steiner. We also use the linkage theorem of Thomas and Wollan~\cite{TW05}. The reduction must join some prescribed vertices, called \defn{roots}, to the recursively built minors. It must also join other prescribed vertices, called \defn{terminals}, to each other in pairs. The reduction does this through a \defn{hub}, which is an induced subgraph with two properties. First, for a prescribed integer $m$, the hub is $m$-linked, so that any $m$ pairs formed by $2m$ distinct vertices of the hub can be joined by disjoint paths inside the hub. Second, the hub has chromatic number at most $\lceil25m/2\rceil+1$, so that deleting the hub costs few colours. We obtain the hub (Lemma~\ref{lem:hub}) from a smallest induced subgraph whose chromatic number is at least a prescribed value. This is a vertex-critical induced subgraph, from which a sharpened form of Mader's theorem (Lemma~\ref{lem:mader}) extracts a dense highly connected subgraph. The density form of the Thomas--Wollan theorem then makes the hub linked (see the remark after Lemma~\ref{lem:hub} for the saving over the connectivity form). Obtaining the hub in this way avoids the small-dense-subgraph theorem of Delcourt and Postle (Section~\ref{sec:extremal}). After the roots and terminals are joined through the hub, the proof of Lemma~\ref{lem:weave} reduces to showing that a $21a$-connected subgraph $G_3$ of large chromatic number is $(a,0)$-woven (see Section~\ref{sec:defs}). If $G_3$ has a $K_{2a}$ minor, then $G_3$ is $(a,0)$-woven by the rooting lemma of Kawarabayashi (Lemma~\ref{lem:kawa}). This is Step~7a of the proof of Lemma~\ref{lem:weave}, and the idea is due to Delcourt and Postle~\cite[Section~7.1]{DP}. So in Step~7b we may assume that $G_3$ has no $K_{2a}$ minor. As a result, the chromatic separations in Step~7b are made in graphs with no $K_{2a}$ minor, rather than no $K_{14a}$ minor as in~\cite{DP,NS}. We also allow an uneven distribution of the branch sets among the three subgraphs used in the recursion. The result is Theorem~\ref{thm:sepexplicit}, with the cost $\cS(\Phi)$ given in Table~\ref{tab:constants}, where $\Phi$ is the parameter in the hypothesis $(\mathcal H_{\Phi,a_0})$ of Section~\ref{sec:prelim}. This cost is $42$ plus the minimum of two terms. The second comes from the uneven distribution, and is the smaller one when $\Phi>190$.

\paragraph{2. An additive bootstrap (Section~\ref{sec:bootstrap}).} Each bootstrapping step of Norin and Steiner splits the graph into two parts, $G_1$ and $G_2$. The vertex set of $G_1$ is partitioned into the vertex sets of large connected induced bipartite subgraphs. The part $G_2$ has no such subgraph. The chromatic-separable reduction of Norin and Steiner~\cite[Theorem~15]{NS}, with constant $C_{\mathrm{sep}}$, is applied to $G_2$. The chromatic-inseparable subgraphs of $G_2$ are handled by a Gy\'arf\'as path argument, which requires a colouring bound for small graphs, called \emph{Gy\'arf\'as graphs} (see Lemma~\ref{lem:NS17}). Norin and Steiner colour the Gy\'arf\'as graphs with the bound obtained at the previous bootstrapping step. We colour them with a fixed earlier bound instead. At every step except the final one, this is the base case (Lemma~\ref{lem:base}, an explicit form of Part~(I)). At the final step, it is the bound given by an extra bootstrapping step, called the \defn{side step}, which is used only for this purpose. As explained after Lemma~\ref{lem:level}, colouring the Gy\'arf\'as graphs with a fixed earlier bound makes each step map a constant $c$ to $2c$ plus a constant independent of $c$. Five main steps, the side step and a final step suffice (Proposition~\ref{prop:explicitbootstrap}). This gives ${\cst}=45{,}024.9873$, compared with about $77{,}681.0$ for seven steps that all use the base case (see the remark after Proposition~\ref{prop:explicitbootstrap}).

\paragraph{3. The chromatic-inseparable case (Section~\ref{sec:insep}).} Let $H$ be a graph with no $K_t$ minor. Let $\Gamma$ be an upper bound on the chromatic number of the small subgraphs of $H$. For the lemma of Delcourt and Postle, these are the subgraphs with $O(t\log^4t)$ vertices. For Proposition~\ref{prop:DP}, they are the subgraphs with at most $N_t=O(t\log^2t)$ vertices. By the lemma of Delcourt and Postle in Section~\ref{sec:DPoverview}, if $H$ is $C_{\mathrm{ins}}(t+\Gamma)$-chromatic-inseparable, then $\chi(H)<2C_{\mathrm{ins}}(t+\Gamma)$. So the large constant $C_{\mathrm{ins}}$ multiplies $\Gamma$, both in the separation cost $C_{\mathrm{ins}}(t+\Gamma)$ and in the chromatic threshold $2C_{\mathrm{ins}}(t+\Gamma)$ (see Section~\ref{sec:prelim}). We rework the proof of Delcourt and Postle as follows, to keep the large constant $C_{\mathrm{ins}}$ out of the separation cost. Proposition~\ref{prop:DP} shows that, if $\log t\geqslant1.6\cdot10^6$, then every graph with no $K_t$ minor and chromatic number at least $(32+2\Cdens)t+2\Gamma+2$ is $(78t+2\Gamma)$-chromatic-separable. Here $\Cdens=18{,}000$ is the constant in the density threshold $d=\Cdens t$ of the chromatic-inseparable case. Thus $\Gamma$ has coefficient only $2$, and the large constant $\Cdens$ enters only the chromatic threshold, not the separation cost. Section~\ref{sec:proof} explains how this affects the final constant.

\paragraph{4. Rooted minors via massed pairs (Section~\ref{sec:wollan}).} In the chromatic-inseparable case, each piece must be woven, that is, for every choice of roots and terminal pairs it must have a rooted complete minor and a linkage that meet only in the roots that are also terminals (Section~\ref{sec:prelim}). Delcourt and Postle obtain this from Wollan's extremal theorem for rooted minors~\cite[Theorem~1.1]{Wol08}, applied to the graph $K_s+\ell K_2$, where $s$ is the number of roots and $\ell\approx10t$ is the number of terminal pairs. Their application of Wollan's theorem requires density about $5\cdot10^5\,t$. We use instead three structural theorems from Wollan's proof~\cite[Theorems~2.1--2.3]{Wol08}, concerning a minimal counterexample (Theorem~\ref{thm:wollanmin}), together with the strategy of Wollan's proof of~\cite[Theorem~5.1]{Wol08} (see the outline after Lemma~\ref{lem:massed}). Lemma~\ref{lem:Ck} then requires only density $70t$. Delcourt and Postle delete the roots and apply Wollan's theorem to the rest of the piece. We instead apply Lemma~\ref{lem:Ck} to the whole piece, roots included. This is possible for two reasons. First, in the proof of Lemma~\ref{lem:W}, each root that is not a terminal is a root of the rooted model, and each root that is a terminal is replaced by a neighbour, as in~\cite[Lemma~5.13]{DP}. Second, keeping the roots costs only a few edges per root in the density condition~(M1) of Section~\ref{sec:wollan}. Lemma~\ref{lem:Ck} needs pieces that are $9t$-connected and have density ${\Ck}t$, and these two requirements give the condition $\cL\Cdens\geqslant{\Ck}+9$ on $\Cdens$ in Proposition~\ref{prop:DP} (see Step~1 of the proof of Proposition~\ref{prop:DP}).

\paragraph{5. A refined form of Lin's theorem (Section~\ref{sec:lin}).} The chromatic-inseparable case uses a theorem that finds small highly connected subgraphs in dense graphs with no $K_t$ minor. The theorem of Delcourt and Postle in Section~\ref{sec:DPoverview} has density constant ${\CDP}\approx3.07\cdot10^6$ (Section~\ref{sec:extremal}). Lin's improvement of it has density constant ${\CLin}\approx1.31\cdot10^9$ (Section~\ref{sec:lin}). We rerun Lin's argument with different parameters, and obtain pieces that are dense, not only highly connected (Theorem~\ref{thm:refinedlin}). For $t\geqslant10^{13}$, every graph with no $K_t$ minor and density $d\geqslant1.8\cdot10^4\,t$ contains a $(9t+1)$-connected piece with at most ${\PL}\,t\log t$ vertices and density more than $\cL d-9t$, where $\cL=0.0044$ and ${\PL}=4.56\cdot10^{18}$. Section~\ref{sec:lin} explains the sources of the improvement. Because the pieces have $O(t\log t)$ vertices, the core built in the proof of Proposition~\ref{prop:DP} has $O(t\log^2t)$ vertices, and so do the small graphs that must be coloured. The exponent $2$ in this order bound determines the number of bootstrapping steps (see the remark after Proposition~\ref{prop:explicitbootstrap}).

\paragraph{6. A more economical construction (Sections~\ref{sec:wollan} and~\ref{sec:insep}).} In the chromatic-inseparable case, each piece must be $(2x,\ell)$-woven, where $x$ is the number of branch sets added at each stage and $\ell$ bounds the number of linkage pairs. Delcourt and Postle take $\ell=10t$ and $x=\lceil t/\sqrt{\log t}\rceil$. We take $x=\lceil t/(M\sqrt{\log t})\rceil$ with $M=10^5$. The reasons for this choice, and its cost, are explained before Lemma~\ref{lem:massed}. At stage $n$ of the construction the number of pairs is $(4n-3)x$. Since there are at most $M\sqrt{\log t}+1$ stages, $(4n-3)x\leqslant\lfloor4.004t\rfloor$ when $\log t\geqslant1.6\cdot10^6$. So we take $\ell=\lfloor4.004t\rfloor$. Finally, as in the proof of Delcourt and Postle, paths from the pieces are joined inside one designated piece, which we call the \emph{junction}, by knitting (Lemma~\ref{lem:knit}). Lemma~\ref{lem:knit} uses the density form of the Thomas--Wollan linkage theorem. So the connectivity required by the knitting step is at most $7.001t$ (Step~4 of the proof of Proposition~\ref{prop:DP}), instead of about $31t$ with the connectivity form of the Thomas--Wollan theorem.

\paragraph{7. Explicit size thresholds (Sections~\ref{sec:bootstrap} and~\ref{sec:proof}).} Our asymptotic statements hold above explicit size thresholds. The thresholds $T_{\mathrm{base}}$ and $T_{\mathrm{boot}}$ are lower bounds on $\log\log a$ in Section~\ref{sec:bootstrap}, and $T_{\mathrm{fin}}$ is a lower bound on $\log\log t$. The other thresholds are lower bounds on $\log t$ or on $t$, and are implied by $\log\log t\geqslant T_{\mathrm{fin}}$. Section~\ref{sec:small} and Case~1 of the proof of Theorem~\ref{thm:all} treat the values below $T_{\mathrm{boot}}$ and $T_{\mathrm{fin}}$, respectively. For $t$ below a size threshold $t_0$, the estimate of Norin and Steiner~\cite[Section~4]{NS} uses the Kostochka--Thomason degeneracy bound, which presumably contributes a factor of order $\sqrt{\log t_0}$ to the constant of Norin and Steiner. In the proof of Theorem~\ref{thm:all}, where $G$ has no $K_t$ minor, we colour with the bound of Theorem~\ref{thm:kos} only when $\log\log t<18.4$. Then this bound gives $\chi(G)<600{,}000\,t+1$.

\medskip
Most of the ingredients above re-optimise arguments of Delcourt and Postle~\cite{DP}, Norin and Steiner~\cite{NS}, Lin~\cite{Lin} and Wollan~\cite{Wol08}. The new ingredients are as follows:
\begin{itemize}[nosep]
\item The hub is built in a new way (Lemma~\ref{lem:hub} and item~1 above).
\item Step~7a of the proof of Lemma~\ref{lem:weave} uses an observation of Delcourt and Postle inside the recursion, so that Step~7b separates graphs with no $K_{2a}$ minor (item~1 above).
\item The branch sets in the proof of Lemma~\ref{lem:weave} are distributed unevenly among the three subgraphs (item~1 above).
\item The bootstrap is additive, with a new schedule (Section~\ref{sec:bootstrap}).
\item Lemma~\ref{lem:massed} adapts Wollan's massed-pair argument to rooted models of $K_s+\ell K_2$ with $s$ small, in which only the clique part needs a minor argument.
\item In Step~5 of the proof of Proposition~\ref{prop:DP}, the first stage of the chromatic-inseparable case is treated separately. At later stages, the construction, which follows Delcourt and Postle, makes the new branch sets pairwise adjacent through a rooted clique model in a piece. This piece exists only from the second stage onwards.
\end{itemize}

%=====================================================================
\subsection{Constants}\label{sec:constants}
%=====================================================================

The constant in Theorem~\ref{thm:all} is
\[
{\cfin}:=1829.6+2\Cdens+24\,{\cst}<1{,}118{,}430 .
\]
Here ${\cst}=45{,}024.9873$ is a constant such that every graph with no $K_a$ minor and at most $N_a=24{\PL}M^2\,a\log^2a$ vertices is ${\cst}a$-colourable (Proposition~\ref{prop:explicitbootstrap} and Corollary~\ref{cor:allscales}), and $\Cdens=18{,}000$ is the density constant of item~3 above (Proposition~\ref{prop:DP}). The term $1829.6$ collects the lower-order terms (see the proof of Theorem~\ref{thm:all}). The term $24{\cst}$ accounts for about $97\%$ of ${\cfin}$.

The terms of ${\cfin}$ come from three changes to the argument of Norin and Steiner, which remove the three reasons for their large constant listed after Part~(III) in Section~\ref{sec:NSoverview}. First, the reduction to chromatic-inseparable subgraphs costs $\cS(\Phi)$ in place of $C_{\mathrm{sep}}(1+\Phi)$, up to lower-order terms (item~1 above). Second, the Gy\'arf\'as graphs are coloured with a fixed earlier bound, and the additive bootstrap of item~2 above gives ${\cst}$. Third, Part~(III) is replaced by a final separation step (Corollary~\ref{cor:sepP} and Lemma~\ref{lem:weavefinal}). This step is one more round of the chromatic-separable reduction. Each of its chromatic separations is supplied by the chromatic-inseparable case (Proposition~\ref{prop:DP} and items~3--6 above), with ${\cst}$ as the bound for the small graphs. In a graph with no $K_b$ minor, such a separation costs $(78+2{\cst})b$ (Corollary~\ref{cor:sepP}). In Lemma~\ref{lem:weavefinal}, which builds a $K_a$ minor, these separations are made in graphs with no $K_{2a}$ minor. So $78+2{\cst}$ plays the role of $\Phi$ in ${\LamU}$. The coefficient $10$ of $\Phi$ in ${\LamU}$ therefore gives the term $20{\cst}$ of ${\cfin}$. Two successive separations account for a further $4{\cst}$. Of this $4{\cst}$, the chromatic threshold contributes $2{\cst}$. The cost of the first separation contributes the other $2{\cst}$, and this extra $2{\cst}$ ensures that each part of the first separation can be separated again. The chromatic threshold also contributes $2\Cdens$ (see Section~\ref{sec:proof}).
These and the other main constants are listed in Table~\ref{tab:constants}.

\begin{table}[!t]
\centering\footnotesize
\begin{tabular}{@{}>{\raggedright}p{1.8cm}>{\raggedright}p{3.9cm}>{\raggedright}p{6.0cm}>{\raggedright\arraybackslash}p{2.0cm}@{}}\toprule
Constant & Value & Role & Where\\\midrule
$\Phi$ & $4.01$ or $615.6401$ ($78+2{\cst}$ plays the role of $\Phi$ in ${\LamF}$) & in $(\mathcal H_{\Phi,a_0})$, $s$-chromatic-inseparable and $\had<b$ give $s\leqslant\Phi b$ & Sec.~\ref{sec:prelim}\\
$\LamE$, $\LamU$ & $(1+\eta)(517.5+12\Phi)$, $(1+\eta)(897.5+10\Phi)$ & coefficient of $a$ in the weaving lemma, in the even and uneven schemes & Lem.~\ref{lem:weave}\\
$\eta$ & $5/(a_0-4)$, $a_0=t/\log t$ & slack in $\LamE$, $\LamU$ and ${\LamF}$ & Lem.~\ref{lem:weave}, \ref{lem:level}, \ref{lem:weavefinal}\\
$\cS(\Phi)$ & $42+\min\{517.5+12\Phi,$ $897.5+10\Phi\}$ & term added by a bootstrapping step whose Gy\'arf\'as graphs are $\Phi a$-colourable & Thm.~\ref{thm:sepexplicit}\\
$\xi$ & $10^{-4}$ & error term of a bootstrapping step (which requires $0\leqslant\Phi\leqslant10^4$) & Lem.~\ref{lem:level}\\
${\Sst}$ & $607.6201$ & each main bootstrapping step maps $c$ to $2c+{\Sst}$ & Lem.~\ref{lem:level}\\
${\cst}$ & $45{,}024.9873$ & colouring graphs with at most $N_a$ vertices & Prop.~\ref{prop:explicitbootstrap}\\
$N_a$ & $24{\PL}M^2\,a\log^2a$ & order bound for the small graphs & Prop.~\ref{prop:explicitbootstrap}\\
$T_{\mathrm{base}}$, $T_{\mathrm{boot}}$ & $1160$, $1161$ & size thresholds on $\log\log a$ in Section~\ref{sec:bootstrap} & Sec.~\ref{sec:base}, Prop.~\ref{prop:explicitbootstrap}\\
$M$ & $10^5$ & each stage adds $x=\lceil t/(M\sqrt{\log t})\rceil$ branch sets & Lem.~\ref{lem:Ck}\\
${\Ck}$ & $70$ & density that makes a $9t$-connected graph $(2x,\lfloor4.004t\rfloor)$-woven & Lem.~\ref{lem:Ck}\\
$\cL$ & $0.0044$ & proportion of density kept by the refined Lin theorem & Thm.~\ref{thm:refinedlin}\\
${\PL}$ & $4.56\cdot10^{18}$ & pieces have at most ${\PL}\,t\log t$ vertices & Thm.~\ref{thm:refinedlin}\\
$\Cdens$ & $18{,}000$ & density threshold $d=\Cdens t$, with $\cL\Cdens\geqslant {\Ck}+9$ & Prop.~\ref{prop:DP}\\
$k$, $d$ & $12t$, $\Cdens t$ & connectivity and density threshold in the chromatic-inseparable case & Prop.~\ref{prop:DP}\\
$\sigma_b$, $\Psi_b$ & $(78+2{\cst})b$, $(32+2\Cdens+2{\cst})b+2$ & separation cost and chromatic threshold for graphs with no $K_b$ minor & Cor.~\ref{cor:sepP}\\
$T_{\mathrm{fin}}$ & $18.4$ & size threshold on $\log\log t$ in Section~\ref{sec:proof} & Lem.~\ref{lem:weavefinal}\\
${\LamF}$ & $(1+\eta)(897.5+10(78+2{\cst}))\approx9.022\cdot10^5$ & coefficient of $a$ in the final weaving lemma & Lem.~\ref{lem:weavefinal}\\
${\KF}$ & $110+2\Cdens+4{\cst}+\frac2t\approx2.162\cdot10^5$ & coefficient of $t$ in the final weaving lemma, which pays for two successive separations & Lem.~\ref{lem:weavefinal}\\
${\cfin}$ & $<1{,}118{,}430$ & $1829.6+2\Cdens+24{\cst}$ & Sec.~\ref{sec:constants}\\\bottomrule
\end{tabular}
\caption{Constants, parameters and thresholds used in more than one place. The parameters of the bootstrap schedule are listed in Table~\ref{tab:schedule}.}\label{tab:constants}
\end{table}

%=====================================================================
\section{Preliminaries}\label{sec:prelim}
%=====================================================================

\subsection{Definitions}\label{sec:defs}

All graphs are finite and simple. For an integer $n\geqslant0$, let \defn{$[n]$}${}:=\{1,\dots,n\}$, so $[0]=\emptyset$. All logarithms are natural. For $k\geqslant1$, we write \defn{$\log^kt$}${}:=(\log t)^k$. For a graph $G$, let \defn{$v(G)$} and \defn{$e(G)$} denote the numbers of vertices and edges of $G$. The \defn{order} of $G$ is $v(G)$. The \defn{density} of a non-null graph $G$ is $e(G)/v(G)$, which is half its average degree. For $v\in V(G)$, let \defn{$N_G(v)$} be the set of neighbours of $v$ in $G$, and let \defn{$\deg_G(v)$}${}:=|N_G(v)|$. Let \defn{$\delta(G)$} and \defn{$\Delta(G)$} denote the minimum and maximum degree. For an integer $k\geqslant0$, a graph is \defn{$k$-connected} if it has more than $k$ vertices and remains connected after deleting any set of fewer than $k$ vertices. For a non-null graph $G$, the \defn{connectivity} \defn{$\kappa(G)$} is the largest $k$ such that $G$ is $k$-connected. So $G$ is $k$-connected if and only if $\kappa(G)\geqslant k$, and $\kappa(G)\leqslant\delta(G)$. A subgraph $H$ of $G$ is \defn{spanning} if $V(H)=V(G)$. For $X\subseteq V(G)$, let \defn{$G[X]$} denote the subgraph induced by $X$. Let \defn{$G-X$}${}:=G[V(G)\setminus X]$. For disjoint $A,B\subseteq V(G)$, let \defn{$e_G(A,B)$} be the number of edges of $G$ with one end in $A$ and the other in $B$, and let \defn{$G(A,B)$} be the bipartite subgraph of $G$ with vertex set $A\cup B$ whose edges are these edges. For graphs $G_1,G_2$, the graph \defn{$G_1\cup G_2$} has vertex set $V(G_1)\cup V(G_2)$ and edge set $E(G_1)\cup E(G_2)$. The \defn{disjoint union} \defn{$G_1+G_2$} is the union of vertex-disjoint copies of $G_1$ and $G_2$. For an integer $\ell\geqslant0$, let \defn{$\ell G$} denote the disjoint union of $\ell$ copies of $G$. The \defn{null graph} has no vertices.

\emph{Contraction.} Let $uv$ be an edge of $G$. \defn{Contracting} $uv$ means deleting $u$ and $v$ and adding a new vertex adjacent to every vertex of $(N_G(u)\cup N_G(v))\setminus\{u,v\}$. The resulting graph is simple. Now let $A\subseteq V(G)$ be non-empty with $G[A]$ connected, and let $T$ be a spanning tree of $G[A]$. The graph \defn{$G/A$} is obtained from $G$ by contracting the edges of $T$ one after another. The vertices of $A$ are replaced by a single vertex $x_A$, which is adjacent to exactly those vertices of $V(G)\setminus A$ that have a neighbour in $A$. In particular, $G/A$ does not depend on the choice of $T$ or on the order of the contractions. If $|A|=1$, then $G/A=G$ and $x_A$ is the vertex of $A$. We say that $G/A$ is obtained by \defn{contracting} $A$. Contracting a connected subgraph $H$ means contracting $V(H)$. Finally, let $A_1,\dots,A_m\subseteq V(G)$ be pairwise disjoint and non-empty, with each $G[A_i]$ connected. Contracting $A_1,\dots,A_m$ means contracting $A_1$, then $A_2$, and so on, up to $A_m$. This is well defined, since each $A_i$ still induces a connected subgraph after $A_1,\dots,A_{i-1}$ are contracted. The result is obtained from $G-(A_1\cup\dots\cup A_m)$ by adding vertices $x_{A_1},\dots,x_{A_m}$, where $x_{A_i}$ is adjacent to every remaining vertex with a neighbour in $A_i$, and $x_{A_i}$ is adjacent to $x_{A_j}$ if and only if $G$ has an edge between $A_i$ and $A_j$. So the result does not depend on the order.

\emph{Colourings.} A colouring of $G$ is a map $c:V(G)\to\mathbb N$ with $c(u)\ne c(v)$ for every edge $uv$, and $\chi(G)$ is the minimum number of colours used by a colouring. An \defn{independent set} is a set of pairwise non-adjacent vertices. The \defn{independence number} \defn{$\alpha(G)$} is the maximum size of an independent set. The \defn{fractional chromatic number} \defn{$\chi_f(G)$} is the minimum of $\sum_Ix_I$ over all non-negative weights $x_I$ on the independent sets $I$ of $G$ such that $\sum_{I\ni v}x_I\geqslant1$ for every vertex $v$.

\emph{Minors and models.} We set $\had(G):=0$ if $G$ is null. A \defn{model} of a graph $H$ in $G$ is a family $\{B_v:v\in V(H)\}$ of pairwise disjoint non-empty subsets of $V(G)$ such that each $G[B_v]$ is connected and there is an edge of $G$ between $B_u$ and $B_v$ for every edge $uv$ of $H$. The sets $B_v$ are called \defn{branch sets}. A \defn{$K_s$ model} is a model of $K_s$, that is, a family of $s$ pairwise adjacent branch sets. A graph $G$ has an $H$ minor if and only if it contains a model of $H$. Let $X=\{x_v:v\in V(H)\}$ be a set of $|V(H)|$ distinct vertices of $G$. The model is \defn{rooted at $X$} (according to the bijection $x_v\mapsto v$) if $x_v\in B_v$ for every $v\in V(H)$. For a $K_a$ model and a set $R$ of $a$ vertices, \defn{rooted at $R$} means rooted at $R$ according to some bijection between $R$ and $V(K_a)$. Equivalently, each branch set contains exactly one vertex of $R$. If $\mathcal A$ is a model, then \defn{$V(\mathcal A)$} denotes the union of its branch sets. Similarly, for a family $\mathcal P$ of paths, \defn{$V(\mathcal P)$}${}:=\bigcup_{P\in\mathcal P}V(P)$. A subgraph or family is \defn{disjoint from} a set $X$ if its vertex set is disjoint from $X$, and two subgraphs are \defn{disjoint} if they are vertex-disjoint.

\emph{Separations.} A \defn{separation} of $G$ is a pair $(A,B)$ of subsets of $V(G)$ with $A\cup B=V(G)$ and no edge between $A\setminus B$ and $B\setminus A$. Its \defn{order} is $|A\cap B|$. It is \defn{trivial} if $A\subseteq B$ or $B\subseteq A$.

\emph{Paths and linkages.} A \defn{path} $P=v_0v_1\dots v_m$ in $G$ is a sequence of distinct vertices with $v_{i-1}v_i\in E(G)$ for each $i\in[m]$. We identify $P$ with the subgraph it forms. Its \defn{ends} are its \defn{first} vertex $v_0$ and its \defn{last} vertex $v_m$. The other vertices are \defn{internal}. For $X\subseteq V(G)$, a family $\mathcal P$ of paths in $G$ is \defn{pairwise disjoint outside $X$} if $V(P)\cap V(Q)\subseteq X$ for all distinct $P,Q\in\mathcal P$. A path $P$ is \defn{induced} in $G$ if $G[V(P)]=P$. Every subpath of an induced path is induced. If $P$ is an induced path of $G$, then $\chi(G[V(P)])\leqslant2$. Given a path $P$ in $G$ with ends $u$ and $v$, a shortest $u$--$v$ path in $G[V(P)]$ is an induced path of $G$ with the same ends and with vertex set contained in $V(P)$. For $A,B\subseteq V(G)$, an \defn{$A$--$B$ path} is a path whose first vertex is its only vertex in $A$ and whose last vertex is its only vertex in $B$. Let $s_1,\dots,s_\ell,t_1,\dots,t_\ell$ be vertices such that $s_i\ne s_j$ and $t_i\ne t_j$ for $i\ne j$, and such that if $s_i=t_j$ then $i=j$. An \defn{$\{(s_i,t_i)\}_{i\in[\ell]}$-linkage} is a family of pairwise vertex-disjoint paths $P_1,\dots,P_\ell$ such that $P_i$ has ends $s_i$ and $t_i$. A graph is \defn{$\ell$-linked}~\cite{Jun70,LM70} if it has an $\{(s_i,t_i)\}_{i\in[\ell]}$-linkage for every choice of $2\ell$ distinct vertices $s_1,\dots,s_\ell,t_1,\dots,t_\ell$. Unlike Thomas and Wollan~\cite{TW05}, we do not require an $\ell$-linked graph to have at least $2\ell$ vertices. This makes no difference below, since every graph that we show to be $\ell$-linked has more than $2\ell$ vertices. Every linkage used with Theorem~\ref{thm:TW} below has $2\ell$ distinct ends.

\emph{Woven graphs.} The following notion is due to Delcourt and Postle~\cite[Definition~5.12]{DP}, in a form that allows any number of terminal pairs up to $b$. Let $b'\in\{0,\dots,b\}$, and let $R=\{r_1,\dots,r_a\}$, $S=\{s_1,\dots,s_{b'}\}$ and $T=\{t_1,\dots,t_{b'}\}$ be sets of vertices of a graph $G$, where $r_1,\dots,r_a$ are distinct and $s_1,\dots,s_{b'},t_1,\dots,t_{b'}$ satisfy the conditions in the definition of a linkage. The sets $R$, $S$ and $T$ need not be disjoint. The graph $G$ is \defn{$(a,b)$-woven} if, for all such $b'$, $R$, $S$ and $T$, there is a $K_a$ model $\mathcal M$ rooted at $R$ and an $\{(s_i,t_i)\}_{i\in[b']}$-linkage $\mathcal P$ with $V(\mathcal M)\cap V(\mathcal P)=R\cap(S\cup T)$. As illustrated in Figure~\ref{fig:woven}, this definition means that the only intersections between the model and the linkage are unavoidable. Since $b'$ ranges over $\{0,\dots,b\}$, every $(a,b)$-woven graph is $(a,c)$-woven for every integer $c$ with $0\leqslant c\leqslant b$. Lemma~\ref{lem:pad} shows that the number of roots can also be decreased. Taking $b'=0$ and any set $R$ of $a$ vertices shows that every $(a,b)$-woven graph with at least $a$ vertices has a $K_a$ minor.

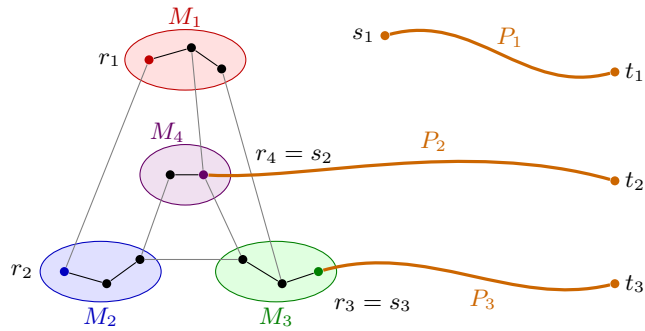
\begin{figure}[ht]
\centering
\begin{tikzpicture}[scale=0.8,
  vtx/.style={circle,fill,inner sep=1.2pt},
  lab/.style={font=\footnotesize},
  lpath/.style={orange!80!black,very thick}]
% branch sets of a K_4 model
\draw[fill=red!12,draw=red!75!black] (3.0,4.5) ellipse (1.0 and 0.5);
\draw[fill=blue!12,draw=blue!75!black] (1.6,1.0) ellipse (1.0 and 0.5);
\draw[fill=green!12,draw=green!50!black] (4.5,1.0) ellipse (1.0 and 0.5);
\draw[fill=violet!12,draw=violet!80!black] (3.0,2.6) ellipse (0.75 and 0.5);
\node[lab,red!75!black] at (3.0,5.2) {$M_1$};
\node[lab,blue!75!black] at (1.6,0.25) {$M_2$};
\node[lab,green!50!black] at (4.5,0.25) {$M_3$};
\node[lab,violet!80!black] at (2.7,3.3) {$M_4$};
\node[vtx,red!75!black] (a1) at (2.4,4.5) {}; \node[vtx] (a2) at (3.1,4.7) {}; \node[vtx] (a3) at (3.6,4.35) {};
\node[vtx,blue!75!black] (b1) at (1.0,1.0) {}; \node[vtx] (b2) at (1.7,0.8) {}; \node[vtx] (b3) at (2.25,1.2) {};
\node[vtx] (c1) at (3.95,1.2) {}; \node[vtx] (c2) at (4.6,0.8) {}; \node[vtx,green!50!black] (c3) at (5.2,1.0) {};
\node[vtx] (d1) at (2.75,2.6) {}; \node[vtx,violet!80!black] (d2) at (3.3,2.6) {};
\draw (a1)--(a2)--(a3); \draw (b1)--(b2)--(b3); \draw (c1)--(c2)--(c3); \draw (d1)--(d2);
\draw[gray] (a1)--(b1); \draw[gray] (a3)--(c2); \draw[gray] (b3)--(c1);
\draw[gray] (d2)--(a2); \draw[gray] (d1)--(b3); \draw[gray] (d2)--(c1);
\node[lab] at (1.75,4.5) {$r_1$}; \node[lab] at (0.3,1.0) {$r_2$};
\node[lab] at (6.1,0.45) {$r_3=s_3$};
% linkage
\node[vtx,orange!80!black] (s1) at (6.3,4.9) {}; \node[lab,left] at (s1) {$s_1$};
\node[vtx,orange!80!black] (t1) at (10.1,4.3) {}; \node[lab,right] at (t1) {$t_1$};
\node[lab] at (4.8,2.9) {$r_4=s_2$};
\node[vtx,orange!80!black] (t2) at (10.1,2.5) {}; \node[lab,right] at (t2) {$t_2$};
\node[vtx,orange!80!black] (t3) at (10.1,0.8) {}; \node[lab,right] at (t3) {$t_3$};
\draw[lpath] (s1) .. controls (7.8,5.3) and (8.6,3.9) .. node[lab,above,pos=0.55] {$P_1$} (t1);
\draw[lpath] (d2) .. controls (5.0,2.5) and (7.6,3.2) .. node[lab,above,pos=0.6] {$P_2$} (t2);
\draw[lpath] (c3) .. controls (7.0,1.5) and (8.4,0.4) .. node[lab,below,pos=0.55] {$P_3$} (t3);
\end{tikzpicture}
\caption{The definition of $(a,b)$-woven with $a=4$ and $b=3$. The shaded branch sets $M_1,\dots,M_4$ form a $K_4$ model rooted at $R=\{r_1,\dots,r_4\}$, and the paths $P_1,P_2,P_3$ form an $\{(s_i,t_i)\}_{i\in[3]}$-linkage. Since $s_2=r_4$ and $s_3=r_3$, the model and the linkage meet exactly in $R\cap(S\cup T)=\{r_3,r_4\}$.}\label{fig:woven}
\end{figure}

\begin{lemma}\label{lem:pad}
Let $a\geqslant a'\geqslant0$ and $b\geqslant0$ be integers. If $G$ is $(a,b)$-woven and $v(G)\geqslant a+2b$, then $G$ is $(a',b)$-woven.
\end{lemma}

\begin{proof}
Let $b'\in\{0,\dots,b\}$, and let $R'=\{r_1,\dots,r_{a'}\}$, $S=\{s_1,\dots,s_{b'}\}$ and $T=\{t_1,\dots,t_{b'}\}$ be as in the definition of $(a',b)$-woven. Since $|R'\cup S\cup T|\leqslant a'+2b'$ and $v(G)\geqslant a+2b$, there are $a-a'$ distinct vertices of $G$ outside $R'\cup S\cup T$. Call these \defn{padding} vertices $r_{a'+1},\dots,r_a$, and let $R:=\{r_1,\dots,r_a\}$. Since $G$ is $(a,b)$-woven, there is a $K_a$ model $\mathcal M=\{M_1,\dots,M_a\}$ with $r_i\in M_i$ and an $\{(s_i,t_i)\}_{i\in[b']}$-linkage $\mathcal P$ with $V(\mathcal M)\cap V(\mathcal P)=R\cap(S\cup T)$. Let $\mathcal M':=\{M_1,\dots,M_{a'}\}$, which is a $K_{a'}$ model rooted at $R'$. The padding vertices lie outside $S\cup T$, so $R\cap(S\cup T)=R'\cap(S\cup T)$, and hence $V(\mathcal M')\cap V(\mathcal P)\subseteq R'\cap(S\cup T)$. Conversely, each $r_i\in R'\cap(S\cup T)$ lies in $M_i\subseteq V(\mathcal M')$ and is an end of a path of $\mathcal P$. Hence $V(\mathcal M')\cap V(\mathcal P)=R'\cap(S\cup T)$.
\end{proof}

\emph{Knit graphs.} The following notion is due to Bollob\'as and Thomason~\cite{BT96}. We use the form of Delcourt and Postle~\cite[Definition~5.7]{DP}, without the requirement $\ell\leqslant k\leqslant v(G)$. A graph $G$ is \defn{$(k,\ell)$-knit} if, for every set $S$ of $k$ vertices and every partition of $S$ into $p\geqslant\ell$ non-empty parts $S_1,\dots,S_p$, there are pairwise disjoint sets $R_1,\dots,R_p$ with $S_i\subseteq R_i$ and each $G[R_i]$ connected.

\emph{Chromatic separability.} The following notion is due to Delcourt and Postle~\cite[Definition~2.4]{DP}. Let $s$ be real. A graph $G$ is \defn{$s$-chromatic-separable} if it has vertex-disjoint subgraphs $H_1,H_2$, possibly null, with $\chi(H_i)\geqslant\chi(G)-s$ for $i=1,2$. Otherwise $G$ is \defn{$s$-chromatic-inseparable}. Delcourt and Postle take $s$ to be an integer, and do not allow null subgraphs. We allow real $s$ and null subgraphs, which changes nothing when $\chi(G)>s$. Taking $H_1=G$ and $H_2$ null shows that every $s$-chromatic-inseparable graph satisfies $\chi(G)>s$. If $G$ is $s$-chromatic-inseparable and $s'\leqslant s$, then $G$ is $s'$-chromatic-inseparable. If $H_1,H_2$ are as above, then $G[V(H_1)]$ and $G[V(H_2)]$ are vertex-disjoint induced subgraphs of $G$ with $\chi(G[V(H_i)])\geqslant\chi(H_i)\geqslant\chi(G)-s$. So $G$ is $s$-chromatic-separable if and only if it has two vertex-disjoint induced subgraphs, possibly null, each with chromatic number at least $\chi(G)-s$.

\emph{Chromatic threshold and separation cost.} Several of our results state that, for a class of graphs and reals $\Psi,\sigma\geqslant0$, every graph $H$ in the class with $\chi(H)\geqslant\Psi$ is $\sigma$-chromatic-separable. Here \defn{$\Psi$} is the \defn{chromatic threshold} and \defn{$\sigma$} is the \defn{separation cost}. The separation cost bounds the loss in chromatic number in passing from $H$ to each of the two vertex-disjoint subgraphs. Such a statement implies that every $s$-chromatic-inseparable graph $H$ in the class satisfies $s<\sigma$ or $\chi(H)<\Psi$. In particular, if $\sigma\leqslant\Psi$, then every such $H$ satisfies $s<\Psi$, since $s<\chi(H)$. It will be convenient to introduce the following hypothesis. Let $\Phi\geqslant0$ and $a_0\geqslant1$ be real. A graph $G$ satisfies \defn{$(\mathcal H_{\Phi,a_0})$} if for all reals $s\geqslant0$ and $b\geqslant a_0$ and for every induced subgraph $H$ of $G$, if $H$ is $s$-chromatic-inseparable and $\had(H)<b$, then $s\leqslant\Phi b$. This property is inherited by induced subgraphs. It is adapted from the hypothesis on the function $f$ in a theorem of Norin and Steiner~\cite[Theorem~15]{NS}. The differences are that the lower bound $a_0$ on $b$ is a free parameter, and that $b$ is not bounded above. The corresponding theorem of Delcourt and Postle~\cite[Theorem~7.1]{DP} uses a different hypothesis, namely an upper bound on $\chi(H)/a$ over small subgraphs $H$ with no $K_a$ minor.

\emph{Tangency and cores.} The following notions are due to Delcourt and Postle~\cite[Definition~6.3]{DP}. A subgraph $J$ of $G$ is \defn{tangent} to a model $\mathcal A=\{A_1,\dots,A_s\}$ if $|V(J)\cap A_i|=1$ for every $i\in[s]$. A set $S\subseteq V(G)$ is a \defn{core} of $\mathcal A$ if, for all distinct $i,j$, some edge of $G$ has one end in $A_i\cap S$ and the other in $A_j\cap S$. Such edges are called \defn{core edges}.

\emph{Bipartite pieces.} Following Norin and Steiner~\cite[Section~3]{NS}, where it is denoted $b(G)$, let \defn{$\operatorname{bip}(G)$} denote the maximum number of vertices in a connected induced bipartite subgraph of $G$. For an integer $k\geqslant1$, we have $\operatorname{bip}(G)<k$ if and only if $G$ has no connected induced bipartite subgraph with exactly $k$ vertices. Indeed, deleting a leaf of a spanning tree of a connected induced bipartite subgraph with more than $k$ vertices leaves a connected induced bipartite subgraph with one vertex fewer.

\emph{The extremal function.} For a positive integer $s$, let \defn{$D(s)$} denote the maximum chromatic number of a graph with no $K_s$ minor. For real $y\geqslant1$, let \defn{$D(y)$}${}:=D(\lfloor y\rfloor)$.

\subsection{Connectivity, linkages and minors}

This subsection collects the results on connectivity, linkages and rooted minors that we use, starting with highly connected subgraphs of graphs with large chromatic number.

\begin{theorem}[Nguyen~{\cite[Theorem~1.1]{Ngu24}}]\label{thm:nguyen}
Let $k\geqslant1$ be an integer. Every graph $F$ with $\chi(F)\geqslant(3+\tfrac1{16})k$ contains a $(k+1)$-connected induced subgraph with more than $\chi(F)-k$ vertices and chromatic number at least $\chi(F)-2k+2$.
\end{theorem}

Nguyen states Theorem~\ref{thm:nguyen} for subgraphs, not necessarily induced. The proof gives an induced subgraph, and the induced form is the one formalised. We use Theorem~\ref{thm:nguyen} in the following form.

\begin{corollary}\label{cor:nguyen}
Let $k\geqslant2$ be an integer and let $F$ be a graph with $\chi(F)\geqslant\lceil3.0625(k-1)\rceil$. Then $F$ contains a $k$-connected induced subgraph $F'$ with $\chi(F')\geqslant\chi(F)-2k$.
\end{corollary}

\begin{proof}
Apply Theorem~\ref{thm:nguyen} with $k-1\geqslant1$ in place of $k$. Since $\chi(F)$ is an integer, $\chi(F)\geqslant\lceil3.0625(k-1)\rceil$ is equivalent to $\chi(F)\geqslant(3+\frac1{16})(k-1)$. This gives a $k$-connected induced subgraph $F'$ of $F$ with $\chi(F')\geqslant\chi(F)-2(k-1)+2=\chi(F)-2k+4\geqslant\chi(F)-2k$.
\end{proof}

We use Menger's theorem in the following form, which is the form formalised in \texttt{LinearHadwiger}. For $A,B\subseteq V(G)$, a \defn{path from $A$ to $B$} is a path whose first vertex lies in $A$ and whose last vertex lies in $B$. It may have further vertices in $A$ or in $B$.

\begin{theorem}[Menger]\label{thm:menger}
Let $G$ be a graph, let $A,B\subseteq V(G)$, and let $k\geqslant0$ be an integer. Suppose that every set of vertices of $G$ that meets every path from $A$ to $B$ in $G$ has at least $k$ vertices. Then $G$ contains $k$ pairwise vertex-disjoint paths from $A$ to $B$. In particular, this holds if $G$ is $k$-connected and $|A|,|B|\geqslant k$.
\end{theorem}

The last sentence holds since, if $Q\subseteq V(G)$ and $|Q|<k$, then $G-Q$ is connected and contains a vertex $a\in A$ and a vertex $b\in B$, so some path from $A$ to $B$ avoids $Q$. Let $P$ be a path from $A$ to $B$. Let $a$ be the last vertex of $P$ in $A$, and let $b$ be the first vertex in $B$ of the subpath of $P$ from $a$ to the last vertex of $P$. The subpath of $P$ from $a$ to $b$ is an $A$--$B$ path, and we say that it is obtained from $P$ by \defn{trimming}. Trimming each path of a family of pairwise vertex-disjoint paths from $A$ to $B$ gives a family of pairwise vertex-disjoint $A$--$B$ paths.

\begin{theorem}[Thomas and Wollan~{\cite[Corollary~1.2]{TW05}}]\label{thm:TW}
Let $\ell\geqslant1$ be an integer. Every $2\ell$-connected graph $G$ with at least $5\ell\,v(G)$ edges is $\ell$-linked.
\end{theorem}

We call Theorem~\ref{thm:TW} the \defn{density form} of the Thomas--Wollan theorem. Its \defn{connectivity form}~\cite[Corollary~1.3]{TW05} states that every $10\ell$-connected graph $G$ is $\ell$-linked, which follows immediately from Theorem~\ref{thm:TW}.

\begin{lemma}[cf.\ Delcourt and Postle~{\cite[Theorem~5.8]{DP}}]\label{lem:knit}
Let $\ell$ and $k$ be integers with $1\leqslant\ell\leqslant k$. Every graph $G$ with $\kappa(G)\geqslant2k-\ell$ and $e(G)\geqslant5k\,v(G)$ is $(k,\ell)$-knit.
\end{lemma}

\begin{figure}[ht]
\centering
\begin{tikzpicture}[scale=0.9,
  vtx/.style={circle,fill,inner sep=1.2pt},
  lv/.style={rectangle,draw,fill=white,inner sep=1.6pt},
  lab/.style={font=\footnotesize},
  lpath/.style={orange!80!black,very thick}]
\draw[rounded corners=8pt,fill=red!8,draw=red!60!black] (-0.55,-0.6) rectangle (3.5,1.55);
\draw[rounded corners=8pt,fill=blue!8,draw=blue!60!black] (4.45,-0.6) rectangle (7.0,1.55);
\node[lab,red!60!black] at (1.5,1.8) {$R_1$};
\node[lab,blue!60!black] at (5.7,1.8) {$R_2$};
\node[vtx] (v12) at (0,0) {}; \node[lv] (v11) at (1.5,0) {}; \node[vtx] (v13) at (3,0) {};
\node[vtx] (v22) at (5,0) {}; \node[lv] (v21) at (6.5,0) {};
\node[vtx] (u12) at (1.1,1.0) {}; \node[vtx] (u13) at (1.9,1.0) {}; \node[vtx] (u22) at (6.5,1.0) {};
\draw (v11)--(u12); \draw (v11)--(u13); \draw (v21)--(u22);
\draw[lpath] (v12) .. controls (0.1,0.85) and (0.55,1.05) .. (u12);
\draw[lpath] (v13) .. controls (2.9,0.85) and (2.45,1.05) .. (u13);
\draw[lpath] (v22) .. controls (5.15,0.85) and (5.8,1.05) .. (u22);
\node[lab,orange!80!black] at (0.05,1.2) {$Q_{1,2}$}; \node[lab,orange!80!black] at (2.95,1.2) {$Q_{1,3}$}; \node[lab,orange!80!black] at (5.2,1.2) {$Q_{2,2}$};
\node[lab,below] at (v11) {$v_{1,1}$}; \node[lab,below] at (v12) {$v_{1,2}$}; \node[lab,below] at (v13) {$v_{1,3}$};
\node[lab,below] at (v21) {$v_{2,1}$}; \node[lab,below] at (v22) {$v_{2,2}$};
\node[lab,above] at (u12) {$u_{1,2}$}; \node[lab,above] at (u13) {$u_{1,3}$}; \node[lab,above] at (u22) {$u_{2,2}$};
\end{tikzpicture}
\caption{The proof of Lemma~\ref{lem:knit} with $p=2$, $S_1=\{v_{1,1},v_{1,2},v_{1,3}\}$ and $S_2=\{v_{2,1},v_{2,2}\}$. Each $u_{i,j}$ is a neighbour of $v_{i,1}$ outside $S$, and the path $Q_{i,j}$ in $G-L$ joins $v_{i,j}$ to $u_{i,j}$. The vertices of $L=\{v_{1,1},v_{2,1}\}$ (squares) lie on no path, and $R_1$ and $R_2$ are shaded.}\label{fig:knit}
\end{figure}
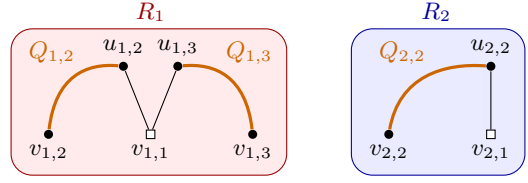

\begin{proof}
Let $|S|=k$, and let $S_1,\dots,S_p$ be a partition of $S$ into non-empty parts, where $p\geqslant\ell$, with $S_i=\{v_{i,1},\dots,v_{i,n_i}\}$. As illustrated in Figure~\ref{fig:knit}, we choose for each $v_{i,j}$ with $j\geqslant2$ a private neighbour $u_{i,j}$ of $v_{i,1}$, and then link $v_{i,j}$ to $u_{i,j}$. Since $\delta(G)\geqslant\kappa(G)\geqslant2k-\ell\geqslant2k-p$, we can choose pairwise distinct vertices $u_{i,j}\notin S$, for $1\leqslant i\leqslant p$ and $2\leqslant j\leqslant n_i$, such that $u_{i,j}$ is adjacent to $v_{i,1}$. Indeed, there are $k-p$ such pairs $(i,j)$, and each $v_{i,1}$ has at least $2k-p$ neighbours, and hence at least $k-p$ neighbours outside $S$. So such a choice exists by Hall's theorem. Let $L:=\{v_{i,1}:i\in[p]\}$. There are $k-p$ pairs $(v_{i,j},u_{i,j})$, for $i\in[p]$ and $2\leqslant j\leqslant n_i$. These pairs have pairwise distinct ends and avoid $L$. The graph $G-L$ is $(\kappa(G)-p)$-connected, and $\kappa(G)-p\geqslant2k-\ell-p\geqslant2(k-p)$. At most $p\,v(G)$ edges of $G$ have an end in $L$, so $e(G-L)\geqslant(5k-p)\,v(G)\geqslant5(k-p)\,v(G-L)$. If $p<k$, then $G-L$ is $(k-p)$-linked by Theorem~\ref{thm:TW}. So $G-L$ has a linkage $\{Q_{i,j}:i\in[p],\ 2\leqslant j\leqslant n_i\}$, where $Q_{i,j}$ is a path from $v_{i,j}$ to $u_{i,j}$. If $p=k$, then there are no pairs, and this linkage is empty. The paths $Q_{i,j}$ avoid $L$. For each $i\in[p]$, let $R_i:=\{v_{i,1}\}\cup\bigcup_{j=2}^{n_i}V(Q_{i,j})$. Then $G[R_i]$ is connected, since $Q_{i,j}$ joins $v_{i,j}$ to $u_{i,j}$ and $u_{i,j}v_{i,1}\in E(G)$ for $2\leqslant j\leqslant n_i$. Also $S_i\subseteq R_i$, and the sets $R_i$ are pairwise disjoint, since the paths $Q_{i,j}$ are pairwise disjoint and avoid $L$.
\end{proof}

Delcourt and Postle~\cite[Theorem~5.8]{DP} state a result of this kind for $Ck$-connected graphs, with an unspecified constant $C$. They state that it is an easy corollary of their Theorem~5.6~\cite{DP}, and remark that the Thomas--Wollan theorem gives the value $10$ for the linkage constant in their Theorem~5.6. Since Lemma~\ref{lem:knit} assumes $e(G)\geqslant5k\,v(G)$, its proof can use Theorem~\ref{thm:TW}. To make $G-L$ $(k-p)$-linked, the density form needs connectivity $2(k-p)$, whereas the connectivity form needs $10(k-p)$.

For a graph $G$ and $W,B\subseteq V(G)$, a \defn{double fan} from $W$ to $B$ in $G$ is a family $\{P_{w,j}:w\in W,\ j\in\{1,2\}\}$ of paths in $G$ that are pairwise disjoint outside $W$, such that each $P_{w,j}$ has first vertex $w$, has its last vertex in $B$, and has no vertex in $W$ other than $w$.

\begin{lemma}[Delcourt and Postle~{\cite[Lemma~5.3]{DP}}]\label{lem:53}
Let $G$ be a graph, and let $A_1,A_2,B\subseteq V(G)$ with $A_1\cap A_2=\emptyset$. Suppose that, for each $i\in\{1,2\}$, the graph $G-A_{3-i}$ contains a double fan from $A_i$ to $B$. Then $G$ contains $|A_1|+|A_2|$ pairwise vertex-disjoint paths from $A_1\cup A_2$ to $B$.
\end{lemma}

Delcourt and Postle assume in addition that $A_1$, $A_2$ and $B$ are disjoint and non-empty, and that the paths in the double fans are $A_i$--$B$ paths. Their argument proves the form above, which is the one in \texttt{LinearHadwiger}: if a set $Q$ meets every path from $A:=A_1\cup A_2$ to $B$, then the two paths of the double fan starting at each $a\in A_i\setminus Q$ meet $Q$ in distinct vertices outside $A$, so $|Q\setminus A|\geqslant2|A_i\setminus Q|$ for $i\in\{1,2\}$. Hence $|Q|\geqslant|A|$, and Theorem~\ref{thm:menger} applies.

\begin{lemma}[Delcourt and Postle~{\cite[Lemma~5.15]{DP}}, see also~{\cite[Lemma~B.11]{NS}}]\label{lem:515}
Let $a,b\geqslant0$ be integers. Let $\mathcal P=\{P_1,\dots,P_b\}$ be an $\{(s_i,t_i)\}_{i\in[b]}$-linkage in a graph $G$, where $P_i$ has ends $s_i$ and $t_i$, and let $S:=\{s_1,\dots,s_b\}$ and $T:=\{t_1,\dots,t_b\}$. Let $H$ be a subgraph of $G$ such that $G[V(H)]$ is $(a,c)$-woven for some integer $c\geqslant b$, and let $R\subseteq V(H)$ with $|R|=a$. Then $G$ has an $\{(s_i,t_i)\}_{i\in[b]}$-linkage $\mathcal P'$ and a $K_a$ model $\mathcal M$ in $G[V(H)]$ rooted at $R$ such that $V(\mathcal P')\subseteq V(H)\cup V(\mathcal P)$ and $R\cap(S\cup T)\subseteq V(\mathcal P')\cap V(\mathcal M)\subseteq R\cap V(\mathcal P)$.
\end{lemma}

If $H$ is $(a,c)$-woven, then so is $G[V(H)]$, since adding edges keeps a graph woven. Delcourt and Postle~\cite{DP} assume that $H$ is $(a,c)$-woven, $a\geqslant1$ and $s_i\ne t_i$, and state that $\mathcal M$ is a model in $G$. Their proof gives the form above, which is the one in \texttt{LinearHadwiger}. It applies the wovenness of $G[V(H)]$ to the pairs formed by the first and last vertices in $V(H)$ of the paths $P_i$ that meet $V(H)$, and replaces the part of each such $P_i$ between these two vertices by the resulting path. These two vertices may coincide, so the condition $s_i\ne t_i$ is not needed. There are at most $b\leqslant c$ such pairs, so no lower bound on $v(H)$ is needed.

\begin{lemma}[Kawarabayashi~{\cite[statement (3.1)]{Kaw07}}, see also~{\cite[Lemma~5.14]{DP}}]\label{lem:kawa}
Let $a\geqslant1$ be an integer and let $G$ be a graph. If $\kappa(G)\geqslant a$ and $G$ has a $K_{2a}$ minor, then $G$ is $(a,0)$-woven.
\end{lemma}

Kawarabayashi states~(3.1) for graphs with no separation of order less than $a$ that separates a vertex from the set of roots. This condition holds when $\kappa(G)\geqslant a$.

\begin{lemma}[Delcourt and Postle~{\cite[Lemma~6.2]{DP}}]\label{lem:62}
Let $G$ be a connected graph and $\emptyset\ne S\subseteq V(G)$. Then there is an induced connected subgraph $H$ of $G$ and a set $S'$ with $S\subseteq S'\subseteq V(H)$, $|S'|\leqslant3|S|$ and $\chi(H-S')\leqslant2$.
\end{lemma}

\subsection{Extremal results}\label{sec:extremal}

This subsection collects extremal results for graphs with no $K_t$ minor, namely edge-density bounds for graphs and for bipartite graphs, small highly connected subgraphs of dense graphs, and a lower bound on the independence number, together with its consequence for colouring.

\begin{theorem}[Kostochka~\cite{Kos82,Kos84} and Thomason~\cite{Tho84}, see also~{\cite[Theorem~4.3]{DP}}]\label{thm:kos}
Let $s\geqslant2$ be an integer. Every non-null graph $G$ with $e(G)\geqslant30\,s\sqrt{\log s}\,v(G)$ has a $K_s$ minor. Consequently, every graph $G$ with no $K_s$ minor satisfies $e(G)\leqslant30\,s\sqrt{\log s}\,v(G)$ and $\chi(G)<60\,s\sqrt{\log s}+1$. In particular, $D(s)<60\,s\sqrt{\log s}+1$.
\end{theorem}

Delcourt and Postle~\cite[Theorem~4.3]{DP} state Theorem~\ref{thm:kos} with the constant $30$, and note that it follows from the results of Kostochka~\cite{Kos82}. The explicit bound of Kostochka~\cite[Theorem~1]{Kos84} holds for all $s$, but with a larger constant (about $270$). The constant $30$ affects the size thresholds, and also the constant in Theorem~\ref{thm:all}. Indeed, in Case~1 of the proof of Theorem~\ref{thm:all}, where $\log\log t<18.4$, Theorem~\ref{thm:kos} gives $\chi(G)<600{,}000\,t+1$, which is below $1{,}118{,}430\,\had(G)$. With the constant $270$ in place of $30$, this bound would be about $5.3\cdot10^6\,t$.

\begin{theorem}[Norin and Postle~{\cite[Theorem~3.2]{NP23}}, see also~{\cite[Theorem~4.1]{DP}}]\label{thm:NP}
Let ${\CNP}:=6400$ and let $t\geqslant3$ be an integer. Let $B$ be a bipartite graph with no $K_t$ minor and with bipartition $(V_1,V_2)$. Then
\[
e(B)\ \leqslant\ {\CNP}\,t\sqrt{\log t}\,\sqrt{|V_1||V_2|}+(t-2)(|V_1|+|V_2|).
\]
\end{theorem}

Norin and Postle show in the proof of~\cite[Theorem~3.2]{NP23} that the constant $6400$ suffices. Theorem~\ref{thm:NP} is formalised in \texttt{LinearHadwiger} with this constant.

Delcourt and Postle~\cite[Theorem~2.3]{DP} proved the following \defn{small-dense-subgraph theorem}. Let ${\CDP}:=480\cdot6400=3{,}072{,}000$. For all integers $t\geqslant3$ and $k\geqslant t$, every non-null graph $G$ with no $K_t$ minor and $e(G)\geqslant{\CDP}k\,v(G)$ has a $k$-connected subgraph with at most ${\CDP}^2\,t\log^3t$ vertices. Delcourt and Postle state this theorem with $480\,{\CNP}$ in place of ${\CDP}$, where ${\CNP}$ is the constant of the bipartite density bound in Theorem~\ref{thm:NP}, which they state as their Theorem~4.1~\cite{DP}. We do not use the small-dense-subgraph theorem, although it is formalised in \texttt{LinearHadwiger}, with a $k$-connected induced subgraph. Instead, Section~\ref{sec:lin} proves a refined form of Lin's theorem, which finds pieces that are both dense and highly connected.

\begin{theorem}[Reed and Seymour~\cite{RS98}, see also~{\cite[Theorem~6 and Corollary~17]{NS}}]\label{thm:RS}
Every graph $J$ satisfies $\chi_f(J)\leqslant2\had(J)$, and hence $v(J)\leqslant2\alpha(J)\had(J)$. In particular, if $a\geqslant2$ is an integer and $J$ has no $K_a$ minor, then $\alpha(J)\geqslant v(J)/(2a-2)$.
\end{theorem}

Duchet and Meyniel~\cite{DM82} proved a slightly stronger form of the last statement of Theorem~\ref{thm:RS} by a direct argument, and Delcourt and Postle~\cite[Theorem~3.3]{DP} use their result. (We use the theorem of Reed and Seymour here, instead of the theorem of Duchet and Meyniel, so that the proof matches the Lean formalisation.)

Results of Delcourt and Postle~\cite[Lemma~3.1 and Corollary~3.2]{DP} imply that every non-null graph $J$ such that $\alpha(J')\geqslant v(J')/p$ for every non-null subgraph $J'$ of $J$ satisfies $\chi(J)\leqslant\bigl(2+\max\{\log(v(J)/p),0\}\bigr)p$. By Theorem~\ref{thm:RS}, this condition holds with $p=2a-2$ for every graph with no $K_a$ minor. We use the result in the following form.

\begin{lemma}[Delcourt and Postle~{\cite[Lemma~3.1 and Corollary~3.2]{DP}}]\label{lem:greedy}
Let $a\geqslant2$ be an integer and let $p:=2a-2$. Every non-null graph $J$ with no $K_a$ minor satisfies $\chi(J)\leqslant\bigl(2+\max\{\log(v(J)/p),0\}\bigr)p$.
\end{lemma}

\subsection{Results of Norin and Steiner}\label{sec:NSres}

Norin and Steiner~\cite[Theorem~4]{NS} proved that, for every $\eps\in(0,\frac12)$, there is $t_0=t_0(\eps)$ such that, for every real $t\geqslant t_0$, every graph $G$ with $\had(G)<t$ and $v(G)\leqslant t(\log t)^{1/2-\eps}$ satisfies $\chi(G)\leqslant(4+\eps)t$. We call this result the \defn{small-graph bound} of Norin and Steiner. We do not use it, and it is not formalised. Lemma~\ref{lem:base} is an explicit instance of it, derived from a quantitative form (Theorem~\ref{thm:NSquant}).

\begin{lemma}[cf.\ Norin and Steiner~{\cite[Corollary~17]{NS}}]\label{lem:NS17}
Let $s$, $b$ and $\Gamma$ be real with $\Gamma\geqslant0$, and let $k\geqslant2$ be an integer. Suppose that every graph $J$ with $\had(J)<b$ and $v(J)\leqslant2bk(k-1)$ satisfies $\chi(J)\leqslant\Gamma$. If $H$ is an $s$-chromatic-inseparable graph with $\had(H)<b$ and $\operatorname{bip}(H)<k$, then $s<\Gamma$.
\end{lemma}

In applications of Lemma~\ref{lem:NS17}, we call the graphs $J$ in its hypothesis \defn{Gy\'arf\'as graphs}. As noted in Section~\ref{sec:defs}, $\operatorname{bip}(H)<k$ means that $H$ has no connected induced bipartite subgraph with exactly $k$ vertices. The proof of Lemma~\ref{lem:NS17} is that of Norin and Steiner~\cite[Corollary~17]{NS}. It uses the following form of the Gy\'arf\'as path argument~\cite{Gya87}, due to Norin and Steiner~\cite[Lemma~13]{NSdraft} (see also~\cite[Lemma~16]{NS}). Let $q\geqslant1$ and $k\geqslant2$ be integers, and let $G$ be a graph with $\operatorname{bip}(G)<k$ and $\chi(G)\geqslant q$ that has no two vertex-disjoint induced subgraphs with chromatic number at least $q$. Then there is a set $X\subseteq V(G)$ with $\alpha(G[X])\leqslant k(k-1)$ and $\chi(G-X)\leqslant q-1$. Norin and Steiner assume only $\operatorname{bip}(G)\leqslant k$. The form above is the one formalised in \texttt{LinearHadwiger}.

\begin{proof}[Proof of Lemma~\ref{lem:NS17}]
If $s<0$, then $s<0\leqslant\Gamma$. So let $s\geqslant0$. If $\chi(H)\leqslant s$, then two null subgraphs show that $H$ is $s$-chromatic-separable. So $\chi(H)>s$. Let $q:=\lceil\chi(H)-s\rceil$. Then $q$ is an integer with $1\leqslant q\leqslant\chi(H)$ and $q<\chi(H)-s+1$. Since $q\geqslant\chi(H)-s$ and $H$ is $s$-chromatic-inseparable, $H$ has no two vertex-disjoint induced subgraphs with chromatic number at least $q$. So the Gy\'arf\'as path argument gives a set $X\subseteq V(H)$ with $\alpha(H[X])\leqslant k(k-1)$ and $\chi(H-X)\leqslant q-1$. Hence $\chi(H[X])\geqslant\chi(H)-\chi(H-X)\geqslant\chi(H)-q+1>s$. Since $\had(H[X])\leqslant\had(H)<b$, Theorem~\ref{thm:RS} gives $v(H[X])\leqslant2\alpha(H[X])\had(H[X])\leqslant2bk(k-1)$. So $\chi(H[X])\leqslant\Gamma$ by hypothesis, and hence $s<\chi(H[X])\leqslant\Gamma$.
\end{proof}

\subsection{Paths to a large set}

The proofs of Lemma~\ref{lem:weave} and Proposition~\ref{prop:DP} use the following two constructions of paths. Both are taken from the proofs of Delcourt and Postle~\cite[Theorem~7.1 and Lemma~6.4]{DP}. Part~(a) is their application of Menger's theorem, which they state without details. We obtain it by replacing each vertex of the set $W$ by two non-adjacent copies. Parts~(b) and~(c) adapt the identification argument in their proof of~\cite[Lemma~6.4]{DP}, combined with their Lemma~5.3 (Lemma~\ref{lem:53}). Instead of identifying the two vertices of each pair, we add an auxiliary vertex adjacent to both, and work in the union of the given paths. Parts~(b) and~(c) are stated in these forms in \texttt{LinearHadwiger}.

\begin{lemma}\label{lem:paths}
Let $G$ be a graph, and let $W,B\subseteq V(G)$.
\begin{enumerate}[label=(\alph*),nosep]
\item Suppose that $W\cap B=\emptyset$, that $G$ is $r$-connected for some integer $r\geqslant2|W|$, and that $|B|\geqslant2|W|$. Then $G$ contains a double fan $\mathcal P$ from $W$ to $B$ with $\chi(G[V(\mathcal P)])\leqslant4|W|$.
\item Let $\mathcal P$ be a double fan from $W$ to $B$ in $G$. Let $m\geqslant0$ be an integer, and let $U\subseteq V(G)$ be disjoint from $B\cup V(\mathcal P)$. Let $\{Q_{i,j}:i\in[m],\ j\in\{1,2\}\}$ be a family $\mathcal Q$ of $2m$ pairwise disjoint paths in $G-W$, each with its first vertex in $U$, no other vertex in $U$, and its last vertex in $B$. Then $G$ contains $|W|+m$ pairwise disjoint paths $P'_w$ $(w\in W)$ and $Q'_i$ $(i\in[m])$ with the following properties. Each $P'_w$ has first vertex $w$. Each $Q'_i$ has as its first vertex the first vertex of $Q_{i,1}$ or of $Q_{i,2}$. Each of these paths has its last vertex in $B$, meets $W\cup U\cup B$ only in its ends, and has all its vertices in $V(\mathcal P)\cup V(\mathcal Q)$.
\item Let $\mathcal P$, $m$, $U$ and $\mathcal Q$ be as in~(b), except that the paths of $\mathcal Q$ may have further vertices in $U$. Then $G$ contains $|W|+m$ pairwise disjoint paths $P'_w$ $(w\in W)$ and $Q'_i$ $(i\in[m])$ with the following properties. Each $P'_w$ has first vertex $w$. Each $Q'_i$ has as its first vertex a vertex of $U$ on $Q_{i,1}$ or on $Q_{i,2}$. Each of these paths has its last vertex in $B$ and meets $W\cup U\cup B$ only in its ends.
\end{enumerate}
\end{lemma}

\begin{figure}[ht]
\centering
\begin{tikzpicture}[scale=0.85,
  vtx/.style={circle,fill,inner sep=1.2pt},
  hv/.style={circle,draw,fill=white,inner sep=1.4pt},
  lab/.style={font=\footnotesize},
  ppath/.style={orange!80!black,thick},
  qpath/.style={blue!70!black,thick},
  res/.style={line width=2.2pt,opacity=0.35}]
% ---------- (a) ----------
\draw[fill=gray!10] (4.2,0.9) ellipse (0.5 and 1.25);
\node[lab] at (4.2,2.4) {$B$};
\node[hv] (w1) at (0,1.6) {}; \node[lab,left] at (w1) {$w^1$};
\node[hv] (w2) at (0,0.2) {}; \node[lab,left] at (w2) {$w^2$};
\node[vtx] (n1) at (1.2,1.7) {}; \node[vtx] (n2) at (1.2,0.3) {};
\draw[gray,dashed] (w1)--(n2); \draw[gray,dashed] (w2)--(n1);
\draw[orange!80!black,thick] (w1)--(n1); \draw[orange!80!black,thick] (w2)--(n2);
\draw[ppath] (n1) .. controls (2.3,1.9) and (3.0,1.4) .. (4.0,1.5);
\draw[ppath] (n2) .. controls (2.3,0.1) and (3.0,0.5) .. (4.0,0.4);
\node[lab,align=center] at (2.4,-0.65) {replace $w^1$ and $w^2$ by $w$};
\node at (2.1,-1.35) {(a)};
% ---------- (b) ----------
\begin{scope}[xshift=7.2cm]
\draw[fill=gray!10] (4.2,0.9) ellipse (0.5 and 1.4);
\node[lab] at (4.2,2.55) {$B$};
\node[vtx] (bw) at (0,1.9) {}; \node[lab,left] at (bw) {$w$};
\node[vtx] (y1) at (0,0.35) {}; \node[vtx] (y2) at (0,-0.25) {};
\node[lab] at (0.05,0.65) {$y_1$}; \node[lab] at (0.05,-0.55) {$y'_1$};
\node[hv] (c1) at (-0.9,0.05) {}; \node[lab,left] at (c1) {$c_1$};
\draw[gray,dashed] (c1)--(y1);
\draw[res,blue!70!black] (c1)--(y2); \draw[gray,dashed] (c1)--(y2);
\draw[res,orange!80!black] (bw) .. controls (1.5,2.4) and (2.8,1.6) .. (4.0,1.8);
\draw[ppath] (bw) .. controls (1.5,2.4) and (2.8,1.6) .. (4.0,1.8);
\draw[ppath] (bw) .. controls (1.5,1.3) and (2.8,1.4) .. (4.0,1.1);
\draw[qpath] (y1) .. controls (1.5,0.6) and (2.8,0.3) .. (4.0,0.5);
\draw[res,blue!70!black] (y2) .. controls (1.5,-0.5) and (2.8,0.0) .. (4.0,-0.1);
\draw[qpath] (y2) .. controls (1.5,-0.5) and (2.8,0.0) .. (4.0,-0.1);
\node[lab,orange!80!black] at (2.2,2.35) {$\mathcal P$};
\node[lab,blue!70!black] at (2.2,-0.75) {$\mathcal Q$};
\node at (2.1,-1.35) {(b)};
\end{scope}
\end{tikzpicture}
\caption{The proof of Lemma~\ref{lem:paths}. (a)~The vertex $w$ is replaced by two non-adjacent copies $w^1$ and $w^2$ (white), each adjacent to the neighbours of $w$ outside $W$. Replacing $w^1$ and $w^2$ by $w$ in the Menger paths from $w^1$ and $w^2$ gives two paths from $w$ to $B$ that meet only in $w$. (b)~The proof of part~(b). Here $W=\{w\}$ and $m=1$. An auxiliary vertex $c_1$ (white), adjacent only to $y_1$ and $y'_1$, is added, and Lemma~\ref{lem:53} gives disjoint paths from $w$ and from $c_1$ in the union of the paths. In the highlighted outcome, the path from $c_1$ continues through $y'_1$, and deleting $c_1$ gives a path from $y'_1$.}\label{fig:paths}
\end{figure}
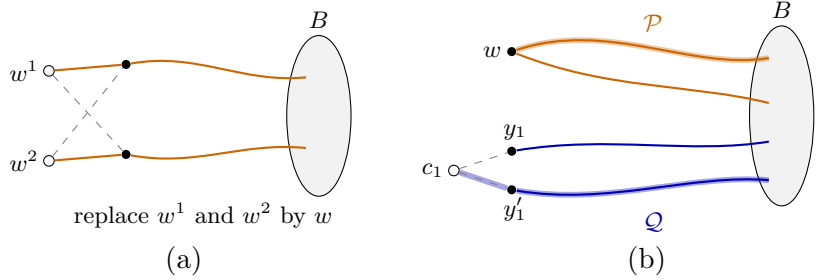

\begin{proof}
(a) As illustrated in Figure~\ref{fig:paths}(a), we replace each vertex of $W$ by two copies and apply Menger's theorem. Let $G^\times$ be obtained from $G-W$ by adding, for each $w\in W$, two new non-adjacent vertices $w^1$ and $w^2$, each adjacent to every vertex of $N_G(w)\setminus W$. Let $W^\times:=\{w^1,w^2:w\in W\}$, so $|W^\times|=2|W|$. We claim that every set $Q\subseteq V(G^\times)$ that meets every path from $W^\times$ to $B$ in $G^\times$ has at least $2|W|$ vertices. Suppose not. Let $O:=Q\setminus W^\times$ and $W_0:=\{w\in W:w^1,w^2\in Q\}$, and let $D:=O\cup W_0\subseteq V(G)$. Then $|D|\leqslant|O|+2|W_0|\leqslant|Q|<2|W|\leqslant r$, so $G-D$ is connected. Since $|Q|<|W^\times|$, some $w^i$ is not in $Q$, and then $w\notin D$. Since $|O|<2|W|\leqslant|B|$, some $b\in B$ is not in $O$, and $b\notin D$ since $b\notin W$. Let $P_0$ be a $w$--$b$ path in $G-D$, let $z$ be the last vertex of $P_0$ in $W$, and let $P_1$ be the subpath of $P_0$ from $z$ to $b$. Since $z\notin W_0$, some $z^j$ is not in $Q$. The second vertex of $P_1$ lies in $N_G(z)\setminus W$, so replacing $z$ by $z^j$ in $P_1$ gives a path in $G^\times$ from $z^j$ to $b$. This path avoids $Q$, since its vertices other than $z^j$ lie in $V(G)\setminus(W\cup D)$ and hence not in $O$. This contradiction proves the claim. By Theorem~\ref{thm:menger}, $G^\times$ has $2|W|$ pairwise vertex-disjoint paths from $W^\times$ to $B$. Since $|W^\times|=2|W|$, each vertex of $W^\times$ is the first vertex of exactly one of them, and hence is the only vertex of $W^\times$ on this path. Replacing the first vertex $w^i$ of each of these paths by $w$ gives a path in $G$ from $w$ to $B$ with no other vertex in $W$, since all its other vertices lie in $V(G)\setminus W$ and its second vertex is a neighbour of $w$. The two paths obtained from $w^1$ and $w^2$ start at $w$, and any two of the resulting paths are disjoint outside $W$. So they form a double fan from $W$ to $B$. Finally, replace each path $P$ of this double fan by a shortest path between its ends in $G[V(P)]$. This is an induced path of $G$ with the same ends and with vertex set contained in $V(P)$, so the resulting family $\mathcal P$ is again a double fan from $W$ to $B$. Each path of $\mathcal P$ is induced in $G$, and so has a proper $2$-colouring. Assign each vertex of $V(\mathcal P)$ to one path of $\mathcal P$ containing it, and colour the vertex with the pair consisting of this path and the colour of the vertex in the $2$-colouring of this path. If two adjacent vertices are assigned to the same path, then they are adjacent on this induced path, and so have distinct colours in its $2$-colouring. So adjacent vertices receive distinct pairs, and $\chi(G[V(\mathcal P)])\leqslant2\cdot2|W|=4|W|$. Moreover, if $R\subseteq V(G)$ contains $W\cup B$ and $G[R]$ satisfies the hypotheses of~(a), then~(a) applied to $G[R]$ gives a double fan $\mathcal P$ from $W$ to $B$ in $G$ with $V(\mathcal P)\subseteq R$ and $\chi(G[V(\mathcal P)])\leqslant4|W|$, since $G[R][V(\mathcal P)]=G[V(\mathcal P)]$.

(b) As illustrated in Figure~\ref{fig:paths}(b), we add an auxiliary vertex for each pair of paths $Q_{i,1},Q_{i,2}$ and apply Lemma~\ref{lem:53}. For $i\in[m]$, let $y_i$ and $y'_i$ be the first vertices of $Q_{i,1}$ and $Q_{i,2}$, and let $Y:=\{y_1,y'_1,\dots,y_m,y'_m\}$. Let $C:=\{c_1,\dots,c_m\}$ be a set of $m$ new vertices. Let $G^+$ be the graph with vertex set $V(\mathcal P)\cup V(\mathcal Q)\cup C$ whose edges are the edges of the paths in $\mathcal P\cup\mathcal Q$ together with the edges $c_iy_i$ and $c_iy'_i$ for $i\in[m]$. The paths of $\mathcal P$ avoid $U$, and each path of $\mathcal Q$ meets $U$ only in its first vertex, so $V(G^+)\cap U=Y$. Each $y\in\{y_i,y'_i\}$ is not in $B$, since $U\cap B=\emptyset$. So the path of $\mathcal Q$ starting at $y$ has a second vertex, and $y$ has exactly two neighbours in $G^+$, namely $c_i$ and this second vertex. In $G^+$ the hypotheses of Lemma~\ref{lem:53} hold with $A_1=W$, $A_2=C$ and the same $B$, since:
\begin{itemize}[nosep]
\item The paths of $\mathcal P$ form a double fan from $W$ to $B$ in $G^+-C$.
\item Adding $c_i$ at the start of $Q_{i,1}$ and of $Q_{i,2}$, for each $i\in[m]$, gives a double fan from $C$ to $B$ in $G^+-W$, since the paths of $\mathcal Q$ are pairwise disjoint and lie in $G-W$.
\end{itemize}
Lemma~\ref{lem:53} gives $|W|+m$ pairwise disjoint paths from $W\cup C$ to $B$ in $G^+$. Since $|W\cup C|=|W|+m$, exactly one of them starts at each vertex of $W\cup C$, and so each of them meets $W\cup C$ only in its first vertex. Replace each of these paths by its subpath from its first vertex to its first vertex in $B$. Let $L$ be one of the resulting paths, and let $y\in Y\cap V(L)$, with $y\in\{y_i,y'_i\}$. Since $y\notin W\cup C\cup B$, the vertex $y$ is an internal vertex of $L$, so both its neighbours in $G^+$ are its neighbours on $L$. In particular $c_i\in V(L)$. Since $L$ meets $W\cup C$ only in its first vertex, $L$ starts at $c_i$, and $y$ is the second vertex of $L$. Hence the paths starting in $W$ avoid $Y$, and, for each $i\in[m]$, the path starting at $c_i$ meets $Y$ only in its second vertex, which lies in $\{y_i,y'_i\}$. For each $i\in[m]$, delete $c_i$ from the path starting at $c_i$. All edges of $G^+$ not incident with $C$ are edges of $G$, so this gives $|W|+m$ pairwise disjoint paths in $G$ with all their vertices in $V(\mathcal P)\cup V(\mathcal Q)$. The path $P'_w$ starts at $w$, and the path $Q'_i$ starts at $y_i$ or at $y'_i$. Each of them meets $W$ only in its first vertex, meets $U$ only in its first vertex since $V(G^+)\cap U=Y$, and meets $B$ only in its last vertex.

(c) For each path $Q_{i,j}$ of $\mathcal Q$, let $Q^\ast_{i,j}$ be the subpath of $Q_{i,j}$ from its last vertex in $U$ to its last vertex. The paths $Q^\ast_{i,j}$ are pairwise disjoint paths in $G-W$, each with its first vertex in $U$, no other vertex in $U$, and its last vertex in $B$. Applying~(b) to them gives the required paths, since the first vertex of each $Q^\ast_{i,j}$ is a vertex of $U$ on $Q_{i,j}$.
\end{proof}

\subsection{A sharpened form of Mader's theorem}

The proofs of Lemmas~\ref{lem:hub} and~\ref{lem:massed} and of Theorem~\ref{thm:refinedlin} use the following form of Mader's theorem.

\begin{lemma}[Mader~\cite{Mad72}, see~{\cite[Theorem~1.4.3]{Die17}}]\label{lem:mader}
Let $k\geqslant1$ be an integer, and let $G$ be a non-null graph with $e(G)\geqslant2k\,v(G)$. Let $\gamma:=e(G)/v(G)$. Let $\gamma':=\frac12\bigl(\gamma+\sqrt{\gamma^2-2\gamma k}\bigr)$. Then $G$ has a $(k+1)$-connected induced subgraph $H$ with $\delta(H)>\gamma$, $v(H)>2\gamma'$ and $e(H)>\gamma'\,v(H)\geqslant(\gamma-k)\,v(H)$.
\end{lemma}

The bounds $e(H)>\gamma'v(H)$ and $v(H)>2\gamma'$ are used in the proof of Lemma~\ref{lem:hub}.

\begin{proof}
The bounds $\delta(H)>\gamma$ and $e(H)>(\gamma-k)v(H)$ are in the statement and proof of~\cite[Theorem~1.4.3]{Die17}. The bounds $v(H)>2\gamma'$ and $e(H)>\gamma'v(H)$ are a slight sharpening that follows from the same proof. We give that proof, and record these bounds along the way. Here $\gamma\geqslant2k$. Call an induced subgraph $G'$ of $G$ \defn{good} if $v(G')\geqslant2k$ and $e(G')>\gamma\,(v(G')-k)$. The graph $G$ itself is good, since $e(G)=\gamma\,v(G)$ and $v(G)>2\gamma\geqslant4k$ (the average degree $2\gamma$ is less than $v(G)$). Let $H$ be a good subgraph with as few vertices as possible. No good graph has exactly $2k$ vertices, since it would have more than $\gamma k\geqslant2k^2>\binom{2k}2$ edges. So $v(H)>2k$. If $H$ had a vertex $v$ with $\deg_H(v)\leqslant\gamma$, then $H-v$ would have $v(H)-1\geqslant2k$ vertices and $e(H-v)>\gamma(v(H)-k)-\gamma=\gamma(v(H-v)-k)$ edges. So $H-v$ would be good, contradicting the choice of $H$. Hence $\delta(H)>\gamma$. In particular $v(H)>\gamma$. Let $n:=v(H)$ and $r:=\sqrt{\gamma^2-2\gamma k}$, which is real since $\gamma\geqslant2k$. Since $H$ is good and $e(H)\leqslant n(n-1)/2<n^2/2$, we have $n^2/2>\gamma(n-k)$, that is, $n^2-2\gamma n+2\gamma k>0$. The roots of this quadratic are $\gamma\pm r$, and $n>\gamma\geqslant\gamma-r$, so $n>\gamma+r=2\gamma'$. Hence
\[
\frac{e(H)}n\ >\ \gamma\Bigl(1-\frac kn\Bigr)\ >\ \gamma-\frac{\gamma k}{\gamma+r}\ =\ \gamma-\frac{\gamma-r}2\ =\ \gamma',
\]
using $(\gamma+r)(\gamma-r)=2\gamma k$. Finally, $\gamma'\geqslant\gamma-k$ is equivalent to $r\geqslant\gamma-2k$. Since $\gamma-2k\geqslant0$, squaring shows that this is equivalent to $2\gamma k\geqslant4k^2$, which holds since $\gamma\geqslant2k$.

It remains to show that $H$ is $(k+1)$-connected. We have $v(H)>2k\geqslant k+1$. Suppose that $H$ has a separation $(U_1,U_2)$ with $|U_1\cap U_2|\leqslant k$ and both $U_1\setminus U_2$ and $U_2\setminus U_1$ non-empty. Let $H_i:=H[U_i]$. A vertex of $U_1\setminus U_2$ has all its neighbours in $U_1$, so $v(H_1)>\delta(H)>\gamma\geqslant2k$. Similarly $v(H_2)>2k$. Each $H_i$ has fewer vertices than $H$, so it is not good, and hence $e(H_i)\leqslant\gamma(v(H_i)-k)$. Every edge of $H$ lies in $H_1$ or $H_2$, and $v(H_1)+v(H_2)\leqslant v(H)+k$. Therefore
\[
e(H)\ \leqslant\ e(H_1)+e(H_2)\ \leqslant\ \gamma\bigl(v(H_1)+v(H_2)-2k\bigr)\ \leqslant\ \gamma\,(v(H)-k),
\]
a contradiction, since $H$ is good. So $H$ is $(k+1)$-connected.
\end{proof}

%=====================================================================
\section{An explicit chromatic-separable reduction}\label{sec:sep}
%=====================================================================

This section proves Theorem~\ref{thm:sepexplicit}, the engine of each bootstrapping step. Roughly, it states that if $\had(G)<t$ and, for every real $b\geqslant a_0$, every $s$-chromatic-inseparable induced subgraph $H$ of $G$ with $\had(H)<b$ satisfies $s\leqslant\Phi b$, then $\chi(G)$ is at most about $\cS(\Phi)\,t$, where $\cS(\Phi)$ is defined in Theorem~\ref{thm:sepexplicit}. The section makes explicit a lemma and a theorem of Norin and Steiner~\cite[Lemma~B.12 and Theorem~15]{NS}. These are adaptations of results of Delcourt and Postle~\cite[Theorems~7.1 and~1.6]{DP}. Lemma~\ref{lem:hub} constructs the hub used in the proof of Lemma~\ref{lem:weave}. Theorem~\ref{thm:sepexplicit} follows from Lemma~\ref{lem:weave}, and is used in Lemma~\ref{lem:level}. Lemma~\ref{lem:weave} is also adapted in Section~\ref{sec:proof}, as Lemma~\ref{lem:weavefinal}.

\begin{lemma}\label{lem:hub}
Let $m\geqslant1$ be an integer and let $q:=\lceil25m/2\rceil+1$. Every graph $F$ with $\chi(F)\geqslant q$ has an $m$-linked induced subgraph $H_0$ with $v(H_0)>10m$ and $\chi(H_0)\leqslant q$.
\end{lemma}

\begin{proof}
Let $S$ be a smallest set of vertices of $F$ with $\chi(F[S])\geqslant q$, and let $F_0:=F[S]$. Then $\chi(F_0-v)<q$ for every vertex $v$ of $F_0$. Since $q\geqslant1$, $F_0$ has a vertex, and deleting a vertex lowers the chromatic number by at most $1$, so $\chi(F_0)=q$. If some vertex $v$ of $F_0$ had $\deg_{F_0}(v)<q-1$, then a $(q-1)$-colouring of $F_0-v$ would extend to $F_0$, a contradiction. Hence $\delta(F_0)\geqslant q-1$, a classical property of critical graphs (see Dirac~\cite{Dir52}). Thus $\delta(F_0)\geqslant25m/2$, so $\gamma:=e(F_0)/v(F_0)\geqslant25m/4\geqslant4m$. Apply Lemma~\ref{lem:mader} to $F_0$ with $2m$ in place of $k$, and let $\gamma':=\frac12(\gamma+\sqrt{\gamma^2-4m\gamma})$. Then $\gamma'\geqslant5m$. Indeed, this is clear if $\gamma\geqslant10m$. Otherwise $0<10m-\gamma$, and $(10m-\gamma)^2\leqslant\gamma^2-4m\gamma$ since this is equivalent to $\gamma\geqslant25m/4$. So $\sqrt{\gamma^2-4m\gamma}\geqslant10m-\gamma$, and hence $\gamma'\geqslant5m$. So $F_0$ has a $(2m+1)$-connected induced subgraph $H_0$ with $e(H_0)>5m\,v(H_0)$ and $v(H_0)>10m$. By Theorem~\ref{thm:TW}, $H_0$ is $m$-linked. Since $F_0$ is an induced subgraph of $F$, so is $H_0$. Finally, $\chi(H_0)\leqslant\chi(F_0)=q$.
\end{proof}

Lemma~\ref{lem:hub} needs only a $(2m+1)$-connected subgraph, because it uses Theorem~\ref{thm:TW}. The connectivity form of the Thomas--Wollan theorem would need a $10m$-connected subgraph, and extracting one by Corollary~\ref{cor:nguyen} needs $\chi(F)\geqslant\lceil3.0625(10m-1)\rceil\approx30.6m$, instead of about $12.5m$. The density form needs the sharper bound $e(H)>\gamma'v(H)$ of Lemma~\ref{lem:mader}. The weaker bound $e(H)>(\gamma-k)v(H)$ gives only $e(H)>4.25m\,v(H)$, which is less than the $5m\,v(H)$ that Theorem~\ref{thm:TW} requires.

\begin{lemma}\label{lem:weave}
Let $\eta\in(0,1]$, $a_0\geqslant4+5/\eta$ and $\Phi\geqslant0$ be real. Let
\[
{\LamE}:=(1+\eta)\bigl(517.5+12\Phi\bigr),\qquad{\LamU}:=(1+\eta)\bigl(897.5+10\Phi\bigr),
\]
and let $\Lambda:=\min\{{\LamE},{\LamU}\}$. Let $a\geqslant1$ be an integer, and let $G$ be a $21a$-connected graph that satisfies $(\mathcal H_{\Phi,a_0})$ (defined in Section~\ref{sec:defs}) and
\[
\chi(G)\ \geqslant\ D\bigl(\tfrac{46}3a_0+4\bigr)+1+\Lambda a .
\]
Then $G$ is $(a,3a)$-woven if $\Lambda={\LamE}$, and $(a,\lceil10a/3\rceil)$-woven otherwise.
\end{lemma}

We first give the idea of the proof, which uses induction on $a$. The roots and terminals are joined by paths to a linked subgraph $H_0$, called the hub. Through the hub, the terminals are joined in pairs, and the roots are joined to a $21a$-connected subgraph $G_3$ of large chromatic number. This reduces the required wovenness of $G$ to the $(a,0)$-wovenness of $G_3$ (Steps~1--6). If $G_3$ has a $K_{2a}$ minor, then Lemma~\ref{lem:kawa} finishes the proof (Step~7a). Otherwise two chromatic separations in a subgraph of $G_3$ give three disjoint $21a$-connected subgraphs of large chromatic number. These subgraphs are woven by induction (Steps~7b and~8). Each branch set of the required $K_a$ model in $G_3$ is assembled from branch sets in two of these three subgraphs (Step~9).

Step~9 distributes the $a$ branch sets among the three subgraphs. In the \defn{uneven scheme}, the subgraph $J'_3$, which is obtained by one chromatic separation rather than two (Step~7b), receives at most $\lceil4a/5\rceil$ of them. Each of the other two subgraphs $J'_1$ and $J'_2$ then receives only $\lceil3a/5\rceil$ branch sets. Each of $J'_1$ and $J'_2$ must still be woven with respect to all $2a$ paths of the linkage in Step~9. So the lemma must give $(a,\lceil10a/3\rceil)$-wovenness rather than $(a,3a)$-wovenness. Indeed, Step~8 applies it by induction with $\lceil3a/5\rceil$ in place of $a$, and Step~9 uses that $\lceil10\lceil3a/5\rceil/3\rceil\geqslant2a$. In the \defn{even scheme}, each subgraph receives at most $\lceil2a/3\rceil$ branch sets, and $3\lceil2a/3\rceil\geqslant2a$, so $(a,3a)$-wovenness suffices. The uneven scheme reduces the coefficient of $\Phi$ from $12$ to $10$, at the cost of increasing the constant term from $517.5$ to $897.5$. The even scheme gives the smaller value of $\Lambda$ when $\Phi<190$, and the two values coincide at $\Phi=190$. So Lemma~\ref{lem:weave} uses the even scheme if and only if $\Phi\leqslant190$. The even scheme is therefore used in the bootstrapping steps whose Gy\'arf\'as graphs are coloured by the base case, where $\Phi=4.01$ (Proposition~\ref{prop:explicitbootstrap}). The uneven scheme is used in the final step of Proposition~\ref{prop:explicitbootstrap}, where $\Phi=615.6401$ is the constant of the side step, and in the final application (Lemma~\ref{lem:weavefinal}), where the role of $\Phi$ is played by $78+2{\cst}\approx9.0\cdot10^4$. The scheme is determined by $\Phi$, so it is the same at every level of the induction below. We give the proof for the uneven scheme ($\Lambda={\LamU}<{\LamE}$), and then list the changes for the even scheme. Figure~\ref{fig:weave} illustrates Steps~3--6 and Step~9.

\begin{proof}
We proceed by induction on $a$. Every subgraph extracted below by Corollary~\ref{cor:nguyen} is induced, and the subgraphs produced by the chromatic separations in Step~7b are chosen induced. Hence every graph to which the induction hypothesis is applied is an induced subgraph of $G$ and satisfies $(\mathcal H_{\Phi,a_0})$. The proof uses the following subgraphs, which are defined in Steps~3, 5 and~7b. Each of them is an induced subgraph of $G$.
\[
G\supseteq G_1\supseteq H_0\qquad\text{and}\qquad G_1\supseteq G_2\supseteq G_3\supseteq G_4\supseteq J_i\supseteq J'_i\quad(i\in\{1,2,3\}).
\]
For the uneven scheme, let $\ell_0:=\lceil10a/3\rceil$ and $h:=2a+4\ell_0$. We show that $G$ is $(a,\ell_0)$-woven. In Step~1, $\ell\in\{0,\dots,\ell_0\}$ denotes the number of terminal pairs. Then
\[
\ell\leqslant\ell_0\leqslant\tfrac{10a+2}3,\qquad a+\ell\leqslant\tfrac{13}3a+1,\qquad a+2\ell\leqslant\tfrac{23}3a+2,\qquad h\leqslant\tfrac{46}3a+4.
\]

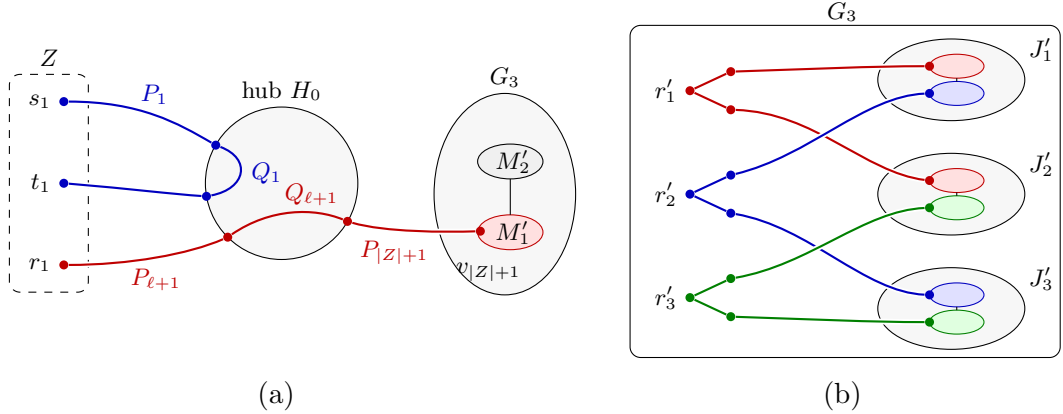
\begin{figure}[ht]
\centering
\begin{tikzpicture}[scale=0.72,
  vtx/.style={circle,fill,inner sep=1.2pt},
  lab/.style={font=\footnotesize},
  rpath/.style={red!75!black,thick,preaction={draw,white,line width=2.2pt}},
  bpath/.style={blue!75!black,thick,preaction={draw,white,line width=2.2pt}},
  gpath/.style={green!50!black,thick,preaction={draw,white,line width=2.2pt}}]
% ---------- (a) Steps 3--6 ----------
\draw[dashed,rounded corners] (-1.0,-0.5) rectangle (0.45,3.5);
\node[lab] at (-0.27,3.8) {$Z$};
\node[vtx,blue!75!black] (s1) at (0,3) {}; \node[lab,left] at (-0.05,3) {$s_1$};
\node[vtx,blue!75!black] (t1) at (0,1.5) {}; \node[lab,left] at (-0.05,1.5) {$t_1$};
\node[vtx,red!75!black] (r1) at (0,0) {}; \node[lab,left] at (-0.05,0) {$r_1$};
\draw[fill=gray!6] (4,1.5) circle (1.4);
\node[lab] at (4,3.2) {hub $H_0$};
\node[vtx,blue!75!black] (u1) at (2.79,2.2) {};
\node[vtx,blue!75!black] (u2) at (2.62,1.26) {};
\node[vtx,red!75!black] (u3) at (3.01,0.51) {};
\node[vtx,red!75!black] (u4) at (5.21,0.8) {};
\draw[fill=gray!6] (8.1,1.3) ellipse (1.3 and 1.85);
\node[lab] at (8.1,3.45) {$G_3$};
\draw[fill=red!12,draw=red!75!black] (8.2,0.6) ellipse (0.6 and 0.32);
\draw[fill=gray!12] (8.2,1.9) ellipse (0.6 and 0.32);
\draw (8.2,0.92)--(8.2,1.58);
\node[lab] at (8.25,0.6) {$M'_1$};
\node[lab] at (8.25,1.9) {$M'_2$};
\node[vtx,red!75!black] (v1) at (7.65,0.62) {};
\draw[bpath] (s1) .. controls (1.4,3.0) and (2.1,2.6) .. node[lab,above] {$P_1$} (u1);
\draw[bpath] (t1) .. controls (1.2,1.4) and (2.0,1.3) .. (u2);
\draw[blue!75!black,thick] (u1) .. controls (3.45,1.95) and (3.35,1.4) .. node[lab,right] {$Q_1$} (u2);
\draw[rpath] (r1) .. controls (1.2,0.0) and (2.3,0.2) .. node[lab,below] {$P_{\ell+1}$} (u3);
\draw[red!75!black,thick] (u3) .. controls (3.8,1.05) and (4.6,1.05) .. (u4);
\node[lab,red!75!black] at (4.55,1.3) {$Q_{\ell+1}$};
\draw[rpath] (u4) .. controls (5.9,0.6) and (7.0,0.6) .. node[lab,below,pos=0.35,red!75!black] {$P_{|Z|+1}$} (v1);
\node[lab,anchor=north] at (7.8,0.22) {$v_{|Z|+1}$};
\node at (3.9,-2.4) {(a)};
% ---------- (b) Step 9 ----------
\begin{scope}[xshift=11.3cm,yshift=-0.2cm]
\draw[rounded corners] (-0.9,-1.5) rectangle (7.0,4.6);
\node[lab] at (3.0,4.85) {$G_3$};
\foreach \y/\c/\i in {3.4/red/1,1.5/blue/2,-0.4/green/3}{
  \ifnum\i=3 \def\cc{green!50!black}\else\def\cc{\c!75!black}\fi
  \node[vtx,\cc] (r\i) at (0.2,\y) {};
  \node[lab,left] at (0.15,\y) {$r'_{\i}$};
  \node[vtx,\cc] (sa\i) at (0.95,\y+0.35) {};
  \node[vtx,\cc] (sb\i) at (0.95,\y-0.35) {};
  \draw[\cc,thick] (sa\i)--(r\i)--(sb\i);
}
\foreach \y/\n in {3.6/1,1.5/2,-0.6/3}{
  \draw[fill=gray!6] (5.0,\y) ellipse (1.35 and 0.75);
  \node[lab] at (6.65,\y+0.55) {$J'_{\n}$};
}
\draw[fill=red!12,draw=red!75!black] (5.1,3.85) ellipse (0.5 and 0.22);
\draw[fill=blue!12,draw=blue!75!black] (5.1,3.35) ellipse (0.5 and 0.22);
\draw (5.1,3.63)--(5.1,3.57);
\draw[fill=red!12,draw=red!75!black] (5.1,1.75) ellipse (0.5 and 0.22);
\draw[fill=green!12,draw=green!50!black] (5.1,1.25) ellipse (0.5 and 0.22);
\draw (5.1,1.53)--(5.1,1.47);
\draw[fill=blue!12,draw=blue!75!black] (5.1,-0.35) ellipse (0.5 and 0.22);
\draw[fill=green!12,draw=green!50!black] (5.1,-0.85) ellipse (0.5 and 0.22);
\draw (5.1,-0.57)--(5.1,-0.63);
\node[vtx,red!75!black] (ta) at (4.6,3.85) {};
\node[vtx,blue!75!black] (tb) at (4.6,3.35) {};
\node[vtx,red!75!black] (tc) at (4.6,1.75) {};
\node[vtx,green!50!black] (td) at (4.6,1.25) {};
\node[vtx,blue!75!black] (te) at (4.6,-0.35) {};
\node[vtx,green!50!black] (tf) at (4.6,-0.85) {};
\draw[rpath] (sa1) .. controls (2.5,3.8) and (3.5,3.85) .. (ta);
\draw[rpath] (sb1) .. controls (2.5,3.0) and (3.4,1.9) .. (tc);
\draw[bpath] (sa2) .. controls (2.4,2.2) and (3.4,3.3) .. (tb);
\draw[bpath] (sb2) .. controls (2.4,0.9) and (3.4,-0.3) .. (te);
\draw[gpath] (sa3) .. controls (2.4,0.2) and (3.4,1.2) .. (td);
\draw[gpath] (sb3) .. controls (2.5,-0.8) and (3.5,-0.85) .. (tf);
\node at (3.0,-2.2) {(b)};
\end{scope}
\end{tikzpicture}
\caption{The proof of Lemma~\ref{lem:weave}. (a)~Steps~3--6. Paths through the $(a+\ell)$-linked hub $H_0$ join the terminals $s_1$ and $t_1$ (blue), and join the root $r_1$ to the branch set $M'_1$ of a $K_a$ model $\{M'_1,M'_2,\dots\}$ in $G_3$ (red). (b)~Step~9. Each root $r'_i$ with $i\in I_{jk}$ is joined through its neighbours $s'_i$ and $s'_{a+i}$ to branch sets of $\mathcal M_j$ and $\mathcal M_k$. Here $r'_1$, $r'_2$ and $r'_3$ (red, blue and green) lie in $I_{12}$, $I_{13}$ and $I_{23}$. Any two of these index pairs intersect, so any two new branch sets are adjacent. Only the branch sets containing terminals are shown.}\label{fig:weave}
\end{figure}

\medskip\noindent\emph{Step 1. Reduction to disjoint terminals.} Let $\ell\in\{0,\dots,\ell_0\}$, and let $R=\{r_1,\dots,r_a\}$, $S=\{s_1,\dots,s_\ell\}$ and $T=\{t_1,\dots,t_\ell\}$ be as in the definition of $(a,\ell_0)$-woven. As illustrated in Figure~\ref{fig:step1}, we modify the roots and terminals, using that $\delta(G)\geqslant21a$. The list $r_1,\dots,r_a,s_1,\dots,s_\ell,t_1,\dots,t_\ell$ has $a+2\ell$ entries, which need not be distinct vertices. Since $2(a+2\ell)\leqslant\frac{46}3a+4\leqslant21a$, each entry has at least $\delta(G)-|R\cup S\cup T|\geqslant21a-(a+2\ell)\geqslant a+2\ell$ neighbours outside $R\cup S\cup T$. So, by Hall's theorem, we can assign to each entry a neighbour of it outside $R\cup S\cup T$, such that the $a+2\ell$ assigned vertices are distinct. For each $i$ with $s_i=t_i$, let $t_i^*$ be the vertex assigned to the entry $t_i$. For each $i$ with $r_i\in S\cup T$, let $r_i^*$ be the vertex assigned to the entry $r_i$. Let $t_i^*:=t_i$ for each $i$ with $s_i\ne t_i$, and let $r_i^*:=r_i$ for each $i$ with $r_i\notin S\cup T$. Let $R^*:=\{r_1^*,\dots,r_a^*\}$ and $T^*:=\{t_1^*,\dots,t_\ell^*\}$. The vertices $r_i^*$ are distinct, the vertices $t_i^*$ are distinct, and $R^*$, $S$ and $T^*$ are pairwise disjoint. So $(R^*,S,T^*)$ is as in the definition of $(a,\ell_0)$-woven.

Suppose that $\mathcal M^*=\{M_1^*,\dots,M_a^*\}$ is a $K_a$ model with $r_i^*\in M_i^*$ for each $i\in[a]$, and that $\mathcal P^*$ is an $\{(s_i,t_i^*)\}$-linkage with $V(\mathcal M^*)\cap V(\mathcal P^*)=\emptyset$. For each $i$ with $s_i=t_i$, replace the path of $\mathcal P^*$ from $s_i$ to $t_i^*$ by the one-vertex path $s_i$. For each $i$ with $r_i\ne r_i^*$, add $r_i$ to $M_i^*$. This keeps $G[M_i^*]$ connected, since $r_ir_i^*\in E(G)$. Each such $r_i$ lies in $S\cup T$. If $r_i=t_j$ and $s_j=t_j$, then $r_i=s_j\in S$. Otherwise $r_i\in S$ or $r_i=t_j=t_j^*\in T^*$. So $r_i\in S\cup T^*\subseteq V(\mathcal P^*)$, and hence $r_i\notin V(\mathcal M^*)$. Hence the new branch sets are pairwise disjoint. Each contains exactly one vertex of $R$, since every $r_j$ either equals $r_j^*\in M_j^*$ or lies in $V(\mathcal P^*)$. Let $\mathcal M$ and $\mathcal P$ be the resulting families. Then $\mathcal M$ is a $K_a$ model rooted at $R$, and $\mathcal P$ is an $\{(s_i,t_i)\}$-linkage. Since $V(\mathcal P)\subseteq V(\mathcal P^*)$ and $V(\mathcal M^*)\cap V(\mathcal P^*)=\emptyset$, every vertex $x\in V(\mathcal M)\cap V(\mathcal P)$ is one of the added vertices, so $x=r_i$ for some $i$ with $r_i\in S\cup T$. Conversely, each $r_i\in S\cup T$ lies in the $i$th branch set of $\mathcal M$ and in $\{s_j,t_j\}$ for some $j$, so $r_i$ is an end of the path of $\mathcal P$ from $s_j$ to $t_j$. Hence $V(\mathcal M)\cap V(\mathcal P)=R\cap(S\cup T)$. So it suffices to treat pairwise disjoint sets $R,S,T$, which we assume from now on.

\begin{figure}[ht]
\centering
\begin{tikzpicture}[scale=0.85,
  vtx/.style={circle,fill,inner sep=1.2pt},
  lab/.style={font=\footnotesize},
  lpath/.style={orange!80!black,very thick}]
\draw[rounded corners=10pt,dashed,red!70!black] (-0.3,0.2) rectangle (3.55,2.15);
\node[lab,red!70!black] at (-0.05,-0.05) {$M_1$};
\draw[fill=red!10,draw=red!70!black] (1.0,1.0) ellipse (0.95 and 0.5);
\node[lab] at (0.25,1.8) {$M_1^*$};
\node[vtx] (a) at (0.5,1.1) {}; \node[vtx] (rs) at (1.5,1.0) {};
\node[lab] at (1.5,0.75) {$r_1^*$};
\draw (a)--(rs);
\node[vtx] (r1) at (2.95,1.0) {}; \node[lab] at (2.95,0.68) {$r_1=s_2$};
\draw[red!70!black,very thick] (rs)--(r1);
\node[vtx] (t2) at (7.2,1.0) {}; \node[lab,right] at (t2) {$t_2$};
\draw[lpath] (r1) .. controls (4.3,1.5) and (5.8,0.5) .. node[lab,above,pos=0.55] {$P_2$} (t2);
\node[vtx] (s1) at (3.6,-0.9) {}; \node[lab,left] at (s1) {$s_1=t_1$};
\node[vtx] (ts) at (5.2,-0.9) {}; \node[lab,right] at (ts) {$t_1^*$};
\draw[orange!80!black,thick,dashed] (s1)--(ts);
\end{tikzpicture}
\caption{Step~1 of the proof of Lemma~\ref{lem:weave}. The root $r_1=s_2$ is replaced in $R$ by a neighbour $r_1^*$, and the pair with $s_1=t_1$ gets a new end $t_1^*$. Given a model $\mathcal M^*$ rooted at $R^*$ and a linkage $\mathcal P^*$ disjoint from it, we add $r_1$ to $M_1^*$ to obtain $M_1$ (dashed), and replace the path from $s_1$ to $t_1^*$ (dashed) by the one-vertex path $s_1$. The model and the linkage then meet exactly in $r_1$.}\label{fig:step1}
\end{figure}
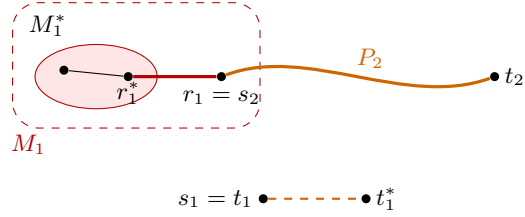

Let $Z:=R\cup S\cup T$, so $|Z|=a+2\ell$. Index $Z$ as $z_i:=s_i$ $(1\leqslant i\leqslant\ell)$, $z_{\ell+i}:=r_i$ $(1\leqslant i\leqslant a)$ and $z_{a+\ell+i}:=t_i$ $(1\leqslant i\leqslant\ell)$.

\medskip\noindent\emph{Step 2. Two easy cases.} Suppose $G$ has a $K_h$ minor. Since $h\geqslant2|Z|$, the graph $G$ has a $K_{2|Z|}$ minor. Since $\kappa(G)\geqslant21a\geqslant|Z|$, Lemma~\ref{lem:kawa} shows that $G$ is $(|Z|,0)$-woven, so $G$ has a $K_{|Z|}$ model $\{B_z:z\in Z\}$ rooted at $Z$. The sets $B_{r_i}$ form a $K_a$ model rooted at $R$. For each $i\leqslant\ell$, the sets $B_{s_i}$ and $B_{t_i}$ are adjacent and induce connected subgraphs, so their union contains an $s_i$--$t_i$ path. These paths are pairwise disjoint and disjoint from the sets $B_{r_i}$. So $G$ is $(a,\ell_0)$-woven. We may therefore assume that $G$ has no $K_h$ minor. If $a\leqslant a_0$, then $h\leqslant\frac{46}3a_0+4$. Since $h$ is an integer, it follows that $h\leqslant\lfloor\frac{46}3a_0+4\rfloor$. The function $D$ is non-decreasing, since a graph with no $K_s$ minor has no $K_{s+1}$ minor. So $\chi(G)\leqslant D(h)\leqslant D(\frac{46}3a_0+4)$, contrary to the hypothesis. Hence $a>a_0\geqslant4+5/\eta\geqslant9$, and so $a\geqslant10$, since $a$ is an integer.

\medskip\noindent\emph{Step 3. The hub.} Let $G_1:=G-Z$. Then $\kappa(G_1)\geqslant21a-|Z|\geqslant2a$. This bound is used only in Step~6. Also $\chi(G_1)\geqslant\chi(G)-|Z|\geqslant\lceil12.5(a+\ell)\rceil+1$, since $\chi(G)-|Z|$ is an integer, $\chi(G)\geqslant1+\Lambda a\geqslant1+897.5a\geqslant62a+16$ and $|Z|+12.5(a+\ell)+1\leqslant62a+16$. By Lemma~\ref{lem:hub} with $a+\ell$ in place of $m$, $G_1$ has an $(a+\ell)$-linked induced subgraph $H_0$, which we call the \defn{hub}, with $v(H_0)>10(a+\ell)$ and $\chi(H_0)\leqslant\lceil12.5(a+\ell)\rceil+1\leqslant12.5(a+\ell)+2$.

\medskip\noindent\emph{Step 4. Paths to the hub.} We have $\kappa(G)\geqslant21a$ and $2|Z|\leqslant\frac{46}3a+4\leqslant21a$. The set $V(H_0)$ is disjoint from $Z$, since $H_0\subseteq G_1$, and $v(H_0)>10(a+\ell)\geqslant2|Z|$. So Lemma~\ref{lem:paths}(a), with $W=Z$, $B=V(H_0)$ and $r=21a$, gives a double fan $\mathcal P_Z$ from $Z$ to $V(H_0)$ in $G$ with $\chi(G[V(\mathcal P_Z)])\leqslant4|Z|$.

\medskip\noindent\emph{Step 5. The next high-chromatic subgraph.} Let $G_2:=G-(Z\cup V(\mathcal P_Z)\cup V(H_0))$. Deleting $Z$, $V(H_0)$ and $V(\mathcal P_Z)$ lowers the chromatic number by at most $|Z|$, $\chi(H_0)$ and $4|Z|$, respectively. So
\[
\chi(G_2)\ \geqslant\ \chi(G)-5|Z|-12.5(a+\ell)-2\ =\ \chi(G)-(17.5a+22.5\ell+2)\ \geqslant\ \chi(G)-(92.5a+17),
\]
since $3\ell\leqslant10a+2$. Since $\chi(G)\geqslant1+897.5a$, the right side is at least $805a-16\geqslant\lceil3.0625(21a-1)\rceil$. By Corollary~\ref{cor:nguyen}, $G_2$ has a $21a$-connected induced subgraph $G_3$ with $\chi(G_3)\geqslant\chi(G_2)-42a$.

\medskip\noindent\emph{Step 6. Reduction to $G_3$.} We show that if $G_3$ is $(a,0)$-woven, then $G$ is $(a,\ell_0)$-woven. Since $G_1$ is $2a$-connected by Step~3, and $V(G_3)$ and $V(H_0)$ are subsets of $V(G_1)$, each with more than $2a$ vertices, Theorem~\ref{thm:menger} gives a family $\mathcal P_Y=\{Q_{i,j}:i\in[a],\ j\in\{1,2\}\}$ of $2a$ pairwise disjoint paths from $V(G_3)$ to $V(H_0)$ in $G_1$. Since $G_3\subseteq G_2=G-(Z\cup V(\mathcal P_Z)\cup V(H_0))$, the set $V(G_3)$ is disjoint from $V(H_0)\cup V(\mathcal P_Z)$. The paths of $\mathcal P_Y$ lie in $G_1=G-Z$. So Lemma~\ref{lem:paths}(c) applies, with $\mathcal P_Z$ and $\mathcal P_Y$ in place of $\mathcal P$ and $\mathcal Q$, and with $W=Z$, $B=V(H_0)$, $U=V(G_3)$ and $m=a$. It gives $|Z|+a=2(a+\ell)$ pairwise disjoint paths $P_1,\dots,P_{2(a+\ell)}$ in $G$, each with its last vertex in $V(H_0)$. For $i\leqslant|Z|$, the path $P_i$ starts at $z_i$. For $i\leqslant a$, the path $P_{|Z|+i}$ starts at a vertex $v_{|Z|+i}$ of $V(G_3)$. Let $u_i\in V(H_0)$ be the last vertex of $P_i$. Each $P_i$ meets $Z\cup V(G_3)\cup V(H_0)$ only in its ends. So each $P_i$ meets $V(H_0)$ only in $u_i$, and meets $Z\cup V(G_3)$ only in its first vertex.

Since $G_3$ is $(a,0)$-woven, it has a $K_a$ model $\{M'_1,\dots,M'_a\}$ with $v_{|Z|+i}\in M'_i$. Since $H_0$ is $(a+\ell)$-linked, it contains pairwise disjoint paths $Q_1,\dots,Q_{a+\ell}$ with $Q_i$ joining $u_i$ to $u_{a+\ell+i}$. For $i\leqslant a$, let
\[
M_i:=V(P_{\ell+i})\cup V(Q_{\ell+i})\cup V(P_{|Z|+i})\cup M'_i .
\]
Here $Q_{\ell+i}$ joins $u_{\ell+i}$ to $u_{a+2\ell+i}=u_{|Z|+i}$. The paths $P_1,\dots,P_{2(a+\ell)}$ are pairwise disjoint, and so are the paths $Q_1,\dots,Q_{a+\ell}$. Each $Q_j$ lies in $H_0$, and each $P_k$ meets $V(H_0)$ only in $u_k$. Each vertex $u_k$ is an end of exactly one path $Q_j$. So each $Q_j$ meets $V(P_1)\cup\dots\cup V(P_{2(a+\ell)})$ only in its ends. Also $V(G_3)\cap V(H_0)=\emptyset$, and each $P_k$ meets $Z\cup V(G_3)$ only in its first vertex. Hence $G[M_i]$ is connected and contains $z_{\ell+i}=r_i$, and the sets $M_i$ are pairwise disjoint. Any two of them are adjacent, because the sets $M'_i$ are. So $\{M_i\}$ is a $K_a$ model rooted at $R$. For $i\leqslant\ell$, the path $P_i$ joins $s_i$ to $u_i$, the path $Q_i$ joins $u_i$ to $u_{a+\ell+i}$, and the path $P_{a+\ell+i}$ joins $z_{a+\ell+i}=t_i$ to $u_{a+\ell+i}$. Hence $P_i\cup Q_i\cup P_{a+\ell+i}$ is an $s_i$--$t_i$ path. These paths are pairwise disjoint and disjoint from every $M_j$. Hence $G$ is $(a,\ell_0)$-woven. It remains to show that $G_3$ is $(a,0)$-woven.

\medskip\noindent\emph{Step 7a. Using Lemma~\ref{lem:kawa} in $G_3$.} If $G_3$ has a $K_{2a}$ minor, then, since $\kappa(G_3)\geqslant21a\geqslant a$, Lemma~\ref{lem:kawa} shows that $G_3$ is $(a,0)$-woven, and we are done by Step~6. So we may assume that $\had(G_3)<2a$.

\medskip\noindent\emph{Step 7b. Three disjoint high-chromatic subgraphs.} To show that $G_3$ is $(a,0)$-woven, fix an arbitrary set $R'=\{r'_1,\dots,r'_a\}$ of $a$ vertices of $G_3$. The graphs $G_4$, $J_1,J_2,J_3$ and $J'_1,J'_2,J'_3$ defined below depend on $R'$. Since $\delta(G_3)\geqslant21a$, we can choose pairwise distinct vertices $s'_1,\dots,s'_{2a}\in V(G_3)\setminus R'$ such that $s'_i$ and $s'_{a+i}$ are neighbours of $r'_i$. Let $S':=\{s'_1,\dots,s'_{2a}\}$ and $G_4:=G_3-(R'\cup S')$. Then $\chi(G_4)\geqslant\chi(G_3)-3a$, $G_4$ is an induced subgraph of $G$, and $\had(G_4)\leqslant\had(G_3)<2a$ by Step~7a. Let $\sigma:=2a\Phi$. By Step~2, $2a>a>a_0$. Let $\sigma':=\frac12(\sigma+\lfloor\sigma\rfloor+1)$, so that $\sigma<\sigma'<\lfloor\sigma\rfloor+1$. If $G_4$ were $\sigma'$-chromatic-inseparable, then $(\mathcal H_{\Phi,a_0})$, with $\sigma'$, $2a$ and $G_4$ in place of $s$, $b$ and $H$, would give $\sigma'\leqslant2a\Phi=\sigma$, which is false. So $G_4$ has vertex-disjoint induced subgraphs $F_1,F_2$ with $\chi(F_i)\geqslant\chi(G_4)-\sigma'>\chi(G_4)-\lfloor\sigma\rfloor-1$. Since chromatic numbers are integers, $\chi(F_i)\geqslant\chi(G_4)-\lfloor\sigma\rfloor\geqslant\chi(G_4)-\sigma$. The same argument applied to $F_1$ gives vertex-disjoint induced subgraphs $F_{1,1},F_{1,2}$ of $F_1$ with $\chi(F_{1,j})\geqslant\chi(F_1)-\sigma$. Let $J_1:=F_{1,1}$, $J_2:=F_{1,2}$ and $J_3:=F_2$. Then $\chi(J_i)\geqslant\chi(G_4)-2\sigma$ for $i\in\{1,2\}$, and $\chi(J_3)\geqslant\chi(G_4)-\sigma$. By Steps~3--5 and the bounds above,
\[
\chi(J_i)\ \geqslant\ \chi(G)-(92.5a+17)-42a-3a-2\sigma\ =\ \chi(G)-(137.5a+17)-4\Phi a
\]
for each $i\in\{1,2,3\}$, since $\chi(J_3)\geqslant\chi(G_4)-\sigma\geqslant\chi(G_4)-2\sigma$. Since $\chi(G)\geqslant1+\bigl(897.5+10\Phi\bigr)a$, the right side is at least $700a\geqslant\lceil3.0625(21a-1)\rceil$. In particular, these subgraphs are non-null. By Corollary~\ref{cor:nguyen}, each $J_i$ contains a $21a$-connected induced subgraph $J'_i$ with $\chi(J'_i)\geqslant\chi(J_i)-42a$.

\medskip\noindent\emph{Step 8. Accounting.} Table~\ref{tab:losses} lists the losses in chromatic number from $G$ to $J'_i$, for both schemes. Collecting the losses for the uneven scheme, with $\sigma=2\Phi a$, we obtain
\[
\chi(J'_i)\ \geqslant\ \chi(G)-L_i,\qquad\text{where}\quad L_i:=\begin{cases}179.5a+17+4\Phi a&\text{if }i\in\{1,2\},\\[2pt]179.5a+17+2\Phi a&\text{if }i=3.\end{cases}
\]
Let $a'_1=a'_2:=\lceil3a/5\rceil$ and $a'_3:=\lceil4a/5\rceil$. Then $a'_i<a$ (as $a\geqslant10$), $a-a'_1=\lfloor2a/5\rfloor\geqslant(2a-4)/5$ and $a-a'_3=\lfloor a/5\rfloor\geqslant(a-4)/5$. Since $a>a_0\geqslant4+5/\eta$ by Step~2, we have $\eta(a-4)\geqslant5$. Hence $(1+\eta)(a-4)\geqslant a+1$ and $(1+\eta)(a-2)\geqslant(1+\eta)(a-4)+2\geqslant a+3$. Since $\Lambda={\LamU}=5(1+\eta)\bigl(179.5+2\Phi\bigr)$,
\[
\Lambda(a-a'_3)\ \geqslant\ (1+\eta)(a-4)\bigl(179.5+2\Phi\bigr)\ \geqslant\ (a+1)\bigl(179.5+2\Phi\bigr)\ \geqslant\ L_3
\]
and
\[
\Lambda(a-a'_1)\ \geqslant\ 2(1+\eta)(a-2)\bigl(179.5+2\Phi\bigr)\ \geqslant\ 2(a+3)\bigl(179.5+2\Phi\bigr)\ \geqslant\ L_1=L_2 .
\]
Therefore $\chi(J'_i)\geqslant D(\frac{46}3a_0+4)+1+\Lambda a'_i$ for each $i\in\{1,2,3\}$. Since $J'_i$ is $21a$-connected, it is $21a'_i$-connected. Also $J'_i$ is an induced subgraph of $G$, so it satisfies $(\mathcal H_{\Phi,a_0})$. So by induction each $J'_i$ is $(a'_i,\lceil10a'_i/3\rceil)$-woven.

\medskip\noindent\emph{Step 9. Weaving.} Let $p:=\lfloor2a/5\rfloor$, and partition $[a]=I_{12}\cup I_{13}\cup I_{23}$ with $|I_{13}|=|I_{23}|=p$ and $|I_{12}|=a-2p$. Then
\[
|I_{12}|+|I_{13}|\ =\ |I_{12}|+|I_{23}|\ =\ a-p\ =\ a'_1\ =\ a'_2\qquad\text{and}\qquad|I_{13}|+|I_{23}|\ =\ 2p\ \leqslant\ a'_3 .
\]
For $j\in\{1,2,3\}$, let $X_j$ be a set of $a'_j$ vertices of $J'_j$, which exists since $v(J'_j)>21a$. For each $i\in I_{jk}$ with $j<k$, choose terminals $x_i\in X_j$ and $x'_i\in X_k$, all distinct. This is possible, since the numbers of terminals needed in $X_1$, $X_2$ and $X_3$ are $|I_{12}|+|I_{13}|$, $|I_{12}|+|I_{23}|$ and $|I_{13}|+|I_{23}|$, which are at most $a'_1$, $a'_2$ and $a'_3$ by the display above.

Since $G_3$ is $21a$-connected and $|R'|=a$, the graph $G_3-R'$ is $20a$-connected. So $G_3-R'$ has minimum degree at least $20a$, and hence at least $10a\,v(G_3-R')$ edges. Hence $G_3-R'$ is $2a$-linked by Theorem~\ref{thm:TW} with $\ell=2a$. The $4a$ vertices $s'_1,\dots,s'_{2a}$ and $x_i,x'_i$ $(i\in[a])$ are distinct, since $S'\cap V(J'_j)=\emptyset$ for each $j\in\{1,2,3\}$, as $J'_j\subseteq G_4$. So $G_3-R'$ contains a linkage $\mathcal P^{(0)}$ consisting of a path from $s'_i$ to $x_i$ and a path from $s'_{a+i}$ to $x'_i$, for each $i\in[a]$. For $j=1,2,3$ in turn, apply Lemma~\ref{lem:515} in the graph $G_3$ with $J'_j$, $a'_j$, $X_j$ and $\mathcal P^{(j-1)}$ in place of $H$, $a$, $R$ and $\mathcal P$. Since $J'_j$ is an induced subgraph of $G_3$, we have $G_3[V(J'_j)]=J'_j$, which is $(a'_j,\lceil10a'_j/3\rceil)$-woven by Step~8. Since $a'_j\geqslant3a/5$, we have $\lceil10a'_j/3\rceil\geqslant2a$, which is the number of pairs of $\mathcal P^{(j-1)}$. So Lemma~\ref{lem:515} applies, with $c=\lceil10a'_j/3\rceil$. We obtain a linkage $\mathcal P^{(j)}$ for the same pairs and a $K_{a'_j}$ model $\mathcal M_j$ in $J'_j$ rooted at $X_j$, with
\[
V(\mathcal P^{(j)})\cap V(\mathcal M_j)\subseteq X_j\qquad\text{and}\qquad V(\mathcal P^{(j)})\subseteq V(J'_j)\cup V(\mathcal P^{(j-1)}).
\]
Let $\mathcal P'':=\mathcal P^{(3)}$. For $j<j'$, induction on $j'$ gives $V(\mathcal P^{(j')})\subseteq V(\mathcal P^{(j)})\cup V(J'_{j+1})\cup\dots\cup V(J'_{j'})$, and the graphs $J'_{j+1},\dots,J'_{j'}$ are disjoint from $\mathcal M_j$. Hence $V(\mathcal P'')\cap V(\mathcal M_j)\subseteq X_j$ for each $j\in\{1,2,3\}$. The case $j=0$ gives $V(\mathcal P'')\subseteq V(\mathcal P^{(0)})\cup V(J'_1)\cup V(J'_2)\cup V(J'_3)$. The linkage $\mathcal P^{(0)}$ lies in $G_3-R'$, and each $J'_j$ is a subgraph of $G_4=G_3-(R'\cup S')$. So $r'_i\notin V(\mathcal P'')$ for each $i\in[a]$.

For $i\in[a]$, let $P''_i$ and $P''_{a+i}$ be the paths of $\mathcal P''$ from $s'_i$ to $x_i$ and from $s'_{a+i}$ to $x'_i$. For $i\in I_{jk}$, let $C_i$ be the branch set of $\mathcal M_j$ containing $x_i$, let $C'_i$ be the branch set of $\mathcal M_k$ containing $x'_i$, and let $B_i:=\{r'_i\}\cup V(P''_i)\cup V(P''_{a+i})\cup C_i\cup C'_i$. Each $G[B_i]$ is connected, because $r'_i$ is adjacent to $s'_i$ and $s'_{a+i}$. The sets $B_i$ are pairwise disjoint since:
\begin{itemize}[nosep]
\item The paths of $\mathcal P''$ are pairwise disjoint.
\item The models $\mathcal M_j$ lie in the disjoint graphs $J'_j$. Distinct vertices of $X_j$ lie in distinct branch sets of $\mathcal M_j$, since $\mathcal M_j$ is rooted at $X_j$.
\item The paths meet the models only in vertices of the sets $X_j$. Each branch set of $\mathcal M_j$ contains exactly one vertex of $X_j$. So, for each $i'\in[a]$, every path of $\mathcal P''$ meets $C_{i'}$ only in $x_{i'}$, and meets $C'_{i'}$ only in $x'_{i'}$. Since the paths of $\mathcal P''$ are pairwise disjoint, $P''_{i'}$ is the only path of $\mathcal P''$ that meets $C_{i'}$, and $P''_{a+i'}$ is the only path of $\mathcal P''$ that meets $C'_{i'}$.
\item $r'_i\notin V(\mathcal P'')$, and $r'_i\notin V(J'_1)\cup V(J'_2)\cup V(J'_3)$, since each $J'_j$ is a subgraph of $G_4=G_3-(R'\cup S')$.
\end{itemize}
Any two of the pairs $\{1,2\},\{1,3\},\{2,3\}$ intersect. So any two of the sets $B_i$ contain distinct branch sets of a common model $\mathcal M_j$, and are therefore adjacent. Hence $\{B_i\}$ is a $K_a$ model rooted at $R'$, and $G_3$ is $(a,0)$-woven. This completes the proof for the uneven scheme.

\medskip\noindent\emph{The even scheme.} Suppose now that $\Lambda={\LamE}$. The proof is the same, with $\ell_0:=3a$ and $h:=14a$, so that $\ell\leqslant3a$, and with the following changes in the numbers. In Step~1, $|Z|\leqslant7a$, so $2(a+2\ell)\leqslant14a\leqslant21a$. In Step~2, if $a\leqslant a_0$ then $\chi(G)\leqslant D(14a)\leqslant D(\frac{46}3a_0+4)$, and so $a>a_0\geqslant4+5/\eta\geqslant9$ and $a\geqslant10$. In Step~3, $\chi(G_1)\geqslant\chi(G)-7a\geqslant50a+1$, since $\chi(G)\geqslant1+\Lambda a\geqslant1+517.5a$. The hub $H_0$ is $(a+\ell)$-linked, where $a+\ell\leqslant4a$, with $v(H_0)>10(a+\ell)$ and $\chi(H_0)\leqslant12.5(a+\ell)+2\leqslant50a+2$. In Step~4, $2|Z|\leqslant14a$, and the colouring bound of Lemma~\ref{lem:paths}(a) is $4|Z|\leqslant28a$. In Step~5, $\chi(G_2)\geqslant\chi(G)-(85a+2)\geqslant432.5a-1\geqslant\lceil3.0625(21a-1)\rceil$. In Step~7b, $\sigma=2a\Phi$ as before, and each of $J_1,J_2,J_3$ satisfies
\[
\chi(J_i)\ \geqslant\ \chi(G_4)-2\sigma\ \geqslant\ \chi(G)-(130a+2)-4\Phi a\ \geqslant\ 300a\ \geqslant\ \lceil3.0625(21a-1)\rceil .
\]
In Step~8, $\chi(J'_i)\geqslant\chi(G)-(172a+2)-4\Phi a$ for each $i\in\{1,2,3\}$, and we take $a'_1=a'_2=a'_3:=\lceil2a/3\rceil$. Then $a-a'_i\geqslant(a-2)/3$. As in the uneven scheme, $(1+\eta)(a-2)\geqslant a+3$, and so
\[
{\LamE}(a-a'_i)\ \geqslant\ (1+\eta)(a-2)\bigl(172.5+4\Phi\bigr)\ \geqslant\ (a+3)\bigl(172.5+4\Phi\bigr)\ \geqslant\ 172a+2+4\Phi a .
\]
Hence $\chi(J'_i)\geqslant D(\frac{46}3a_0+4)+1+{\LamE}a'_i$, and by induction each $J'_i$ is $(a'_i,3a'_i)$-woven. In Step~9, let $I_{12},I_{13},I_{23}$ have sizes $(n,n,n)$, $(n+1,n,n)$ or $(n+1,n+1,n)$ according as $a=3n$, $3n+1$ or $3n+2$. Then each of the three sums $|I_{12}|+|I_{13}|$, $|I_{12}|+|I_{23}|$ and $|I_{13}|+|I_{23}|$ is at most $\lceil2a/3\rceil$. Lemma~\ref{lem:515} applies to each $J'_j$ with $c=3a'_j$, since $J'_j$ is $(a'_j,3a'_j)$-woven and $3a'_j\geqslant2a$. The rest of the proof is unchanged.
\end{proof}

\paragraph{Comparison with earlier proofs.} The proof of Lemma~\ref{lem:weave} follows the proof of a theorem of Delcourt and Postle~\cite[Theorem~7.1]{DP}, and its adaptation by Norin and Steiner~\cite[Lemma~B.12]{NS}. Hypothesis $(\mathcal H_{\Phi,a_0})$ replaces the bound of Delcourt and Postle on small subgraphs and the hypothesis of Norin and Steiner on chromatic inseparability. The main differences are as follows. First, we construct the hub by Lemma~\ref{lem:hub}, whereas Delcourt and Postle obtain it from their Theorem~2.3, and Norin and Steiner from~\cite[Theorem~B.1]{NS} and the theorem of Gir\~ao and Narayanan~\cite[Theorem~B.3]{NS}. Second, Step~7a applies Lemma~\ref{lem:kawa} to $G_3$, which Delcourt and Postle~\cite[Section~7.1]{DP} observe would suffice if only $(a,0)$-wovenness were needed. So Step~7b separates graphs with no $K_{2a}$ minor, whereas these proofs separate graphs with no $K_{14a}$ minor. Third, we use Corollary~\ref{cor:nguyen} in place of the theorem of Gir\~ao and Narayanan in Steps~5 and~7b. Fourth, the uneven scheme is new, and the even scheme is the distribution of Delcourt and Postle. Finally, Step~6 groups the paths into pairs $\{Q_{i,1},Q_{i,2}\}$, since Lemma~\ref{lem:53} requires two paths at each start vertex.

\begin{table}[ht]
\centering\small
\begin{tabular}{@{}lll@{}}\toprule
Operation & Uneven scheme & Even scheme\\\midrule
Delete $Z$ and the paths $\mathcal P_Z$ (Steps~3 and~5) & $5|Z|=5a+10\ell$ & $35a$\\
Delete the hub $H_0$ (Steps~3 and~5) & $12.5(a+\ell)+2$ & $50a+2$\\
\quad Subtotal (Step~5) & $17.5a+22.5\ell+2\leqslant92.5a+17$ & $85a+2$\\
Extract $G_3$ (Step~5) & $42a$ & $42a$\\
Delete $R'\cup S'$ (Step~7b) & $3a$ & $3a$\\
Chromatic separations (Step~7b) & $2\sigma$, $2\sigma$, $\sigma$ for $J_1,J_2,J_3$ & $2\sigma$ for each $J_i$\\
Extract $J'_i$ (Step~7b) & $42a$ & $42a$\\\midrule
Total loss (Step~8) & $179.5a+17+2\sigma$ ($i=1,2$) & $172a+2+2\sigma$\\
 & $179.5a+17+\sigma$ ($i=3$) & \\
where & $\sigma=2a\Phi$ & $\sigma=2a\Phi$\\
Sizes $a'_1,a'_2,a'_3$ & $\lceil3a/5\rceil,\lceil3a/5\rceil,\lceil4a/5\rceil$ & $\lceil2a/3\rceil$ each\\\bottomrule
\end{tabular}
\caption{Upper bounds on the losses in chromatic number in the proof of Lemma~\ref{lem:weave}, where $\ell$ is the number of terminal pairs, with $\ell\leqslant\lceil10a/3\rceil\leqslant(10a+2)/3$ in the uneven scheme and $\ell\leqslant3a$ in the even scheme, and $|Z|=a+2\ell$.}\label{tab:losses}
\end{table}

Given the chromatic separations of Step~7b, the uneven scheme minimises the coefficient of $\Phi$. Suppose that $J'_1,J'_2,J'_3$ host fractions $f_1,f_2,f_3$ of the $a$ branch sets in Step~9. Each branch set uses two of these subgraphs, so $f_1+f_2+f_3=2$. Ignoring rounding, the coefficient of $\Phi$ in $\Lambda/(1+\eta)$ is $2\max\{2/(1-f_1),2/(1-f_2),1/(1-f_3)\}$, since $J'_1,J'_2$ are obtained by two separations of cost $2a\Phi$ and $J'_3$ by one. This is minimised, with value $10$, at $f_1=f_2=\frac35$ and $f_3=\frac45$. Delcourt and Postle use $f_i=\frac23$, which gives $12$.

\begin{theorem}\label{thm:sepexplicit}
Let $t\geqslant2$ be an integer, and let $\eta\in(0,1]$, $a_0\geqslant4+5/\eta$ and $\Phi\geqslant0$ be real. Let $\cS(\Phi):=42+\min\bigl\{517.5+12\Phi,\ 897.5+10\Phi\bigr\}$. Let $\Lambda=\min\{{\LamE},{\LamU}\}$ be as in Lemma~\ref{lem:weave}. If $G$ satisfies $\had(G)<t$ and $(\mathcal H_{\Phi,a_0})$, then
\[
\chi(G)\ <\ D\bigl(\tfrac{46}3a_0+4\bigr)+1+(42+\Lambda)\,t ,
\]
where $42+\Lambda=\cS(\Phi)+\eta\bigl(\cS(\Phi)-42\bigr)$.
\end{theorem}

Theorem~\ref{thm:sepexplicit} parallels a theorem of Norin and Steiner~\cite[Theorem~15 and Section~B.3]{NS}, and the derivation by Delcourt and Postle of their Theorem~1.6 from their Theorem~7.1~\cite{DP}.

\begin{proof}
Suppose for a contradiction that $\chi(G)$ is at least the right side of the displayed bound. Then $\chi(G)\geqslant42t+\Lambda t\geqslant\lceil3.0625(21t-1)\rceil$, since $D(\frac{46}3a_0+4)\geqslant0$ and $\Lambda\geqslant517.5$. By Corollary~\ref{cor:nguyen}, $G$ has a $21t$-connected induced subgraph $G'$ with $\chi(G')\geqslant\chi(G)-42t\geqslant D(\frac{46}3a_0+4)+1+\Lambda t$. The graph $G'$ satisfies $(\mathcal H_{\Phi,a_0})$, since it is an induced subgraph of $G$. By Lemma~\ref{lem:weave} with $a=t$, $G'$ is $(t,0)$-woven. Since $v(G')>21t\geqslant t$, we may take any $t$ vertices of $G'$ as roots, so $G'$ has a $K_t$ minor, contradicting $\had(G)<t$.
\end{proof}

%=====================================================================
\section{Colouring graphs of order \texorpdfstring{$O(a\log^2a)$}{O(a log² a)}}\label{sec:bootstrap}
%=====================================================================

This section proves Corollary~\ref{cor:allscales}, which states that, for every integer $a\geqslant3$, every graph with no $K_a$ minor and at most $N_a$ vertices is ${\cst}a$-colourable, where $N_a\in O(a\log^2a)$ is defined in Proposition~\ref{prop:explicitbootstrap}. It is used in Section~\ref{sec:proof}. As in Parts~(I) and~(II) of the proof of Norin and Steiner (Section~\ref{sec:NSoverview}), we start from an explicit base case (Section~\ref{sec:base}) and apply bootstrapping steps (Sections~\ref{sec:level} and~\ref{sec:schedule}). Unlike in Part~(II), the Gy\'arf\'as graphs are coloured with a fixed earlier result. So each step maps a constant $c$ to $2c$ plus a term that does not depend on $c$ (see the remark after Lemma~\ref{lem:level}). We call such a step \defn{additive}, since the cost $\cS(\Phi)$ of the chromatic-separable reduction is added to $2c$, rather than multiplied by the previous constant as in Part~(II). The base case (Lemma~\ref{lem:base}) is an explicit instance of the small-graph bound of Norin and Steiner (Section~\ref{sec:NSres}), with base constant $4.01$ and exponent $0.49$, derived from a quantitative theorem of Norin and Steiner (Theorem~\ref{thm:NSquant}). Each bootstrapping step adds an error term $\xi=10^{-4}$ to the constant (see Lemma~\ref{lem:level}). Section~\ref{sec:small} treats the values of $a$ with $\log\log a<T_{\mathrm{boot}}$, where the threshold $T_{\mathrm{boot}}$ is defined in Proposition~\ref{prop:explicitbootstrap}.

\subsection{The base case}\label{sec:base}

This subsection makes the small-graph bound of Norin and Steiner (Section~\ref{sec:NSres}) explicit, for the constant $4.01$ and the exponent $0.49$ used below. We derive it from the following quantitative theorem. It is Theorem~2 of an earlier draft~\cite{NSdraft} of the paper of Norin and Steiner~\cite{NS}, and it is formalised in \texttt{LinearHadwiger} (see Appendix~\ref{app:lean}). Let $T_{\mathrm{base}}:=1160$.

\begin{theorem}[Norin and Steiner~{\cite[Theorem~2]{NSdraft}}]\label{thm:NSquant}
For every graph $G$ and every real $\eps\in(0,1]$,
\[
\chi(G)\ \leqslant\ 4\had(G)+\eps\,v(G)+A_\eps,\qquad\text{where }A_\eps:=(100/\eps)^{2000/\eps^2}.
\]
\end{theorem}

The next lemma bounds the additive term $A_\eps$ for the value of $\eps$ used in Lemma~\ref{lem:base}.

\begin{lemma}\label{lem:Aeps}
Let $a$ be real with $y:=\log\log a\geqslant T_{\mathrm{base}}$, and let $\eps:=0.0099\,(\log a)^{-0.49}$. Then $0<\eps\leqslant1$ and $A_\eps\leqslant10^{-4}a$.
\end{lemma}

\begin{proof}
Since $\log a=e^y\geqslant1$, we have $0<\eps\leqslant0.0099$. Let $K:=2000/0.0099^2$ and $U:=e^{0.98y}\geqslant1$. Then $2000/\eps^2=KU$, $\log(100/\eps)=\log(100/0.0099)+0.49y$ and $\log a=e^y=Ue^{0.02y}$. Since $100/0.0099$ and $10^4$ are less than $2.7^{10}<e^{10}$, we have $\log(100/0.0099)<10$ and $\log10^4<10$. So $\log A_\eps=KU\log(100/\eps)<KU(10+0.49y)$. Since $\log10^4<10\leqslant10U$, it suffices to show that $KU(10+0.49y)+10U\leqslant Ue^{0.02y}$, that is,
\[
K(10+0.49y)+10\ \leqslant\ e^{0.02y} .
\]
Since $e^z\geqslant1+z$ for every real $z$, we have $e^{0.02y}=e^{23.2}e^{0.02y-23.2}\geqslant e^{23.2}(0.02y-22.2)$, where $0.02y-22.2\geqslant1$. Since $e>2.7182818283$ and $e^{0.2}\geqslant1+0.2+0.2^2/2=1.22$, we have $e^{23.2}=e^{23}e^{0.2}\geqslant2.7182818283^{23}\cdot1.22>1.188\cdot10^{10}$. So it suffices to show that $f(y):=1.188\cdot10^{10}(0.02y-22.2)-K(10+0.49y)-10$ is non-negative for $y\geqslant1160$. Since $1.188\cdot10^{10}$ exceeds both $24.5K$ and $578.4K+10$, the function $f$ is linear with positive slope $0.02\cdot1.188\cdot10^{10}-0.49K$, and $f(1160)=1.188\cdot10^{10}-578.4K-10>0$. So the inequality holds for every $y\geqslant T_{\mathrm{base}}$. Hence $\log A_\eps<\log a-\log10^4$, that is, $A_\eps<10^{-4}a$.
\end{proof}

\begin{lemma}\label{lem:base}
For every real $a$ with $\log\log a\geqslant T_{\mathrm{base}}$, every graph $J$ with $\had(J)<a$ and $v(J)\leqslant a(\log a)^{0.49}$ satisfies $\chi(J)\leqslant4.01\,a$.
\end{lemma}

\begin{proof}
Let $\eps:=0.0099\,(\log a)^{-0.49}$. By Lemma~\ref{lem:Aeps}, $0<\eps\leqslant1$ and $A_\eps\leqslant10^{-4}a$. Also $4\had(J)<4a$ and $\eps\,v(J)\leqslant0.0099\,a$. So Theorem~\ref{thm:NSquant} gives $\chi(J)\leqslant4a+0.0099\,a+10^{-4}a=4.01\,a$.
\end{proof}

With $\eps=0.0099\,(\log a)^{-0.49}$, the bound $A_\eps\leqslant10^{-4}a$ holds if and only if $\log\log a$ is at least about $1159.49$. For $\log\log a$ below about $1158.5$, no choice of $\eps$ makes $4a+\eps a(\log a)^{0.49}+A_\eps$ at most $4.01\,a$. So this argument cannot lower $T_{\mathrm{base}}$ by much.

\subsection{A bootstrapping step}\label{sec:level}

The following lemma is an explicit form of a theorem of Norin and Steiner~\cite[Theorem~5]{NS}. Hypothesis~(a) is the bound that the lemma extends from exponent $\beta$ to exponent $\beta+\zeta$, and hypothesis~(b) is the bound used for the Gy\'arf\'as graphs. The conditions $(\beta_{\mathrm G}-2\zeta)T\geqslant1$, $\zeta T\geqslant5$ and $\beta_{\mathrm G}\leqslant1$ ensure that the Gy\'arf\'as graphs are small enough for~(b), and the conditions $T\geqslant1000.9$ and $\Phi\leqslant10^4$ ensure that the error terms are at most $\xi t$. The proof needs only $T\geqslant100$. We state the lemma with $T\geqslant1000.9$ because this is the value in the formalisation, and every application of the lemma has $T\geqslant1160.9$.

\begin{lemma}\label{lem:level}
Let $\xi:=10^{-4}$. Let $\beta$ be real, and let $T\geqslant1000.9$, $\zeta>0$, $\beta_{\mathrm G}\leqslant1$, $c\geqslant0$ and $0\leqslant\Phi\leqslant10^4$ be real such that $(\beta_{\mathrm G}-2\zeta)T\geqslant1$ and $\zeta T\geqslant5$. Suppose that:
\begin{enumerate}[label=(\alph*),nosep]
\item For every real $t$ with $\log\log t\geqslant T$, every graph $G$ with $\had(G)<t$ and $v(G)\leqslant t(\log t)^{\beta}$ satisfies $\chi(G)\leqslant ct$.
\item For every real $a$ with $\log\log a\geqslant T-0.1$, every graph $J$ with $\had(J)<a$ and $v(J)\leqslant a(\log a)^{\beta_{\mathrm G}}$ satisfies $\chi(J)\leqslant\Phi a$.
\end{enumerate}
Then, for every real $t$ with $\log\log t\geqslant T$, every graph $G$ with $\had(G)<t$ and $v(G)\leqslant t(\log t)^{\beta+\zeta}$ satisfies $\chi(G)\leqslant(2c+\cS(\Phi)+\xi)t$. In particular, if $\Phi=4.01$, then $\chi(G)\leqslant(2c+{\Sst})t$, where ${\Sst}:=\cS(4.01)+\xi=42+517.5+12\cdot4.01+10^{-4}=607.6201$.
\end{lemma}

In the corresponding step of Norin and Steiner (Part~(II) in Section~\ref{sec:NSoverview}), which deduces $\mathrm{LH}(\frac43\beta)$ from $\mathrm{LH}(\beta)$, the Gy\'arf\'as graphs are coloured using the hypothesis $\mathrm{LH}(\beta)$ itself. That is, they are coloured with the constant $C_\beta$ obtained in the previous bootstrapping step. The constant $C_\beta$ then enters the hypothesis of the chromatic-separable reduction, and the new constant $4C_\beta+2C_{\mathrm{sep}}(1+C_\beta)$ of Part~(II) contains the product $C_{\mathrm{sep}}C_\beta$. In Lemma~\ref{lem:level} the Gy\'arf\'as graphs are coloured by hypothesis~(b), with a fixed constant $\Phi$ such as the constant $4.01$ of Lemma~\ref{lem:base}. So the term $\cS(\Phi)$ does not depend on $c$. The factor $2$ in $2c+\cS(\Phi)$ comes from colouring $G_1$ through its contraction $G'_1$ in the proof of Lemma~\ref{lem:level}.

\begin{proof}
We follow the proof of Theorem~5 by Norin and Steiner~\cite{NS}. Apart from the change explained above, we verify $(\mathcal H_{\Phi,a_0})$ with $a_0=t/\log t$, that is, for every $b\geqslant t/\log t$ and not only for $b\leqslant t$. We also use Theorem~\ref{thm:sepexplicit} in place of their Theorem~15, split $G$ as in~\cite[Lemma~12]{NSdraft}, and pass to real $t$ by rounding $t$ up.
Let $t$ be real with $y:=\log\log t\geqslant T$, let $G$ be a graph with $\had(G)<t$ and $v(G)\leqslant t(\log t)^{\beta+\zeta}$, let $p:=(\log t)^\zeta$ and let $k:=\lceil p\rceil$. Let $L:=\log t=e^y$ and $a_0:=t/L$. Since $T>0$ and $(\beta_{\mathrm G}-2\zeta)T\geqslant1$, we have $\beta_{\mathrm G}>2\zeta>0$. The condition $T\geqslant1000.9$ is used only through its consequence $y\geqslant100$, which gives $L\geqslant e^{100}\geqslant2^{100}\geqslant10^{30}$, and $11y\leqslant e^y=L$, since $e^y\geqslant y^2/2\geqslant11y$ for $y\geqslant22$. Since $t=e^L\geqslant L^3/6$ and $L^2\geqslant10^{60}$, we also have $t\geqslant a_0\geqslant L^2/6\geqslant10^{59}$. Since $\zeta y\geqslant\zeta T\geqslant5$, we have $p=e^{\zeta y}\geqslant e^5\geqslant32$. So $k\geqslant2$ and $k(k-1)<(p+1)p\leqslant\frac{33}{32}p^2$. Choose a maximal family $H_1,\dots,H_m$ of pairwise vertex-disjoint connected induced bipartite subgraphs of $G$, each with exactly $k$ vertices. Let $G_1:=G[V(H_1)\cup\dots\cup V(H_m)]$ and $G_2:=G-V(G_1)$. By maximality, $G_2$ has no connected induced bipartite subgraph with exactly $k$ vertices, so $\operatorname{bip}(G_2)<k$ (see Section~\ref{sec:defs}).
\begin{enumerate}[label=(\roman*),nosep]
\item \emph{Colouring $G_1$.} The sets $V(H_i)$ partition $V(G_1)$ into sets inducing connected subgraphs of chromatic number at most $2$. Let $G'_1$ be obtained from $G_1$ by contracting each $H_i$ to a vertex. Then $G'_1$ is a minor of $G$, so $\had(G'_1)<t$. Also $v(G'_1)=m\leqslant\lfloor v(G)/k\rfloor\leqslant v(G)/p\leqslant t(\log t)^\beta$. By~(a), $G'_1$ has a colouring $\psi$ with at most $ct$ colours. For each $i$, let $\phi_i$ be a proper $2$-colouring of $H_i$. For each $i$, let $z_i$ be the vertex of $G'_1$ obtained by contracting $V(H_i)$, and give each vertex $v\in V(H_i)$ the pair $(\psi(z_i),\phi_i(v))$. Adjacent vertices in the same $H_i$ differ in the second coordinate. If $v\in V(H_i)$ and $v'\in V(H_j)$ are adjacent with $i\ne j$, then $z_iz_j\in E(G'_1)$, so $\psi(z_i)\ne\psi(z_j)$, and the pairs of $v$ and $v'$ differ in the first coordinate. Hence $\chi(G_1)\leqslant2\chi(G'_1)\leqslant2ct$. The colouring is a special case of a lemma of Norin and Steiner~\cite[Lemma~13]{NS}, who credit the argument to Reed and Seymour~\cite{RS98} and to Steiner~\cite{Ste22}.
\item \emph{The Gy\'arf\'as graphs.} Let $H$ be an $s$-chromatic-inseparable induced subgraph of $G_2$ with $\had(H)<b$ for some real $b\geqslant a_0$, and let $b':=\min\{b,t\}$. Then $\had(H)<b'$, as $\had(H)\leqslant\had(G)<t$. Since $b'\geqslant a_0$ and $11y\leqslant L$, we have $\log b'\geqslant\log a_0=L-y\geqslant\frac{10}{11}L$. As $\log\frac{11}{10}<0.1$, it follows that $\log\log b'\geqslant y-0.1\geqslant T-0.1$. So~(b) applies with $b'$ in place of $a$. Since $y\geqslant T$ and $\beta_{\mathrm G}-2\zeta>0$, we have $L^{\beta_{\mathrm G}-2\zeta}=e^{(\beta_{\mathrm G}-2\zeta)y}\geqslant e^{(\beta_{\mathrm G}-2\zeta)T}\geqslant e>2.7$, that is, $L^{\beta_{\mathrm G}}\geqslant2.7\,p^2$. Since $0<\beta_{\mathrm G}\leqslant1$, we have $(\frac{10}{11})^{\beta_{\mathrm G}}\geqslant\frac{10}{11}$, and so
\[
(\log b')^{\beta_{\mathrm G}}\ \geqslant\ \bigl(\tfrac{10}{11}L\bigr)^{\beta_{\mathrm G}}\ \geqslant\ \tfrac{10}{11}\cdot2.7\,p^2\ >\ 2\cdot\tfrac{33}{32}\,p^2\ >\ 2k(k-1) .
\]
Consider Lemma~\ref{lem:NS17} with $b'$ and $\Phi b'\geqslant0$ in place of $b$ and $\Gamma$. Every Gy\'arf\'as graph $J$ in its hypothesis, that is, every graph $J$ with $\had(J)<b'$ and $v(J)\leqslant2b'k(k-1)$, satisfies $v(J)\leqslant b'(\log b')^{\beta_{\mathrm G}}$, and so $\chi(J)\leqslant\Phi b'$ by~(b) with $b'$ in place of $a$. Since $H$ is an induced subgraph of $G_2$, we have $\operatorname{bip}(H)\leqslant\operatorname{bip}(G_2)<k$. Also $H$ is $s$-chromatic-inseparable and $\had(H)<b'$. So Lemma~\ref{lem:NS17} gives $s<\Phi b'\leqslant\Phi b$. So $G_2$ satisfies $(\mathcal H_{\Phi,a_0})$.
\item \emph{The chromatic-separable reduction.} Apply Theorem~\ref{thm:sepexplicit} to $G_2$, with $t':=\lceil t\rceil$ in place of $t$, $a_0=t/\log t$ and $\eta:=5/(a_0-4)$. Then $t'\geqslant2$, $a_0=4+5/\eta$ and $0<\eta\leqslant10/a_0\leqslant10^{-58}$, as $a_0\geqslant10^{59}$. This gives $\chi(G_2)<D(\frac{46}3a_0+4)+1+(42+\Lambda)t'$, where $42+\Lambda=\cS(\Phi)+\eta c_0$ with $c_0:=\cS(\Phi)-42\geqslant0$. Since $t'\leqslant t+1$, it follows that $\chi(G_2)<\cS(\Phi)\,t+R$, where
\[
R\ :=\ D\bigl(\tfrac{46}3a_0+4\bigr)+1+\cS(\Phi)+\eta c_0(t+1) .
\]
Let $s_0:=\lfloor\frac{46}3a_0+4\rfloor\leqslant16a_0$. By Theorem~\ref{thm:kos}, $D(\frac{46}3a_0+4)+1\leqslant60\,s_0\sqrt{\log s_0}+2\leqslant960\,a_0\sqrt{\log(16t)}+2$. Since $\log(16t)=L+\log16\leqslant2L$, $\sqrt{2L}\leqslant10^{-12}L$ (as $L\geqslant10^{30}$) and $a_0L=t$, this is at most $960\cdot10^{-12}\,t+2\leqslant10^{-7}t$, as $t\geqslant10^{59}$. Since $0\leqslant\Phi\leqslant10^4$, we have $\cS(\Phi)\leqslant939.5+10\Phi<10^6$ and $c_0\leqslant897.5+10\Phi<10^6$. Hence, as $\eta\leqslant10^{-58}$ and $t\geqslant10^{59}$,
\[
R\ \leqslant\ 10^{-7}t+10^6+10^{-52}(t+1)\ \leqslant\ 10^{-7}t+2\cdot10^{-52}\,t\ \leqslant\ \xi t .
\]
So $\chi(G_2)<(\cS(\Phi)+\xi)t$.
\end{enumerate}
Hence $\chi(G)\leqslant\chi(G_1)+\chi(G_2)\leqslant(2c+\cS(\Phi)+\xi)t$.
\end{proof}

\subsection{The schedule}\label{sec:schedule}

We now combine Lemmas~\ref{lem:base} and~\ref{lem:level}.

\begin{proposition}\label{prop:explicitbootstrap}
Let ${\cst}:=45{,}024.9873$. Let $N_a:=24{\PL}M^2\,a\log^2a$, where ${\PL}=4.56\cdot10^{18}$ and $M=10^5$ are as in Theorem~\ref{thm:refinedlin} and Lemma~\ref{lem:Ck} below. Let $T_{\mathrm{boot}}:=1161$. For every real $a$ with $\log\log a\geqslant T_{\mathrm{boot}}$, every graph $J$ with $\had(J)<a$ and $v(J)\leqslant N_a$ satisfies $\chi(J)\leqslant {\cst} a$.
\end{proposition}

\begin{proof}
Every application of Lemma~\ref{lem:level} in this proof has $T\in\{T_{\mathrm{boot}}-0.1,T_{\mathrm{boot}}\}$, so $T\geqslant T_{\mathrm{boot}}-0.1\geqslant1000.9$, and $\Phi\in\{4.01,c_{\mathrm{side}}\}$, so $0\leqslant\Phi\leqslant10^4$.

Let $\beta_0:=0.49$, the exponent of Lemma~\ref{lem:base}, and let $\vartheta:=0.0005$. Let $\zeta:=\beta_0/2-\vartheta=0.2445$ and $\beta_j:=\beta_0+j\zeta$ for $j\geqslant0$, so that $\beta_1=0.7345$ and $\beta_5=1.7125$. Let $\zeta_{\mathrm f}:=\beta_1/2-1.1\vartheta=0.3667$ and $\beta_{\mathrm{fin}}:=\beta_5+\zeta_{\mathrm f}=4.25\beta_0-6.6\vartheta=2.0792$. In each application below, $(\zeta,\beta_{\mathrm G})$ is $(0.2445,0.49)$ or $(0.3667,0.7345)$. In both cases $\beta_{\mathrm G}\leqslant1$, and $\beta_{\mathrm G}-2\zeta$ is $2\vartheta=0.001$ or $2.2\vartheta=0.0011$, respectively. Since $T\geqslant1000.9$, it follows that $(\beta_{\mathrm G}-2\zeta)T\geqslant0.001\cdot1000.9>1$ and $\zeta T\geqslant0.2445\cdot1000.9>5$. Every constant $c$ below is positive. So the conditions of Lemma~\ref{lem:level} on $T$, $\zeta$, $\beta_{\mathrm G}$, $c$ and $\Phi$ hold in every application.

\emph{The side step.} Apply Lemma~\ref{lem:level} with $T=T_{\mathrm{boot}}-0.1=1160.9$, $\beta=\beta_{\mathrm G}=\beta_0$, $c=\Phi=4.01$ and $\zeta=0.2445$. Hypotheses~(a) and~(b) hold by Lemma~\ref{lem:base}, since $T-0.1=1160.8\geqslant T_{\mathrm{base}}$. Lemma~\ref{lem:level} gives the constant $c_{\mathrm{side}}:=2\cdot4.01+{\Sst}=615.6401$ for exponent $\beta_1=0.7345$. That is, for every real $a$ with $\log\log a\geqslant T_{\mathrm{boot}}-0.1$, every graph $J$ with $\had(J)<a$ and $v(J)\leqslant a(\log a)^{0.7345}$ satisfies $\chi(J)\leqslant c_{\mathrm{side}}a$. The constant $c_{\mathrm{side}}$ equals the constant $\bar c_1$ of the first main step below. The side step differs from the first main step only in its size threshold. It is needed because hypothesis~(b) of the final step requires the bound for every real $a$ with $\log\log a\geqslant T_{\mathrm{boot}}-0.1$, whereas the main steps give it only when $\log\log a\geqslant T_{\mathrm{boot}}$.

\emph{The main steps.} Let $\bar c_0:=4.01$ and $\bar c_j:=2\bar c_{j-1}+{\Sst}$ for $j\geqslant1$. For $j=1,\dots,5$ in turn, apply Lemma~\ref{lem:level} with $T=T_{\mathrm{boot}}$, $\beta=\beta_{j-1}$, $\beta_{\mathrm G}=\beta_0$, $c=\bar c_{j-1}$, $\Phi=4.01$ and $\zeta=0.2445$. Hypothesis~(a) holds by Lemma~\ref{lem:base} for $j=1$ and by the previous application for $j\geqslant2$, and~(b) holds by Lemma~\ref{lem:base}. This gives the constant $\bar c_5=2^5\cdot4.01+31\,{\Sst}=18{,}964.5431$ for exponent $\beta_5=1.7125$ and size threshold $T_{\mathrm{boot}}$.

\emph{The final step.} Apply Lemma~\ref{lem:level} with $T=T_{\mathrm{boot}}$, $\beta=\beta_5$, $\zeta=\zeta_{\mathrm f}$, $\beta_{\mathrm G}=\beta_1$, $c=\bar c_5$ and $\Phi=c_{\mathrm{side}}$. Hypothesis~(a) is the result of the fifth main step, and~(b) is the result of the side step, since $T-0.1=T_{\mathrm{boot}}-0.1$. Since $c_{\mathrm{side}}>190$, we have $\cS(c_{\mathrm{side}})=42+897.5+10c_{\mathrm{side}}=7095.901$. So the constant obtained is $2\bar c_5+\cS(c_{\mathrm{side}})+10^{-4}={\cst}$, for exponent $\beta_{\mathrm{fin}}=2.0792$. Hence every graph $J$ with $\had(J)<a$ and $v(J)\leqslant a(\log a)^{2.0792}$ satisfies $\chi(J)\leqslant {\cst} a$, whenever $\log\log a\geqslant T_{\mathrm{boot}}$.

Finally, $N_a\leqslant a(\log a)^{2.0792}$ whenever $(\log a)^{0.0792}\geqslant24{\PL}M^2$. Since $24{\PL}M^2\leqslant2.7^{79}<e^{79}$, this holds whenever $0.0792\log\log a\geqslant79$, and $0.0792\,T_{\mathrm{boot}}\geqslant79$. So the proposition follows.
\end{proof}

\begin{table}[htbp]
\centering\small
\begin{tabular}{@{}llllll@{}}\toprule
Step & Exponent & Gy\'arf\'as graphs & Scheme & Constant & Threshold\\\midrule
Base case & $0.49$ & -- & -- & $4.01$ & $T_{\mathrm{base}}=1160$\\
Side step & $0.49\to0.7345$ & base case & even & $c_{\mathrm{side}}=615.6401$ & $1160.9$\\
Main step 1 & $0.49\to0.7345$ & base case & even & $\bar c_1=615.6401$ & $T_{\mathrm{boot}}=1161$\\
Main step 2 & $0.7345\to0.979$ & base case & even & $\bar c_2=1838.9003$ & $T_{\mathrm{boot}}$\\
Main step 3 & $0.979\to1.2235$ & base case & even & $\bar c_3=4285.4207$ & $T_{\mathrm{boot}}$\\
Main step 4 & $1.2235\to1.468$ & base case & even & $\bar c_4=9178.4615$ & $T_{\mathrm{boot}}$\\
Main step 5 & $1.468\to1.7125$ & base case & even & $\bar c_5=18{,}964.5431$ & $T_{\mathrm{boot}}$\\
Final step & $1.7125\to2.0792$ & side step & uneven & ${\cst}=45{,}024.9873$ & $T_{\mathrm{boot}}$\\\bottomrule
\end{tabular}
\caption{The schedule in the proof of Proposition~\ref{prop:explicitbootstrap}.}\label{tab:schedule}
\end{table}

Table~\ref{tab:schedule} summarises the schedule. For each step, it lists the exponents before and after the step, the result used to colour the Gy\'arf\'as graphs, the scheme of Lemma~\ref{lem:weave} (which is even if and only if $\Phi\leqslant190$), the constant obtained, and the size threshold, which is a lower bound on $\log\log a$. The final exponent $\beta_{\mathrm{fin}}$ must exceed $2$, since the inequality $N_a\leqslant a(\log a)^{\beta_{\mathrm{fin}}}$ requires $(\log a)^{\beta_{\mathrm{fin}}-2}\geqslant24{\PL}M^2$, and $24{\PL}M^2>1$. Since $\beta_{\mathrm{fin}}=4.25\beta_0-6.6\vartheta$, this schedule requires $\beta_0>8/17$, that is, a base case with exponent $\frac12-\delta$ for some $\delta<1/34$. The value of $\vartheta$ balances two thresholds. A smaller $\vartheta$ raises the threshold on $T$ imposed by the condition $(\beta_{\mathrm G}-2\zeta)T\geqslant1$ of Lemma~\ref{lem:level}, which is $1/(2\vartheta)$ for the side and main steps and $1/(2.2\vartheta)$ for the final step. A larger $\vartheta$ lowers $\beta_{\mathrm{fin}}$, and so raises the threshold $\log(24{\PL}M^2)/(\beta_{\mathrm{fin}}-2)$. With the chosen $\vartheta$, these thresholds are at most $1000$ and less than $873.4$, respectively. (The proof of Proposition~\ref{prop:explicitbootstrap} uses the cruder bound $24{\PL}M^2<e^{79}$, which needs only $\log\log a\geqslant79/0.0792$, and $79/0.0792<998$.) Both are below the threshold $1000.9$ of Lemma~\ref{lem:level}, and hence below $T_{\mathrm{base}}$. So $T_{\mathrm{boot}}$ is determined by $T_{\mathrm{base}}$. The threshold $2$ for $\beta_{\mathrm{fin}}$ comes from the order bound $N_t$ in Proposition~\ref{prop:DP}, which is in $O(t\log^2t)$ because the pieces of Theorem~\ref{thm:refinedlin} have $O(t\log t)$ vertices and the core $S$ built in the at most $M\sqrt{\log t}+1$ stages of the proof of Proposition~\ref{prop:DP} has $O(M^2t\log^2t)$ vertices. With pieces of order $O(t\log^3t)$, as in the small-dense-subgraph theorem of Delcourt and Postle, the order bound would be in $O(t\log^4t)$, and more steps would be needed. If the Gy\'arf\'as graphs are coloured with Lemma~\ref{lem:base} at every step, then each step raises the exponent by less than $0.49/2$, since the condition $(\beta_{\mathrm G}-2\zeta)T\geqslant1$ of Lemma~\ref{lem:level} forces $2\zeta<\beta_{\mathrm G}=0.49$. So seven steps are needed, and they give the constant $2^7\cdot4.01+127\,{\Sst}=77{,}681.0327$. Colouring the Gy\'arf\'as graphs of the final step with the result of the side step allows a larger increment $\zeta_{\mathrm f}=0.3667$, so one doubling is saved, at the cost of the larger additive term $\cS(c_{\mathrm{side}})$. Appendix~\ref{app:further} discusses other schedules.

\subsection{Small values of \texorpdfstring{$a$}{a}}\label{sec:small}

When $\log\log a<T_{\mathrm{boot}}$, graphs with no $K_a$ minor and at most $N_a$ vertices are coloured by the following bound.

\begin{lemma}\label{lem:indep}
Let $N_a$ be as in Proposition~\ref{prop:explicitbootstrap}. For every integer $a\geqslant3$, every graph $J$ with no $K_a$ minor and $v(J)\leqslant N_a$ satisfies
\[
\chi(J)\ \leqslant\ (144.4+4\log\log a)\,a .
\]
\end{lemma}

\begin{proof}
This is the proof of Theorem~1.5 by Delcourt and Postle~\cite{DP}, with $N_a$ in place of their order bound and with Theorem~\ref{thm:RS} (through Lemma~\ref{lem:greedy}) in place of the theorem of Duchet and Meyniel.
We may assume that $J$ is non-null. Let $p:=2a-2$, so $p\geqslant a$ since $a\geqslant3$. Lemma~\ref{lem:greedy} gives $\chi(J)\leqslant\bigl(2+\max\{\log(v(J)/p),0\}\bigr)p$. Since $v(J)\leqslant N_a$ and $a/p\leqslant1$, we have $v(J)/p\leqslant N_a/p\leqslant24{\PL}M^2\log^2a$. Since $24{\PL}M^2\leqslant2.7^{70}<e^{70}$, we have $\log(24{\PL}M^2)\leqslant70$. Since also $\log\log a>0$, it follows that $\max\{\log(v(J)/p),0\}\leqslant70+2\log\log a$. Hence
\[
\chi(J)\ \leqslant\ (72+2\log\log a)\cdot(2a-2)\ <\ (144+4\log\log a)\,a\ \leqslant\ (144.4+4\log\log a)\,a .
\]
\end{proof}

Let $a\geqslant3$ be an integer with $\log\log a<T_{\mathrm{boot}}$. Then, by Lemma~\ref{lem:indep}, every graph $J$ with no $K_a$ minor and $v(J)\leqslant N_a$ satisfies $\chi(J)\leqslant(144.4+4T_{\mathrm{boot}})a=4788.4\,a\leqslant {\cst} a$. Together with Proposition~\ref{prop:explicitbootstrap}, we obtain the following.

\begin{corollary}\label{cor:allscales}
Let ${\cst}=45{,}024.9873$ and $N_a$ be as in Proposition~\ref{prop:explicitbootstrap}. For every integer $a\geqslant3$, every graph $J$ with no $K_a$ minor and $v(J)\leqslant N_a$ satisfies $\chi(J)\leqslant {\cst} a$.
\end{corollary}

%=====================================================================
\section{Rooted minors via massed pairs}\label{sec:wollan}
%=====================================================================

This section proves Lemma~\ref{lem:Ck}, which states that, for an integer $t$ with $\log t\geqslant1.6\cdot10^6$, every $9t$-connected graph with at least $70t$ times as many edges as vertices is $(2x,\lfloor4.004t\rfloor)$-woven, where $x=\lceil t/(M\sqrt{\log t})\rceil$ and $M=10^5$. It is used in Section~\ref{sec:insep}, where it shows that the small dense pieces are woven. Lemma~\ref{lem:W} reduces wovenness to rooted models, in massed pairs, of the disjoint union of a clique of order $2x$ and a matching, Lemma~\ref{lem:massed} provides these rooted models, and Lemma~\ref{lem:Ck} combines the two. Only the clique needs a density argument, and the matching is routed by the linkage theorem of Thomas and Wollan. Since $2x$ is tiny compared with $t$, the clique model lies in a small set of vertices, and deleting this set leaves enough connectivity for the linkage. Recall that $K_s+\ell K_2$ is the disjoint union of $K_s$ and a matching with $\ell$ edges. For a graph $G$ and $Y\subseteq V(G)$, let $\rho_G(Y)$ denote the number of edges of $G$ with at least one end in $Y$. Following Wollan~\cite{Wol08} (see also~\cite{TW05}), a pair $(G,X)$ with $X\subseteq V(G)$ is \defn{$\lambda$-massed} if:
\begin{enumerate}[label=(M\arabic*),nosep]
\item $\rho_G(V(G)\setminus X)>\lambda\,|V(G)\setminus X|$, and
\item every non-trivial separation $(A,B)$ of $G$ with $X\subseteq A$ and order at most $|X|-1$ satisfies $\rho_G(B\setminus A)\leqslant\lambda|B\setminus A|$.
\end{enumerate}
Here, as in Section~\ref{sec:defs} and in Wollan's paper, a separation $(A,B)$ is trivial if $A\subseteq B$ or $B\subseteq A$. If $B\subseteq A$, then $\rho_G(B\setminus A)=\rho_G(\emptyset)=0$. If $X\subseteq A\subseteq B$, then $(A,B)$ has order $|A|\geqslant|X|$. So condition~(M2) is the same when it is required for every separation $(A,B)$ of $G$ with $X\subseteq A$ and order at most $|X|-1$, as in the formalisation. We use (M2) in this last form.

Wollan~\cite{Wol08} proves extremal theorems for rooted minors by studying a minimal counterexample. We use the following consequence of his structural results, rather than his final numerical bound. Let $F$ be a graph. A graph $G$ is \defn{$F$-universal at} a set $Z\subseteq V(G)$ if, for every injection $\pi:Z\to V(F)$, the graph $G$ has a model of $F[\pi(Z)]$ rooted at $Z$ according to $\pi$. A graph $J$ is \defn{$F$-universal} if $v(J)\geqslant|V(F)|$ and $J$ is $F$-universal at every set of $|V(F)|$ vertices of $J$. Equivalently, $v(J)\geqslant|V(F)|$ and, for every set $Y\subseteq V(J)$ with $|Y|=|V(F)|$ and every bijection $\pi:Y\to V(F)$, the graph $J$ has a model of $F$ rooted at $Y$ according to $\pi$.

\begin{theorem}[Wollan~{\cite[Theorems~2.1--2.3]{Wol08}}]\label{thm:wollanmin}
Let $F$ be a graph, and let $\lambda$ be real with $1\leqslant|V(F)|\leqslant\lambda$. Suppose that every non-null graph $D$ with $v(D)\leqslant2\lambda+1$ and $\delta(D)\geqslant\lambda-|V(F)|$ has an $F$-universal induced subgraph. Then for every $\lambda$-massed pair $(G,X)$ with $|X|\leqslant|V(F)|$ and every injection $\pi:X\to V(F)$, the graph $G$ has a model of $F[\pi(X)]$ rooted at $X$ according to $\pi$. That is, $G$ is $F$-universal at $X$.
\end{theorem}

Wollan does not state Theorem~\ref{thm:wollanmin} in this form, but it is the first step of his proof of~\cite[Theorem~5.1]{Wol08}, and follows from his Theorems~2.2 and~2.3 as below (his Theorem~2.1 is used in their proofs). These theorems concern an \defn{$(F,\lambda)$-minimal} pair, that is, a $\lambda$-massed pair $(G,X)$ with $|X|\leqslant|V(F)|$ such that $G$ is not $F$-universal at $X$, with $v(G)$ as small as possible and, subject to this, with $\rho_G(V(G)\setminus X)$ as small as possible. Wollan's Theorems~2.1--2.3 are formalised in \texttt{LinearHadwiger}, following his proofs, and Theorem~\ref{thm:wollanmin} is derived from them in the formalisation of this paper.

\begin{proof}
Suppose that the theorem fails, and let $(G,X)$ be an $(F,\lambda)$-minimal pair. Then $|X|\geqslant3$, as Wollan assumes in his proof of~\cite[Theorem~2.2(2)]{Wol08}. Indeed, let $\pi:X\to V(F)$ be an injection. If $|X|\leqslant1$, or if $X=\{x_1,x_2\}$ with $\pi(x_1)\pi(x_2)\notin E(F)$, then the sets $\{x\}$ with $x\in X$ form a model of $F[\pi(X)]$ rooted at $X$ according to $\pi$. Now let $X=\{x_1,x_2\}$ with $\pi(x_1)\pi(x_2)\in E(F)$. If $x_1$ and $x_2$ lie in a common component of $G$, then splitting an $x_1$--$x_2$ path into two subpaths, one containing $x_1$ and the other containing $x_2$, gives such a model. Otherwise, let $C$ be the vertex set of the component of $G$ containing $x_1$. The separations $\bigl((V(G)\setminus C)\cup\{x_1\},C\bigr)$ and $\bigl(C\cup\{x_2\},V(G)\setminus C\bigr)$ of $G$ have order $1=|X|-1$, so by~(M2), $\rho_G(Y_i)\leqslant\lambda|Y_i|$ for $Y_1:=C\setminus\{x_1\}$ and $Y_2:=V(G)\setminus(C\cup\{x_2\})$. The sets $Y_1$ and $Y_2$ partition $V(G)\setminus X$, and no edge joins them. So $\rho_G(V(G)\setminus X)=\rho_G(Y_1)+\rho_G(Y_2)\leqslant\lambda|V(G)\setminus X|$, contrary to~(M1). Hence $|X|\leqslant2$ is impossible for an $(F,\lambda)$-minimal pair, and so $|X|\geqslant3$. No vertex $z\in V(G)\setminus X$ is isolated. Otherwise $(G-z,X)$ is $\lambda$-massed, since $\rho$ is unchanged, $|V(G)\setminus X|$ decreases, and a separation $(A,B)$ of $G-z$ as in~(M2) gives the separation $(A\cup\{z\},B)$ of $G$ with the same order and the same set $B\setminus A$. Since $G-z$ is not $F$-universal at $X$, this contradicts the minimality of $v(G)$. By~\cite[Theorem~2.2(2)]{Wol08}, some vertex $z\in V(G)\setminus X$ has $\deg_G(z)\leqslant2\lambda$. Let $D:=G[\{z\}\cup N_G(z)]$, so that $v(D)\leqslant2\lambda+1$. Let $w\in N_G(z)$. By~\cite[Theorem~2.2(1)]{Wol08}, the vertices $z$ and $w$ have at least $\lfloor\lambda\rfloor$ common neighbours if $w\notin X$, and at least $\lfloor\lambda\rfloor-t_w$ common neighbours if $w\in X$, where $t_w\leqslant|X|-1$ is the number of neighbours of $z$ in $X\setminus\{w\}$. In both cases they have at least $\lfloor\lambda\rfloor-(|X|-1)$ common neighbours. These common neighbours and $z$ lie in $N_D(w)$, so $\deg_D(w)\geqslant\lfloor\lambda\rfloor-|X|+2>\lambda-|X|$. Since $z$ is not isolated, it has such a neighbour $w$. As $z$ is adjacent to $w$ and to their common neighbours, $\deg_D(z)=\deg_G(z)>\lambda-|X|$. Hence $\delta(D)\geqslant\lambda-|V(F)|$, and by hypothesis $D$ has an $F$-universal induced subgraph. This is an $F$-universal subgraph of $G$, contrary to~\cite[Theorem~2.3]{Wol08}.
\end{proof}

Wollan's main theorem~\cite[Theorem~5.1]{Wol08} combines his structural results with a general argument that works for every graph $F$. Its consequence~\cite[Theorem~1.1]{Wol08}, which is the form used by Delcourt and Postle~\cite[Theorem~5.9]{DP}, states that every $|V(F)|$-connected graph $G$ with at least $(9c+26{,}833|V(F)|)\,v(G)$ edges has a rooted model of $F$ for every choice of roots, where $c\geqslant1$ is such that every graph with at least $c$ times as many edges as vertices has an $F$ minor. Lemmas~\ref{lem:massed} and~\ref{lem:Ck} replace this theorem for the graphs $F=K_s+\ell K_2$ that we need, where $s\leqslant2\cdot10^{-8}t$ and $\ell=\lfloor4.004t\rfloor$.

\begin{lemma}\label{lem:W}
Let $p\geqslant1$ and $\ell\geqslant0$ be integers, let $F:=K_p+\ell K_2$, and let $\lambda>1$ be real with $\lambda\geqslant p+2\ell$. Suppose that for every $\lambda$-massed pair $(G_0,X)$ with $|X|\leqslant p+2\ell$ and every injection $\pi:X\to V(F)$, the graph $G_0$ has a model of $F[\pi(X)]$ rooted at $X$ according to $\pi$. Then every graph $G$ with $e(G)\geqslant\lambda\,v(G)$ and $\kappa(G)\geqslant2p+2\ell$ is $(p,\ell)$-woven.
\end{lemma}

\begin{proof}
The proof combines Wollan's derivation of~\cite[Theorem~1.1]{Wol08}, which checks the massed conditions using connectivity, with the argument of Delcourt and Postle~\cite[Lemmas~5.11 and~5.13]{DP}. As illustrated in Figure~\ref{fig:W}, the latter uses the graph $K_p+\ell K_2$ and attaches each root $r_i$ that is a terminal to a neighbour $r'_i$.
Let $\ell'\in\{0,\dots,\ell\}$, and let $R=\{r_1,\dots,r_p\}$, $S=\{s_1,\dots,s_{\ell'}\}$ and $T=\{t_1,\dots,t_{\ell'}\}$ be as in the definition of $(p,\ell)$-woven. Let $I:=\{i\in[p]:r_i\in S\cup T\}$. Since $\delta(G)\geqslant\kappa(G)\geqslant2p+2\ell$ and $|R\cup S\cup T|\leqslant p+2\ell$, each $r_i$ has at least $p$ neighbours outside $R\cup S\cup T$. So, by Hall's theorem, there are pairwise distinct vertices $r'_1,\dots,r'_p\notin R\cup S\cup T$ with $r'_i$ adjacent to $r_i$ for each $i\in[p]$. For $i\in[p]$, let $x_i:=r'_i$ if $i\in I$, and $x_i:=r_i$ otherwise. Let $X:=\{x_1,\dots,x_p\}\cup S\cup T$. The vertices $x_1,\dots,x_p$ are distinct and lie outside $S\cup T$. Moreover $|X|\leqslant p+2\ell\leqslant\lambda$ and $|X|\leqslant\kappa(G)$.

We claim that $(G,X)$ is $\lambda$-massed. For~(M2), let $(A,B)$ be a separation of $G$ with $X\subseteq A$ and order at most $|X|-1<\kappa(G)$. Since $G$ is $\kappa(G)$-connected, one of $A\setminus B$ and $B\setminus A$ is empty. If $A\setminus B=\emptyset$, then $X\subseteq A\subseteq A\cap B$, contradicting $|A\cap B|<|X|$. So $B\setminus A=\emptyset$, and hence $\rho_G(B\setminus A)=0$. So~(M2) holds. For~(M1), note that every edge of $G$ not counted by $\rho_G(V(G)\setminus X)$ has both ends in $X$. So $e(G)\leqslant\rho_G(V(G)\setminus X)+\binom{|X|}2\leqslant\rho_G(V(G)\setminus X)+\binom{|X|+1}2$. Since $1\leqslant p\leqslant|X|\leqslant\lambda$ and $\lambda>1$, we have $|X|+1<2\lambda$, and hence $\binom{|X|+1}2=\frac{(|X|+1)|X|}2<\lambda|X|$. Since $e(G)\geqslant\lambda\,v(G)$,
\[
\rho_G(V(G)\setminus X)\ \geqslant\ e(G)-\tbinom{|X|+1}{2}\ >\ \lambda v(G)-\lambda|X|\ =\ \lambda|V(G)\setminus X| .
\]

Let $F$ have clique vertices $q_1,\dots,q_p$ and matching edges $a_1b_1,\dots,a_\ell b_\ell$. Define $\pi(x_i):=q_i$. If $s_j\ne t_j$, put $\pi(s_j):=a_j$ and $\pi(t_j):=b_j$. If $s_j=t_j$, put $\pi(s_j):=a_j$. Since the sets $\{s_j,t_j\}$ with $j\in[\ell']$ are pairwise disjoint, this defines an injection $\pi:X\to V(F)$. By hypothesis, $G$ has a model $\{B_f:f\in\pi(X)\}$ of $F[\pi(X)]$ rooted at $X$ according to $\pi$.

For each $j$ with $s_j\ne t_j$, the sets $B_{a_j}$ and $B_{b_j}$ induce connected subgraphs and are joined by an edge, so their union contains an $s_j$--$t_j$ path $P_j$. If $s_j=t_j$, let $P_j$ be the path with the single vertex $s_j$. The paths $P_j$ are pairwise disjoint and disjoint from every $B_{q_i}$. For $i\in[p]$, let $M_i:=B_{q_i}\cup\{r_i\}$. If $i\notin I$, then $r_i=x_i\in B_{q_i}$, so $M_i=B_{q_i}$. If $i\in I$, then $r_i\in S\cup T$ lies in some set $B_{a_j}$ or $B_{b_j}$, so not in any $B_{q_{i'}}$, and $G[M_i]$ is connected since $r_ir'_i\in E(G)$ and $r'_i\in B_{q_i}$. So $\mathcal M:=\{M_1,\dots,M_p\}$ is a $K_p$ model rooted at $R$, and $\mathcal P:=\{P_1,\dots,P_{\ell'}\}$ is an $\{(s_j,t_j)\}$-linkage. Let $w\in M_i\cap V(P_j)$ for some $i\in[p]$ and $j\in[\ell']$. Then $w\in B_{a_j}\cup B_{b_j}$, so $w\notin B_{q_i}$. Hence $w=r_i$ and $i\in I$. Conversely, if $i\in I$, then $r_i\in\{s_j,t_j\}$ for some $j\in[\ell']$, so $r_i\in M_i\cap V(P_j)$. Therefore $V(\mathcal M)\cap V(\mathcal P)=R\cap(S\cup T)$.
\end{proof}

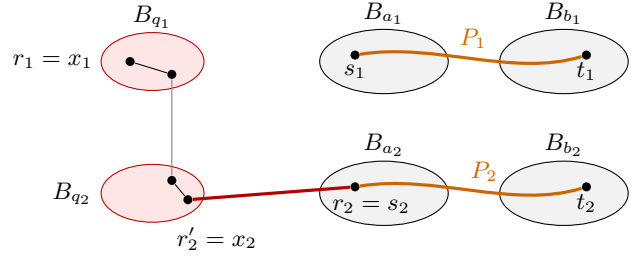
\begin{figure}[ht]
\centering
\begin{tikzpicture}[scale=0.85,
  vtx/.style={circle,fill,inner sep=1.2pt},
  lab/.style={font=\footnotesize},
  lpath/.style={orange!80!black,very thick}]
\draw[fill=red!10,draw=red!70!black] (1.0,3.1) ellipse (0.8 and 0.45);
\draw[fill=red!10,draw=red!70!black] (1.0,1.05) ellipse (0.8 and 0.45);
\node[lab] at (1.0,3.8) {$B_{q_1}$}; \node[lab] at (-0.25,1.05) {$B_{q_2}$};
\node[vtx] (x1) at (0.65,3.1) {}; \node[lab] at (-0.55,3.1) {$r_1=x_1$};
\node[vtx] (q1b) at (1.3,2.9) {}; \node[vtx] (q2a) at (1.3,1.25) {};
\node[vtx] (x2) at (1.55,0.95) {}; \node[lab] at (2.0,0.35) {$r'_2=x_2$};
\draw (x1)--(q1b); \draw (q2a)--(x2); \draw[gray] (q1b)--(q2a);
\foreach \y/\j in {3.1/1,1.05/2}{
  \draw[fill=gray!10] (4.6,\y) ellipse (1.0 and 0.5);
  \draw[fill=gray!10] (7.4,\y) ellipse (1.0 and 0.5);
  \node[lab] at (4.6,\y+0.75) {$B_{a_\j}$}; \node[lab] at (7.4,\y+0.75) {$B_{b_\j}$};
}
\node[vtx] (s1) at (4.15,3.2) {}; \node[lab] at (4.15,2.95) {$s_1$};
\node[vtx] (t1) at (7.75,3.2) {}; \node[lab] at (7.75,2.95) {$t_1$};
\node[vtx] (r2) at (4.15,1.15) {}; \node[lab] at (4.4,0.85) {$r_2=s_2$};
\node[vtx] (t2) at (7.75,1.15) {}; \node[lab] at (7.75,0.9) {$t_2$};
\draw[red!70!black,very thick] (x2)--(r2);
\draw[lpath] (s1) .. controls (5.4,3.4) and (6.6,2.85) .. node[lab,above,pos=0.5] {$P_1$} (t1);
\draw[lpath] (r2) .. controls (5.4,1.35) and (6.6,0.8) .. node[lab,above,pos=0.55] {$P_2$} (t2);
\end{tikzpicture}
\caption{The proof of Lemma~\ref{lem:W} with $p=\ell=2$. Since $r_2=s_2$, the root $r_2$ is replaced in $X$ by its neighbour $r'_2$. In the rooted model of $K_2+2K_2$, the path $P_j$ joins $s_j$ to $t_j$ inside $B_{a_j}\cup B_{b_j}$. The branch sets $B_{q_1}$ and $B_{q_2}\cup\{r_2\}$ form a $K_2$ model rooted at $\{r_1,r_2\}$ that meets the linkage only in $r_2$.}\label{fig:W}
\end{figure}

\paragraph{The weaving constant.} We call the constant ${\Ck}$ of Lemma~\ref{lem:Ck} below the \defn{weaving constant}. The other parameters of Lemma~\ref{lem:Ck} are chosen as follows. The number $\ell=\lfloor4.004t\rfloor$ is an upper bound on the number of linkage pairs in each application of Lemma~\ref{lem:515} in the proof of Proposition~\ref{prop:DP}. Delcourt and Postle~\cite[proof of Lemma~6.4]{DP} bound the number of pairs by $10t$. They also take $x=\lceil t/\sqrt{\log t}\rceil$, so that the density that Theorem~\ref{thm:kos} requires for a $K_{2x}$ minor alone is about $60t$. The factor $M=10^5$ makes $s=2x\leqslant2\cdot10^{-8}t$ when $\log t\geqslant1.6\cdot10^6$. Then the $K_s$ model in the proof of Lemma~\ref{lem:massed} lies in a set $Z$ of at most $0.008t$ vertices of a $(9t+1)$-connected graph $J$, and $J-Z$ is still $8.99t$-connected, which suffices for the linkage found there. The cost is that Proposition~\ref{prop:DP} has at most $M\sqrt{\log t}+1$ stages instead of about $\sqrt{\log t}$. This affects only the order bound $N_t=24{\PL}M^2\,t\log^2t$ and the size thresholds, not the final constant.

\begin{figure}[ht]
\centering
\begin{tikzpicture}[scale=0.9,
  vtx/.style={circle,fill,inner sep=1.3pt},
  lab/.style={font=\footnotesize},
  rpath/.style={red!75!black,thick,preaction={draw,white,line width=2.2pt}}]
\draw[fill=gray!5] (0,0) ellipse (5.2 and 2.7);
\node[lab] at (-4.1,2.25) {$J$};
\draw[dashed,fill=gray!14] (0,0.2) ellipse (1.5 and 1.05);
\node[lab] at (0,1.5) {$Z$};
\draw[fill=red!15,draw=red!75!black] (-0.55,0.2) ellipse (0.45 and 0.65);
\draw[fill=orange!20,draw=orange!80!black] (0.55,0.2) ellipse (0.45 and 0.65);
\draw (-0.1,0.2)--(0.1,0.2);
\node[lab] at (-0.6,0.45) {$B_1$};
\node[lab] at (0.55,0.45) {$B_2$};
\node[vtx,red!75!black] (z1) at (-0.65,-0.15) {}; \node[lab] at (-0.4,-0.15) {$z_1$};
\node[vtx,orange!80!black] (z2) at (0.65,-0.15) {}; \node[lab] at (0.4,-0.15) {$z_2$};
\node[vtx,red!75!black] (y1) at (-2.0,-1.0) {}; \node[lab,below] at (-2.0,-1.05) {$y_1$};
\node[vtx,orange!80!black] (y2) at (2.0,-1.0) {}; \node[lab,below] at (2.0,-1.05) {$y_2$};
\draw[red!75!black,thick] (z1)--(y1);
\draw[orange!80!black,thick] (z2)--(y2);
% roots in Y
\node[vtx,red!75!black] (r1) at (-4.0,0.6) {}; \node[lab,left] at (-4.05,0.6) {$r_1$};
\node[vtx,orange!80!black] (r2) at (4.0,0.6) {}; \node[lab,right] at (4.05,0.6) {$r_2$};
\node[vtx,blue!75!black] (u1) at (-3.6,-1.6) {}; \node[lab,left] at (-3.65,-1.6) {$u_1$};
\node[vtx,cyan!60!black] (v1) at (3.6,-1.6) {}; \node[lab,right] at (3.65,-1.6) {$u'_1$};
\draw[red!75!black,thick] (r1) .. controls (-3.2,1.5) and (-2.6,-0.2) .. (y1);
\node[lab,red!75!black] at (-3.35,1.55) {$P_1$};
\draw[orange!80!black,thick] (r2) .. controls (3.2,1.5) and (2.6,-0.2) .. (y2);
\node[lab,orange!80!black] at (3.35,1.55) {$P_2$};
\node[vtx,cyan!60!black] (w1) at (-2.9,-2.25) {};
\draw[gray,thick,dotted] (u1)--(w1);
\draw[cyan!60!black,thick] (w1) .. controls (-1.6,-2.6) and (2.4,-2.6) .. (v1);
\node[lab] at (0,-2.05) {$P'_1$};
\end{tikzpicture}
\caption{The proof of Lemma~\ref{lem:massed} with $s=2$ and $\ell=1$. A set $Z$ of at most $0.008t$ vertices of $J-Y$, chosen by averaging, contains a $K_s$ model $\{B_1,\dots,B_s\}$. A linkage in $J-Z$ joins each root $r_i$ by a path $P_i$ to a neighbour $y_i$ of a vertex $z_i\in B_i$, and joins $u_1$ to $u'_1$ by a path $P'_1$. The branch set for $a_1$ is $\{u_1\}$, and the branch set for $b_1$ is the rest of $P'_1$, joined to $u_1$ by the first edge of $P'_1$ (dotted).}\label{fig:massed}
\end{figure}
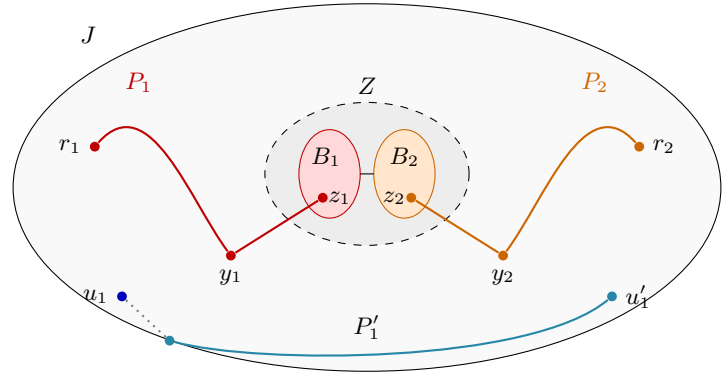

\begin{lemma}\label{lem:massed}
Let $t$ be an integer with $\log t\geqslant1.6\cdot10^6$. Let $s:=2x$, where $x=\lceil t/(M\sqrt{\log t})\rceil$, let $\ell:=\lfloor4.004t\rfloor$ and $F:=K_s+\ell K_2$. Then for every $70t$-massed pair $(G,X)$ with $|X|\leqslant s+2\ell$ and every injection $\pi:X\to V(F)$, the graph $G$ has a model of $F[\pi(X)]$ rooted at $X$ according to $\pi$.
\end{lemma}

The formalised form of Lemma~\ref{lem:massed} allows every real $t$ with $\log t\geqslant1.6\cdot10^6$, and so Lemma~\ref{lem:massed} is a special case of it.

Together with Theorem~\ref{thm:wollanmin}, the proof of Lemma~\ref{lem:massed} follows Wollan's proof of~\cite[Theorem~5.1]{Wol08} (in particular Observation~5.2 and Claim~5.3 there). Like Wollan, we take a minimal counterexample and find a dense subgraph in the neighbourhood of a vertex of low degree (both in the proof of Theorem~\ref{thm:wollanmin}), then find a minor inside a small set of vertices of this subgraph (Wollan chooses this set at random, and we choose it by averaging), and link the roots to this minor using the theorem of Thomas and Wollan. The differences are that we use Lemma~\ref{lem:mader} to obtain a highly connected dense subgraph $J$, and that only the clique part of $F$ needs a minor argument, while the matching edges are routed directly by the linkage.

\begin{proof}
Let $L:=\log t$, $h:=s+2\ell$, $\lambda:=70t$ and $\gamma_s:=30\,s\sqrt{\log s}$. Since $t=e^L\geqslant L^3/6$ and $L\geqslant1.6\cdot10^6$, we have $t\geqslant(1.6\cdot10^6)^2L/6\geqslant4\cdot10^{11}L\geqslant4\cdot10^{11}\sqrt L$ and $\sqrt L\geqslant1264$. In particular $t\geqslant10^{13}$. Since $t/(M\sqrt L)>1$, we have $x\geqslant2$ and $s=2x\geqslant4$. Also $x<t/(M\sqrt L)+1$, so
\[
s\sqrt L\ \leqslant\ \frac{2t}M+2\sqrt L .
\]
Hence $s\leqslant2t/(1264M)+2\leqslant2\cdot10^{-8}t$, since $2\leqslant4\cdot10^{-9}t$. In particular $s\leqslant t$, so $\log s\leqslant L$, and $\gamma_s\leqslant30s\sqrt L\leqslant60t/M+60\sqrt L\leqslant0.0006001t$, since $60\sqrt L\leqslant10^{-7}t$. Also $\ell\leqslant4.004t$, so $h\leqslant8.0081t$, $m:=s+\ell\leqslant4.005t$ and $2s+2\ell\leqslant9t$.

Since $|V(F)|=h\geqslant1$ and $\lambda\geqslant h$, Theorem~\ref{thm:wollanmin} reduces the lemma to the following claim. Every non-null graph $D$ with $v(D)\leqslant2\lambda+1$ and $\delta(D)\geqslant\lambda-h$ has an $F$-universal induced subgraph. Let $D$ be such a graph. Then $v(D)\leqslant140t+1\leqslant141t$ and $\delta(D)\geqslant70t-8.0081t\geqslant61.99t$, so $\gamma:=e(D)/v(D)\geqslant30.99t$. In particular, $\gamma\geqslant2k$ with $k=9t$, as Lemma~\ref{lem:mader} requires.

By Lemma~\ref{lem:mader} with $k=9t$, $D$ has a $(9t+1)$-connected induced subgraph $J$ with $\delta(J)>\gamma\geqslant30.99t>30.5t$ and $e(J)>(\gamma-k)\,v(J)\geqslant21.99t\,v(J)>21.5t\,v(J)$. Also $v(J)\leqslant v(D)\leqslant141t$ and $v(J)>\delta(J)>h$. We show that $J$ is $F$-universal.

Let $Y\subseteq V(J)$ with $|Y|=h$, and let $\pi:Y\to V(F)$ be a bijection. Let $K_s+\ell K_2$ have clique vertices $q_1,\dots,q_s$ and matching edges $a_1b_1,\dots,a_\ell b_\ell$. Let $r_i:=\pi^{-1}(q_i)$, $u_j:=\pi^{-1}(a_j)$ and $u'_j:=\pi^{-1}(b_j)$.

As illustrated in Figure~\ref{fig:massed}, we first find a small $K_s$ model in $J-Y$, and then use a linkage that avoids this model to join each root $r_i$ to it and to join $u_j$ to $u'_j$ for each $j\in[\ell]$.

\emph{A small clique model avoiding $Y$.} Let $J_0:=J-Y$ and $N:=v(J_0)$. Each vertex of $Y$ lies in at most $v(J)$ edges, so $e(J_0)\geqslant e(J)-h\,v(J)>(21.5t-8.0081t)\,v(J)\geqslant13.49t\,v(J)\geqslant13.49t\,N$. Also $N\leqslant v(J)\leqslant141t$, and $N=v(J)-h>\delta(J)-h>30.5t-8.0081t>22t$. We use the following averaging argument. Let $J'$ be a graph with $n\geqslant2$ vertices, and let $0<p\leqslant1$. Then some set $Z\subseteq V(J')$ satisfies $|Z|\leqslant pn+2$ and $e(J'[Z])\geqslant p^2e(J')$. Indeed, let $n_0:=\min\{n,\lceil pn\rceil+1\}$, so that $2\leqslant n_0\leqslant pn+2$. Each edge of $J'$ lies in $J'[Z]$ for exactly $\binom{n-2}{n_0-2}$ of the $\binom n{n_0}$ sets $Z$ of $n_0$ vertices, so the average of $e(J'[Z])$ over these sets is $\frac{n_0(n_0-1)}{n(n-1)}\,e(J')$. If $n_0=n$, this is $e(J')\geqslant p^2e(J')$. Otherwise $n_0=\lceil pn\rceil+1$, so $n_0-1\geqslant pn$ and $n_0(n_0-1)\geqslant(pn+1)pn\geqslant p^2n(n-1)$. In both cases some set $Z$ of $n_0$ vertices satisfies $e(J'[Z])\geqslant p^2e(J')$.

We apply this with $J'=J_0$, $n=N$, $\nu:=0.0079t-2$ and $p:=\nu/N$. Here $0<p\leqslant1$, since $0<\nu<22t<N$. This gives $Z\subseteq V(J_0)$ with $|Z|\leqslant\nu+2=0.0079t$ and
\begin{align*}
e(J_0[Z])\ &\geqslant\ \frac{\nu^2}{N^2}\,e(J_0)\ \geqslant\ \frac{13.49t\,\nu^2}N\ \geqslant\ \frac{13.49}{141}\,\nu^2\\
&\geqslant\ 0.0956\,(0.00789t)^2\ >\ 0.0006001t\cdot0.0079t\ \geqslant\ \gamma_s|Z| ,
\end{align*}
using $N\leqslant141t$, $\nu\geqslant0.00789t$ (as $t\geqslant10^{13}$), $\gamma_s\leqslant0.0006001t$ and $|Z|\leqslant0.0079t$. In particular $Z\ne\emptyset$. By Theorem~\ref{thm:kos}, $J_0[Z]$ has a $K_s$ model $\{B_1,\dots,B_s\}$. Also $Z\cap Y=\emptyset$ and $|Z|\leqslant0.008t$.

\emph{Linking around the clique model.} For each $i\in[s]$ choose $z_i\in B_i$ and a neighbour $y_i$ of $z_i$ in $J$, with $y_i\notin Z\cup Y$ and $y_1,\dots,y_s$ distinct. This is possible by Hall's theorem, since each $z_i$ has more than $30.5t$ neighbours in $J$, at most $|Z|+|Y|\leqslant8.0161t$ of them lie in $Z\cup Y$, and $30.5t-8.0161t\geqslant t\geqslant s$. The graph $J-Z$ is $(9t+1-|Z|)$-connected, and $9t+1-|Z|\geqslant8.99t\geqslant2m$. Also $e(J-Z)\geqslant e(J)-|Z|\,v(J)>(21.5t-0.008t)\,v(J)\geqslant5m\,v(J-Z)$. By Theorem~\ref{thm:TW}, $J-Z$ is $m$-linked. The $2m$ vertices $r_i,y_i$ ($i\in[s]$) and $u_j,u'_j$ ($j\in[\ell]$) are distinct and lie in $J-Z$. So $J-Z$ has pairwise disjoint paths $P_i$ from $r_i$ to $y_i$ and $P'_j$ from $u_j$ to $u'_j$. Let $B'_i:=B_i\cup V(P_i)$, which induces a connected subgraph since $y_iz_i\in E(J)$. Let $C_{q_i}:=B'_i$ for $i\in[s]$, and let $C_{a_j}:=\{u_j\}$ and $C_{b_j}:=V(P'_j)\setminus\{u_j\}$ for $j\in[\ell]$. Since $u_j\ne u'_j$, the set $C_{b_j}$ is the vertex set of the non-empty subpath $P'_j-u_j$, which contains $u'_j$. These sets are pairwise disjoint, because the paths $P_i$ and $P'_j$ are pairwise disjoint and avoid $Z$. The sets $C_{q_i}$ are pairwise adjacent because the sets $B_i$ are, and $C_{a_j}$ and $C_{b_j}$ are adjacent because the first edge of $P'_j$ joins them. So $\{C_f:f\in V(F)\}$ is a model of $F$ rooted at $Y$ according to $\pi$. Hence $J$ is an $F$-universal induced subgraph of $D$, which proves the claim.
\end{proof}

\begin{lemma}\label{lem:Ck}
Let ${\Ck}:=70$. If $t$ is an integer with $\log t\geqslant1.6\cdot10^6$, then every graph $G$ with $e(G)\geqslant {\Ck}t\,v(G)$ and $\kappa(G)\geqslant9t$ is $(2x,\lfloor4.004t\rfloor)$-woven, where $x=\lceil t/(M\sqrt{\log t})\rceil$ and $M=10^5$.
\end{lemma}

\begin{proof}
Let $s:=2x$ and $\ell:=\lfloor4.004t\rfloor$. By the proof of Lemma~\ref{lem:massed}, $s+2\ell\leqslant8.0081t\leqslant70t$ and $2s+2\ell\leqslant4\cdot10^{-8}t+8.008t\leqslant9t\leqslant\kappa(G)$. By Lemma~\ref{lem:massed}, the hypothesis of Lemma~\ref{lem:W} holds with $p=s$ and $\lambda=70t$. So $G$ is $(2x,\lfloor4.004t\rfloor)$-woven by Lemma~\ref{lem:W}.
\end{proof}

%=====================================================================
\section{A refined form of Lin's theorem}\label{sec:lin}
%=====================================================================

This section proves Theorem~\ref{thm:refinedlin}, which supplies the small dense pieces of the chromatic-inseparable case (Section~\ref{sec:insep}). That case needs small highly connected subgraphs in dense graphs with no $K_t$ minor. Delcourt and Postle use their small-dense-subgraph theorem (Section~\ref{sec:extremal}). Lin~\cite[Theorem~1.4]{Lin} proved a variant in which the highly connected subgraph has only $O(t\log t)$ vertices. Let ${\CLin}:=\max\{30{,}300,32\cdot6400^2\}$, which equals $1{,}310{,}720{,}000$. Lin's theorem states that, for all integers $t\geqslant3$ and $k\geqslant t$, every non-null graph $G$ with no $K_t$ minor and $e(G)\geqslant{\CLin}k\,v(G)$ has a $k$-connected subgraph with at most ${\CLin}^2\,t\log t$ vertices. Lin states the theorem with an unspecified constant, and ${\CLin}$ is the value given by Lin's proof~\cite{Lin}, where $6400={\CNP}$ is the constant of the Norin--Postle bound (Theorem~\ref{thm:NP}). We do not use Lin's theorem. The density constant ${\CLin}\approx1.31\cdot10^9$ is quadratic in the Norin--Postle constant, and much larger than the constant ${\CDP}\approx3.07\cdot10^6$ of the small-dense-subgraph theorem. In our application the density constant determines $\Cdens$ (Proposition~\ref{prop:DP}). The bound on the number of vertices of the subgraph enters only $N_t$, and hence only the size thresholds. We therefore rerun Lin's argument with parameters that make the density constant small, at the cost of a larger bound on the number of vertices. We also extract a subgraph that is dense, not only highly connected, since Lemma~\ref{lem:Ck} needs density ${\Ck}t$.

\begin{theorem}\label{thm:refinedlin}
Let $\cL:=0.0044$ and ${\PL}:=4.56\cdot10^{18}$. Let $t\geqslant10^{13}$ and $d\geqslant1.8\cdot10^4\,t$ be integers, and let $k$ be an integer with $1\leqslant k\leqslant\cL d/2$. Every non-null graph $G$ with no $K_t$ minor and $e(G)\geqslant d\,v(G)$ has a set $A$ of at most ${\PL}\,t\log t$ vertices such that $G[A]$ is $(k+1)$-connected and $e(G[A])>(\cL d-k)\,|A|$.
\end{theorem}

For a graph $J$ and $A\subseteq V(J)$, the \defn{weight} of $A$ in $J$ is $w_J(A):=\sum_{v\in A}\deg_J(v)$. The proof of Theorem~\ref{thm:refinedlin} contracts disjoint connected subgraphs of large weight while few edges are lost (Step~1), discards the vertices with many neighbours among the contracted vertices (Step~2), and finds a set $S'$ of large minimum degree among the remaining vertices (Step~3). A maximal connected set $U\subseteq S'$ of small weight whose contraction loses few edges then has a dense neighbourhood with few vertices (Steps~4--6), and Lemma~\ref{lem:mader} extracts the required set from it (Step~7).

The improvement over Lin's theorem has four sources. First, since $d\geqslant1.8\cdot10^4\,t$, the linear term $(t-2)(|V_1|+|V_2|)$ of the Norin--Postle bound (Theorem~\ref{thm:NP}) is at most $dn/(1.8\cdot10^4)$ in Claim~2 below, where $n=v(G)$. Second, each contracted subgraph has weight at least $Q$ in $G$, and the threshold $Q$ below is taken very large. This choice of $Q$ makes the other term of the Norin--Postle bound at most $3.17\cdot10^{-8}\,dn$ in Claim~2. The cost is a larger ${\PL}$, since the final set has fewer than $Q$ vertices. Third, the parameters $r$ and $\delta_0$ of Lemma~\ref{lem:core} are chosen so that the inequality in Step~3 holds with little room to spare, which makes the density obtained in Step~6 large. Fourth, in Steps~6 and~7 we extract a dense subgraph by a simple averaging argument, and then use the proof of Mader's theorem~\cite{Mad72} to obtain a highly connected subgraph whose density is close to that of the dense subgraph (Lemma~\ref{lem:mader}). Besides Theorems~\ref{thm:kos} and~\ref{thm:NP} and Lemma~\ref{lem:mader}, the proof uses one further tool.

\begin{lemma}[Delcourt and Postle~{\cite[Lemma~4.2]{DP}}, see also~{\cite[Lemma~2.2]{Lin}}]\label{lem:core}
Let $r>2$ and $\delta_0>0$ be real. Let $G$ be a graph and $S\subseteq V(G)$. If
\[
(r-2)\,e(G[S])\ >\ (r-1)\,\delta_0\,|S|+e_G(S,V(G)\setminus S),
\]
then there is a non-empty set $S'\subseteq S$ such that $\delta(G[S'])\geqslant\delta_0$ and $|N_G(v)\cap S'|\geqslant\deg_G(v)/r$ for every $v\in S'$.
\end{lemma}

\begin{proof}[Proof of Theorem~\ref{thm:refinedlin}]
The proof follows Lin~\cite[Section~3]{Lin}, which in turn follows the proof of a theorem of Delcourt and Postle~\cite[Theorem~2.3]{DP}. Claims~1--3 below are Lin's claims~\cite[Claims~3.1--3.3]{Lin} with our parameters. Let $n:=v(G)$. Since $t\geqslant10^{13}$, we have $d\geqslant1.8\cdot10^{17}$. Since $dn$ is an integer and $e(G)\geqslant dn$, deleting edges gives a spanning subgraph $G_0$ of $G$ with $e(G_0)=dn$, and $G_0$ has no $K_t$ minor. If a set $A$ has the required properties in $G_0$, then it has them in $G$, since $G_0[A]$ is a spanning subgraph of $G[A]$, adding edges preserves $(k+1)$-connectivity, and $e(G[A])\geqslant e(G_0[A])$. So we may assume that $e(G)=dn$. Then
\[
w_G(V(G))\ =\ 2e(G)\ =\ 2dn .
\]
Since $G$ has no $K_t$ minor and $e(G)\geqslant dn$, Theorem~\ref{thm:kos} gives $d\leqslant30\,t\sqrt{\log t}$. Define the parameters
\[
\eps:=0.083,\quad\theta:=0.003,\quad\eps':=0.02,\quad r:=2.92,\quad\delta_0:=0.37\,d,
\]
and let
\[
Q\ :=\ \frac{2\cdot10^{15}\,{\CNP}^2\,t^2\log t}{d} .
\]

\medskip\noindent\emph{Step 1. Contracting heavy subgraphs.} Choose pairwise vertex-disjoint non-null connected subgraphs $H_1,\dots,H_m$ of $G$ with the following two properties, with $m$ as large as possible. First, $w_G(V(H_i))\geqslant Q$ for each $i\in[m]$. Second, the graph $G'$ obtained by contracting each $H_i$ to a new vertex $x_i$ satisfies
\[
e(G)-e(G')\ \leqslant\ \eps\sum_{i=1}^mw_G(V(H_i)).
\]
The empty family has both properties, so such a family exists. The graph $G'$ is a minor of $G$, so it has no $K_t$ minor. Let $X:=\{x_1,\dots,x_m\}$ and $O:=V(G')\setminus X$. Since the sets $V(H_i)$ are disjoint, $|X|Q\leqslant w_G(V(G))=2dn$. Hence
\[
|X|\ \leqslant\ \frac{2dn}{Q}\ =\ \frac{2d^2n}{2\cdot10^{15}\,{\CNP}^2\,t^2\log t} .
\]
Also $e(G')\geqslant e(G)-2\eps\,dn=(1-2\eps)\,dn$.

\medskip\noindent\emph{Claim 1.} There is no non-empty set $A\subseteq O$ such that $G'[A]$ is connected, $w_{G'}(A)\geqslant Q$ and $e(G')-e(G'/A)\leqslant\eps\,w_{G'}(A)$. In particular, $\deg_{G'}(v)<Q$ for every $v\in O$.

\medskip\noindent\emph{Proof of Claim 1.} The contractions forming $G'$ do not change adjacencies within $O$. So $G'[O]=G[O]$, and $\deg_{G'}(v)\leqslant\deg_G(v)$ for $v\in O$. Suppose that $A$ exists. Then $w_G(A)\geqslant w_{G'}(A)\geqslant Q$. Moreover $G'/A$ is obtained from $G$ by contracting $H_1,\dots,H_m$ and $G[A]$. Since
\[
e(G)-e(G'/A)\ =\ \bigl(e(G)-e(G')\bigr)+\bigl(e(G')-e(G'/A)\bigr)\ \leqslant\ \eps\sum_{i=1}^mw_G(V(H_i))+\eps\,w_G(A),
\]
the family $H_1,\dots,H_m,G[A]$ contradicts the maximality of $m$. For the second statement, take $A=\{v\}$, for which $e(G')-e(G'/A)=0$. \hfill$\lozenge$

\medskip\noindent\emph{Step 2. Removing the vertices attached to $X$.} Let
\[
Y:=\bigl\{v\in O:|N_{G'}(v)\cap X|\geqslant\theta\deg_{G'}(v)\bigr\},
\]
and let $S:=V(G')\setminus(X\cup Y)$.

\medskip\noindent\emph{Claim 2.} $e(G'[X\cup Y])+e_{G'}(S,X\cup Y)\leqslant\eps'\,dn$.

\medskip\noindent\emph{Proof of Claim 2.} Apply Theorem~\ref{thm:NP} to $G'(X,O)$. Since $|O|\leqslant n$ and $|X|+|O|\leqslant n$,
\[
e_{G'}(X,O)\ \leqslant\ {\CNP}\,t\sqrt{\log t}\,\sqrt{|X|\,n}+tn\ \leqslant\ 3.17\cdot10^{-8}\,dn+\frac td\,dn\ \leqslant\ \nu'\,dn ,
\]
where $\nu':=3.17\cdot10^{-8}+1/(1.8\cdot10^4)<5.559\cdot10^{-5}$. Here the second inequality holds since $({\CNP}\,t\sqrt{\log t})^2\,|X|\,n\leqslant10^{-15}(dn)^2$, by the bound on $|X|$. By the definition of $Y$,
\[
\theta\sum_{v\in Y}\deg_{G'}(v)\ \leqslant\ e_{G'}(X,Y)\ \leqslant\ e_{G'}(X,O),
\]
so $\sum_{v\in Y}\deg_{G'}(v)\leqslant(\nu'/\theta)\,dn$. Theorem~\ref{thm:kos} and the bound on $|X|$ give
\[
e(G'[X])\ \leqslant\ 30\,t\sqrt{\log t}\,|X|\ \leqslant\ \frac{60}{2\cdot10^{15}\,{\CNP}^2}\cdot\frac{d}{t\sqrt{\log t}}\cdot dn\ <\ 10^{-19}\,dn ,
\]
using $d\leqslant30\,t\sqrt{\log t}$. We have
\begin{align*}
\sum_{v\in X\cup Y}\deg_{G'}(v)\ &=\ 2e(G'[X\cup Y])+e_{G'}(S,X\cup Y)\qquad\text{and}\\
\sum_{v\in X}\deg_{G'}(v)\ &=\ 2e(G'[X])+e_{G'}(X,O).
\end{align*}
Since $X$ and $Y$ are disjoint, the first sum equals the second sum plus $\sum_{v\in Y}\deg_{G'}(v)$. Hence
\begin{align*}
e(G'[X\cup Y])+e_{G'}(S,X\cup Y)\ &\leqslant\ 2e(G'[X])+e_{G'}(X,O)+\sum_{v\in Y}\deg_{G'}(v)\\
&\leqslant\ \bigl(2\cdot10^{-19}+\nu'(1+1/\theta)\bigr)dn\ <\ 0.01859\,dn\ <\ \eps'\,dn .
\end{align*}
This proves Claim~2. \hfill$\lozenge$

\medskip\noindent\emph{Step 3. A dense subset.} By Claim~2,
\[
e(G'[S])\ =\ e(G')-e(G'[X\cup Y])-e_{G'}(S,X\cup Y)\ \geqslant\ (1-2\eps-\eps')\,dn\ =\ 0.814\,dn .
\]
Since $|S|\leqslant n$, Claim~2 also gives
\begin{align*}
(r-1)\,\delta_0\,|S|+e_{G'}(S,V(G')\setminus S)\ &\leqslant\ (1.92\cdot0.37+\eps')\,dn\ <\ 0.7305\,dn\\
&<\ 0.92\cdot0.814\,dn\ \leqslant\ (r-2)\,e(G'[S]) .
\end{align*}
By Lemma~\ref{lem:core}, there is a non-empty set $S'\subseteq S$ such that $\delta(G'[S'])\geqslant\delta_0$ and $|N_{G'}(v)\cap S'|\geqslant\deg_{G'}(v)/r$ for every $v\in S'$.

\medskip\noindent\emph{Step 4. A comparison of $Q$ and $\delta_0$.} Since $\log x\leqslant x-1$ for $x>0$ and $\log10<3$,
\[
\log t\ =\ 12\log10+\log(t/10^{12})\ <\ 35+t/10^{12}\ \leqslant\ 9.7\cdot10^{-12}\,t ,
\]
where the last inequality holds because $t\geqslant10^{13}$. Since $d\geqslant1.8\cdot10^4\,t$, it follows that
\[
2\cdot10^{15}\,{\CNP}^2\,t^2\log t\ <\ 2\cdot10^{15}\,{\CNP}^2\cdot9.7\cdot10^{-12}\,t^3\ <\ 0.37^2\,(1.8\cdot10^4)^3\,t^3\ \leqslant\ 0.37^2\,d^3 .
\]
Dividing by $d$ gives $Q<\delta_0^2$.

\medskip\noindent\emph{Step 5. A maximal light set.} Choose a non-empty set $U\subseteq S'$ such that $G'[U]$ is connected, $w_{G'}(U)<Q$ and $e(G')-e(G'/U)\leqslant\eps\,w_{G'}(U)$, with $|U|$ as large as possible. Every singleton in $S'$ has these properties, by Claim~1. Some vertex of $S'\setminus U$ has a neighbour in $U$. Indeed, otherwise $N_{G'}(u)\cap S'\subseteq U$ for every $u\in U$. Since $\delta(G'[S'])\geqslant\delta_0$, it follows that $|U|\geqslant\delta_0$ and $\deg_{G'}(u)\geqslant\delta_0$ for every $u\in U$. So $w_{G'}(U)\geqslant\delta_0|U|\geqslant\delta_0^2>Q$ by Step~4, contrary to the choice of $U$.

\medskip\noindent\emph{Claim 3.} Let $z\in S'\setminus U$ have a neighbour in $U$. Then $e(G')-e(G'/(U\cup\{z\}))>\eps\,w_{G'}(U\cup\{z\})$.

\medskip\noindent\emph{Proof of Claim 3.} Suppose not. The graph $G'[U\cup\{z\}]$ is connected. If $w_{G'}(U\cup\{z\})<Q$, then $U\cup\{z\}$ contradicts the maximality of $|U|$. Otherwise $w_{G'}(U\cup\{z\})\geqslant Q$, and since $U\cup\{z\}\subseteq S'\subseteq O$, this contradicts Claim~1. \hfill$\lozenge$

\medskip\noindent\emph{Step 6. A dense neighbourhood.} Let $A\subseteq V(G')$ be non-empty with $G'[A]$ connected. Contracting $A$ deletes the edges of $G'[A]$. For each vertex $y\notin A$ with a neighbour in $A$, it replaces the $|N_{G'}(y)\cap A|$ edges between $y$ and $A$ by one edge. Hence
\begin{equation}\label{eq:linloss}
e(G')-e(G'/A)\ =\ e(G'[A])+\sum_{\substack{y\in V(G')\setminus A\\N_{G'}(y)\cap A\ne\emptyset}}\bigl(|N_{G'}(y)\cap A|-1\bigr).
\end{equation}
Let $R$ be the set of vertices of $V(G')\setminus(U\cup X)$ with a neighbour in $U$, and let $R':=R\cap S'$. Each vertex of $R$ is a neighbour of a vertex of $U$, so $|R|\leqslant w_{G'}(U)$. Since $S'\cap X=\emptyset$, the set $R'$ consists of the vertices of $S'\setminus U$ with a neighbour in $U$. So~\eqref{eq:linloss} with $A=U$ gives
\begin{align*}
\sum_{u\in U}|N_{G'}(u)\cap S'|\ &=\ 2e(G'[U])+|R'|+\sum_{y\in R'}\bigl(|N_{G'}(y)\cap U|-1\bigr)\\
&\leqslant\ 2\bigl(e(G')-e(G'/U)\bigr)+|R'|\ \leqslant\ 2\eps\,w_{G'}(U)+|R'|.
\end{align*}
By the choice of $S'$, the sum on the left is at least $w_{G'}(U)/r$. Hence
\[
|R'|\ \geqslant\ \Bigl(\frac1r-2\eps\Bigr)w_{G'}(U)\ \geqslant\ 0.1764\,w_{G'}(U)\ \geqslant\ 0.1764\,|R| .
\]
By Step~5, some vertex $z_0\in S'\setminus U$ has a neighbour in $U$, and $z_0\in R'$. So $R'$, and hence $R$, is non-empty.
Let $z\in R'$ be arbitrary, and let $C_z$ be the set of vertices of $V(G')\setminus(U\cup\{z\})$ that are adjacent to $z$ and have a neighbour in $U$. The graph $G'[U\cup\{z\}]$ is connected. Compare~\eqref{eq:linloss} for $A=U\cup\{z\}$ with~\eqref{eq:linloss} for $A=U$. The first term grows by $|N_{G'}(z)\cap U|$, the term $|N_{G'}(z)\cap U|-1$ of $z$ disappears, the term of each vertex of $C_z$ grows by $1$, and every other vertex contributes the same amount, or $0$, to both. Hence
\[
e(G')-e(G'/(U\cup\{z\}))\ =\ e(G')-e(G'/U)+1+|C_z| .
\]
By Claim~3 and the choice of $U$,
\[
1+|C_z|\ >\ \eps\,w_{G'}(U\cup\{z\})-\eps\,w_{G'}(U)\ =\ \eps\deg_{G'}(z).
\]
Since $z\in S$, we have $z\notin Y$, and so $|N_{G'}(z)\cap X|<\theta\deg_{G'}(z)$. Every vertex of $C_z$ lies in $N_{G'}(z)\cap R$ or in $N_{G'}(z)\cap X$. Hence $|N_{G'}(z)\cap R|>(\eps-\theta)\deg_{G'}(z)-1$. Since $\deg_{G'}(z)\geqslant\delta_0=0.37\,d$ and $d\geqslant1.8\cdot10^{17}$,
\[
|N_{G'}(z)\cap R|\ >\ 0.08\cdot0.37\,d-1\ \geqslant\ 0.02959\,d .
\]
This holds for every $z\in R'$. Each edge of $G'[R']$ contributes $2$ to $\sum_{z\in R'}|N_{G'}(z)\cap R|$, each edge of $G'[R]$ with exactly one end in $R'$ contributes $1$, and the other edges of $G'[R]$ contribute $0$. Since $R'\subseteq R$, it follows that
\[
0.02959\,d\,|R'|\ \leqslant\ \sum_{z\in R'}|N_{G'}(z)\cap R|\ \leqslant\ e(G'[R])+e(G'[R']) .
\]
Suppose that $e(G'[R])<\cL d\,|R|$ and $e(G'[R'])<\cL d\,|R'|$. Since $|R|\leqslant|R'|/0.1764$ and $R'$ is non-empty,
\[
e(G'[R])+e(G'[R'])\ <\ \cL d\Bigl(\frac1{0.1764}+1\Bigr)|R'|\ <\ 0.02935\,d\,|R'|\ <\ 0.02959\,d\,|R'| ,
\]
a contradiction. So one of $G'[R]$ and $G'[R']$ has at least $\cL d$ times as many edges as vertices. Let $R^\ast$ be $R$ or $R'$ accordingly, so that $R^\ast\subseteq R$ is non-empty and $e(G'[R^\ast])\geqslant\cL d\,|R^\ast|$.

\medskip\noindent\emph{Step 7. Conclusion.} Let $\gamma:=e(G'[R^\ast])/|R^\ast|$. By Step~6, $\gamma\geqslant\cL d\geqslant2k$. By Lemma~\ref{lem:mader}, $G'[R^\ast]$ has a $(k+1)$-connected induced subgraph $H$ with $e(H)>(\gamma-k)\,v(H)\geqslant(\cL d-k)\,v(H)$. Let $A:=V(H)$. Then $A\subseteq R\subseteq O$, and $G'[O]=G[O]$ (see the proof of Claim~1). So $G[A]=H$, and $A$ has the required connectivity and density. Finally, since $d\geqslant1.8\cdot10^4\,t$,
\[
|A|\ \leqslant\ |R|\ \leqslant\ w_{G'}(U)\ <\ Q\ \leqslant\ \frac{2\cdot10^{15}\,{\CNP}^2}{1.8\cdot10^4}\,t\log t\ <\ {\PL}\,t\log t .
\]
Here $2\cdot10^{15}\,{\CNP}^2/(1.8\cdot10^4)<4.552\cdot10^{18}$.
\end{proof}

Since $d\leqslant30\,t\sqrt{\log t}$ and $d\geqslant1.8\cdot10^4\,t$, no graph satisfies the hypotheses of Theorem~\ref{thm:refinedlin} unless $\log t\geqslant3.6\cdot10^{5}$. In Proposition~\ref{prop:DP}, Theorem~\ref{thm:refinedlin} is applied with $\log t\geqslant1.6\cdot10^6$.

%=====================================================================
\section{The chromatic-inseparable case}\label{sec:insep}
%=====================================================================

The following proposition is a version of a lemma of Delcourt and Postle~\cite[Lemma~2.5]{DP} in which the bound $\Gamma$ on the chromatic number of small subgraphs enters additively. It uses Lemma~\ref{lem:Ck} and Theorem~\ref{thm:refinedlin}.

\begin{proposition}\label{prop:DP}
Let $\cL=0.0044$ and ${\PL}=4.56\cdot10^{18}$ be the constants of Theorem~\ref{thm:refinedlin}, and let ${\Ck}=70$ and $M=10^5$ be the constants of Lemma~\ref{lem:Ck}. Let $\Cdens:=18{,}000$, so that $\cL\Cdens\geqslant {\Ck}+9$. Let $t$ be an integer with $\log t\geqslant1.6\cdot10^6$ (so in particular $t\geqslant10^{13}$). Let $N_t$ be as in Proposition~\ref{prop:explicitbootstrap}, with $t$ in place of $a$. Let $H$ be a graph with no $K_t$ minor, and let $\Gamma\geqslant0$ be a real number such that every induced subgraph of $H$ with at most $N_t$ vertices has chromatic number at most $\Gamma$. Let
\[
\sigma:=78t+2\Gamma\qquad\text{and}\qquad\Psi:=(32+2\Cdens)t+2\Gamma+2 .
\]
If $\chi(H)\geqslant\Psi$, then $H$ is $\sigma$-chromatic-separable. Consequently, if $H$ is $s$-chromatic-inseparable for some real $s\geqslant0$, then $s<(2\Cdens+33)\,t+2\Gamma$.
\end{proposition}

In the first conclusion, $\Cdens$ enters the chromatic threshold $\Psi$ but not the separation cost $\sigma$. The lemma of Delcourt and Postle~\cite[Lemma~2.5]{DP} has the shape of the first conclusion with $\Psi=2\sigma$, so their density constant appears in both. The second conclusion is the single-parameter form, which matches hypothesis $(\mathcal H_{\Phi,a_0})$. The final application (Section~\ref{sec:proof}) uses the first conclusion. The paragraph after the proof compares this proof with that of Delcourt and Postle.

The proof of Proposition~\ref{prop:DP} below builds a $K_t$ model in about $M\sqrt{\log t}$ stages, each adding $x$ branch sets, while keeping a $k$-connected induced subgraph $J^\ast$ of large chromatic number that is tangent to the model (the statement $(\ast)_n$ below). At each stage, $J^\ast$ contains many small dense pieces, since otherwise the chromatic number of $J^\ast$ would be at most about $2d+\Gamma$ (Step~1). By Lemma~\ref{lem:Ck}, these pieces are woven. Paths from the old branch sets, from rooted clique models in the pieces and from a new subgraph $H_3$ of large chromatic number are joined in one piece, the junction, by knitting (Steps~2--4, with Step~5 for the first stage). The branch sets are then shrunk by Lemma~\ref{lem:62}, so that the whole model has chromatic number at most $6t+\Gamma$ (Step~6). Since $H$ is $\sigma$-chromatic-inseparable, a $k$-connected subgraph of large chromatic number outside the model meets $H_3$ in at least $k$ vertices, and its union with $H_3$ restores the invariant (Step~7).

\begin{proof}
\emph{Setup.} Let $x:=\lceil t/(M\sqrt{\log t})\rceil$ and $r:=\lceil M\sqrt{\log t}\rceil$, so that $rx\geqslant t$. Let $k:=12t$ and $d:=\Cdens t$. Here $k$ is the connectivity of the high-chromatic subgraphs extracted by Corollary~\ref{cor:nguyen}, and $d$ is the density threshold. Theorem~\ref{thm:refinedlin} is applied to subgraphs with at least $d$ times as many edges as vertices. Let $\Theta:=6t+\Gamma+2k$. Since $k=12t$ and $d=\Cdens t$,
\[
\sigma=30t+2\Gamma+4k\qquad\text{and}\qquad\Psi=8t+2\Gamma+2k+2d+2 .
\]

For $1\leqslant n\leqslant r$ we have $n\leqslant M\sqrt{\log t}+1$, and $(M\sqrt{\log t}+1)^2\leqslant4M^2\log t$. Hence the numbers $n{\PL}\,t\log t$ and $6n^2{\PL}\,t\log t$ are at most $N_t$.

Each application of Corollary~\ref{cor:nguyen} below extracts a $k$-connected induced subgraph from a graph with chromatic number at least $\Psi-\Theta-24t-\Gamma=2d-22t+2$, as we verify at each application. This is at least $\lceil3.0625(k-1)\rceil$, because $d=1.8\cdot10^4\,t$ and $k=12t$. Each application loses at most $2k$ in chromatic number.

\emph{Reduction.} It suffices to prove the following claim:
\begin{quote}
($\dagger$) if $H$ is $\sigma$-chromatic-inseparable and $\chi(H)\geqslant\Psi$, then $H$ has a $K_t$ minor.
\end{quote}
Indeed, since $H$ has no $K_t$ minor, ($\dagger$) shows that if $\chi(H)\geqslant\Psi$ then $H$ is $\sigma$-chromatic-separable. The second conclusion follows from the first by the remark on chromatic thresholds in Section~\ref{sec:prelim}, since $\sigma=78t+2\Gamma\leqslant\Psi=(32+2\Cdens)t+2\Gamma+2\leqslant(2\Cdens+33)\,t+2\Gamma$.

\emph{Induction.} Assume that $H$ is $\sigma$-chromatic-inseparable, $\chi(H)\geqslant\Psi$, and $H$ has no $K_t$ minor. We prove by induction on $n\in\{0,1,\dots,r\}$ the following statement:
\begin{quote}
$(\ast)_n$: $H$ has a $K_{nx}$ model $\mathcal A=\{A_1,\dots,A_{nx}\}$ and a set $S\subseteq V(H)$ with the following properties:
\begin{enumerate}[label=(I\arabic*),nosep]
\item $S$ is a core of $\mathcal A$, and $|S|\leqslant2n^2{\PL}\,t\log t$.
\item There is a set $R\subseteq V(H)$ such that $J^\ast:=H[R]$ is $k$-connected.
\item $J^\ast$ is tangent to $\mathcal A$, that is, $|V(J^\ast)\cap A_j|=1$ for every $j\in[nx]$.
\item $V(J^\ast)\cap V(\mathcal A)\subseteq S$.
\item $\chi(J^\ast)\geqslant\chi(H)-\Theta$.
\end{enumerate}
\end{quote}
Since $rx\geqslant t$, the statement $(\ast)_r$ gives a $K_t$ minor of $H$, contradicting our assumption and so proving ($\dagger$). In the inductive step, the pieces and the next high-chromatic subgraph $H_3$ are found inside $J^\ast$. By~(I3), each old branch set contains exactly one vertex of $J^\ast$, and in Step~4 a path from this vertex extends the branch set.

\emph{The case $n=0$.} Since $\chi(H)\geqslant\Psi\geqslant\lceil3.0625(k-1)\rceil$, Corollary~\ref{cor:nguyen} shows that $H$ has a $k$-connected induced subgraph $J^\ast=H[R]$ with $\chi(J^\ast)\geqslant\chi(H)-2k\geqslant\chi(H)-\Theta$. Take $\mathcal A$ to be the empty model and $S:=\emptyset$. Then (I1)--(I5) hold, (I1), (I3) and (I4) vacuously.

\emph{The inductive step.} Let $1\leqslant n\leqslant r$ and assume $(\ast)_{n-1}$. Let $\mu':=(n-1)x$ and $\mu:=nx$. Then $\mu'\leqslant t+M\sqrt{\log t}\leqslant1.0001t$ and $\mu=\mu'+x\leqslant3t$. Moreover $r\leqslant M\sqrt{\log t}+1$ and $x\leqslant t/(M\sqrt{\log t})+1$, so
\begin{align*}
3\mu'+\mu\ =\ (4n-3)x\ &\leqslant\ (4M\sqrt{\log t}+1)\Bigl(\frac t{M\sqrt{\log t}}+1\Bigr)\\
&=\ 4t+\frac t{M\sqrt{\log t}}+4M\sqrt{\log t}+1\ \leqslant\ 4.004t ,
\end{align*}
since $\log t\geqslant1.6\cdot10^6$. As $3\mu'+\mu$ is an integer, $3\mu'+\mu\leqslant\lfloor4.004t\rfloor$. Let $\mathcal A'=\{A'_1,\dots,A'_{\mu'}\}$, $S'$ and $H':=J^\ast$ be given by $(\ast)_{n-1}$. By (I3), for $j\in[\mu']$ there is a unique vertex $w_j$ in $V(H')\cap A'_j$. Since the sets $A'_j$ are disjoint, the vertices $w_1,\dots,w_{\mu'}$ are distinct, and
\begin{equation}\label{eq:HprimeA}
V(H')\cap V(\mathcal A')=\{w_1,\dots,w_{\mu'}\}\subseteq S' .
\end{equation}
The inclusion holds by (I4). Every subgraph of $H'$ that avoids $\{w_1,\dots,w_{\mu'}\}$ is therefore disjoint from $V(\mathcal A')$.

\medskip\noindent\emph{Step 1. Small pieces.} Let $H_1:=H'-\{w_1,\dots,w_{\mu'}\}$, an induced subgraph of $H$. Then $\chi(H_1)\geqslant\chi(H')-\mu'\geqslant\chi(H)-\Theta-\mu'$. Choose a maximal family of pairwise disjoint sets $K\subseteq V(H_1)$ such that $|K|\leqslant {\PL}\,t\log t$, $H[K]$ is $(9t+1)$-connected and $e(H[K])>{\Ck}t\,|K|$.

Suppose for a contradiction that this family has fewer than $n$ members, and let $U$ be the union of its members. Then $|U|\leqslant n{\PL}\,t\log t\leqslant N_t$, so $\chi(H[U])\leqslant\Gamma$. Let $T\subseteq V(H_1)\setminus U$ be non-empty. Suppose that $e(H[T])\geqslant d\,|T|$. Theorem~\ref{thm:refinedlin}, applied to $H[T]$ with $9t$ in place of $k$, gives a set $A\subseteq T$ such that $|A|\leqslant{\PL}\,t\log t$, $H[A]$ is $(9t+1)$-connected and $e(H[A])>(\cL d-9t)\,|A|=70.2t\,|A|\geqslant{\Ck}t\,|A|$. Theorem~\ref{thm:refinedlin} applies because $H[T]$ has no $K_t$ minor, $t\geqslant10^{13}$, $d=\Cdens t\geqslant1.8\cdot10^4\,t$ and $9t\leqslant\cL d/2$. The set $A$ could be added to the family, contradicting maximality. So $e(H[T])<d\,|T|$, and $H[T]$ has a vertex of degree less than $2d$. Hence $H_1-U$ is $(2d-1)$-degenerate and $\chi(H_1-U)\leqslant2d$. Then $\chi(H_1)\leqslant2d+\Gamma$. But $\chi(H_1)\geqslant\Psi-\Theta-\mu'=2d+\Gamma+2t+2-\mu'>2d+\Gamma$, since $\mu'\leqslant1.0001t$. This is a contradiction.

So the family has at least $n$ members. Choose $n$ of them, $K_0,K_1,\dots,K_{n-1}$, and let $J_i:=H[K_i]$ for $0\leqslant i\leqslant n-1$. Call $J_0,J_1,\dots,J_{n-1}$ the \defn{pieces}, and call $J_0$ the \defn{junction}. Each piece $J$ is $(9t+1)$-connected, so $v(J)>9t+1$, and $e(J)>{\Ck}t\,v(J)$. By Lemma~\ref{lem:Ck}, $J$ is $(2x,\lfloor4.004t\rfloor)$-woven. By~\eqref{eq:HprimeA}, the pieces are disjoint from $V(\mathcal A')$.

\medskip\noindent\emph{Step 2. Paths to the junction $J_0$.} If $n=1$, let $W:=\emptyset$ and $\mathcal Q_1:=\emptyset$. Now let $n\geqslant2$. For $i\in[n-1]$, choose a set $X_i$ of $2x$ vertices of $J_i$ and label them $w_{\mu'+2(i-1)x+1},\dots,w_{\mu'+2ix}$. Since $2(n-1)x=2\mu'$, the labels run up to $3\mu'$. Let $W:=\{w_1,\dots,w_{3\mu'}\}$. Then $|W|=3\mu'\leqslant6t$. The vertices $w_1,\dots,w_{\mu'}$ lie in $H'$ but not in $H_1$, while each $X_i$ lies in $J_i\subseteq H_1$. So the $3\mu'$ labelled vertices are distinct, and $W\subseteq V(H')$ is disjoint from $V(J_0)$. The graph $H'$ is $k$-connected, $2|W|=6\mu'\leqslant6.0006t\leqslant12t=k$, and $v(J_0)>9t+1\geqslant2|W|$. Since $H'=H[R]$ for the set $R$ of~(I2), Lemma~\ref{lem:paths}(a) applied inside $R$, with $H$ in place of $G$ and $B=V(J_0)$ (see the end of its proof), gives a double fan $\mathcal Q_1$ from $W$ to $V(J_0)$ in $H$ with $V(\mathcal Q_1)\subseteq V(H')$ and $\chi(H[V(\mathcal Q_1)])\leqslant4|W|=12\mu'$. Since $H'=H[R]$ is an induced subgraph of $H$, $\mathcal Q_1$ is also a double fan in $H'$. For every $n\geqslant1$, let $\overline J:=V(J_0)\cup V(J_1)\cup\dots\cup V(J_{n-1})$. Then $|\overline J|\leqslant n{\PL}\,t\log t\leqslant N_t$, so $\chi(H[\overline J])\leqslant\Gamma$.

\medskip\noindent\emph{Step 3. The next high-chromatic subgraph.} Let $H_2:=H_1-(\overline J\cup V(\mathcal Q_1))$. So $H_2$ is obtained from $H'$ by removing the vertices of $\{w_1,\dots,w_{\mu'}\}$, of $\overline J$ and of $V(\mathcal Q_1)$. Their removal lowers the chromatic number by at most $\mu'$, $\Gamma$ and $4|W|=12\mu'$, respectively. Since $13\mu'\leqslant13(t+M\sqrt{\log t})\leqslant24t$, we obtain $\chi(H_2)\geqslant\chi(H')-24t-\Gamma$. By~(I5), $\chi(H')\geqslant\chi(H)-\Theta\geqslant\Psi-\Theta$. Hence $\chi(H_2)\geqslant\Psi-\Theta-24t-\Gamma$, and so Corollary~\ref{cor:nguyen} applies to $H_2$. So $H_2$ has a $k$-connected induced subgraph $H_3$ with
\[
\chi(H_3)\ \geqslant\ \chi(H)-\Theta-24t-\Gamma-2k\ =\ \chi(H)-\sigma .
\]
By construction, $H_3$ is disjoint from $\overline J$, from $V(\mathcal Q_1)$ and from $W\subseteq\{w_1,\dots,w_{\mu'}\}\cup V(\mathcal Q_1)$. By~\eqref{eq:HprimeA}, it is also disjoint from $V(\mathcal A')$.

\medskip\noindent\emph{Step 4. Building a $K_\mu$ model when $n\geqslant2$.}

\emph{A linkage to the junction.} The graph $H'-W$ is $2\mu$-connected, as $k-|W|\geqslant12t-6t\geqslant2\mu$. The sets $V(H_3)$ and $V(J_0)$ avoid $W$. They have more than $k\geqslant2\mu$ and more than $9t+1\geqslant2\mu$ vertices, respectively. So, by Theorem~\ref{thm:menger}, $H'-W$ has $2\mu$ pairwise disjoint paths from $V(H_3)$ to $V(J_0)$. Trimming them gives a family $\mathcal Q_2=\{Q_{i,j}:i\in[\mu],\ j\in\{1,2\}\}$ of $2\mu$ pairwise disjoint $V(H_3)$--$V(J_0)$ paths in $H'-W$. By Step~3, $V(H_3)$ is disjoint from $V(J_0)\cup V(\mathcal Q_1)$. So Lemma~\ref{lem:paths}(b), with $H'$, $\mathcal Q_1$ and $\mathcal Q_2$ in place of $G$, $\mathcal P$ and $\mathcal Q$, and with $B=V(J_0)$, $U=V(H_3)$ and $m=\mu$, gives $3\mu'+\mu$ pairwise disjoint paths $P_1,\dots,P_{3\mu'+\mu}$ in $H'$ with all their vertices in $V(\mathcal Q_1)\cup V(\mathcal Q_2)$. Here $P_i$ starts at $w_i$ for $i\leqslant3\mu'$, and $P_{3\mu'+i}$ starts at a vertex $w_{3\mu'+i}\in V(H_3)$ for $i\in[\mu]$. Let $u_i\in V(J_0)$ be the last vertex of $P_i$. Each $P_i$ meets $W\cup V(H_3)\cup V(J_0)$ only in its ends. The vertices $w_1,\dots,w_{3\mu'+\mu}$ and $u_1,\dots,u_{3\mu'+\mu}$ are $2(3\mu'+\mu)$ distinct vertices, since the paths $P_i$ are pairwise disjoint and $W\cup V(H_3)$ is disjoint from $V(J_0)$. So $\mathcal P=\{P_1,\dots,P_{3\mu'+\mu}\}$ is a $\{(w_i,u_i)\}$-linkage in $H'$. Also $V(\mathcal P)\subseteq V(H')$ and $V(H')\cap V(\mathcal A')=\{w_1,\dots,w_{\mu'}\}\subseteq W$. Hence $\mathcal P$ has the following three intersection properties:
\[
\begin{gathered}
V(\mathcal P)\cap V(H_3)=\{w_{3\mu'+1},\dots,w_{3\mu'+\mu}\},\qquad V(\mathcal P)\cap V(J_0)=\{u_1,\dots,u_{3\mu'+\mu}\},\\
V(\mathcal P)\cap V(\mathcal A')=\{w_1,\dots,w_{\mu'}\}.
\end{gathered}
\]

\emph{Rooted models in the pieces.} Each $J_i$ is $(2x,\lfloor4.004t\rfloor)$-woven, by Step~1, and hence $(2x,3\mu'+\mu)$-woven, since $\lfloor4.004t\rfloor\geqslant3\mu'+\mu$. As $3\mu'+\mu$ is the number of pairs of $\mathcal P$, Lemma~\ref{lem:515} applies with $c=3\mu'+\mu$. Let $\mathcal P^{(0)}:=\mathcal P$. For $i=1,\dots,n-1$ in turn, apply Lemma~\ref{lem:515} in the graph $H'$, with the linkage $\mathcal P^{(i-1)}$, the subgraph $J_i=H'[K_i]$ in place of $H$, and the root set $X_i$. This gives a $\{(w_l,u_l)\}$-linkage $\mathcal P^{(i)}$ in $H'$ and a $K_{2x}$ model $\mathcal M_i$ in $J_i$ rooted at $X_i$, with
\[
V(\mathcal P^{(i)})\cap V(\mathcal M_i)\subseteq X_i\qquad\text{and}\qquad V(\mathcal P^{(i)})\subseteq V(J_i)\cup V(\mathcal P^{(i-1)}).
\]
Let $\mathcal P':=\mathcal P^{(n-1)}$. For $i<i'$, the second inclusion gives $V(\mathcal P^{(i')})\subseteq V(\mathcal P^{(i)})\cup V(J_{i+1})\cup\dots\cup V(J_{i'})$, and the graphs $J_{i+1},\dots,J_{i'}$ are disjoint from $\mathcal M_i\subseteq J_i$. Hence $V(\mathcal P')\cap V(\mathcal M_i)\subseteq X_i$. Since every vertex of $X_i$ is the first vertex of a path of every linkage $\mathcal P^{(i')}$ with $0\leqslant i'\leqslant n-1$, in fact $V(\mathcal P')\cap V(\mathcal M_i)=X_i$. Also $V(\mathcal P')\subseteq V(\mathcal P)\cup V(J_1)\cup\dots\cup V(J_{n-1})$. The graphs $J_i$ are disjoint from $H_3$, from $V(\mathcal A')$ and from $J_0$. So the three displayed intersection properties of $\mathcal P$ also hold for $\mathcal P'$. For $i\in[n-1]$ and $w\in X_i$, let $B(w)$ denote the branch set of $\mathcal M_i$ containing $w$. For $l\in[3\mu'+\mu]$, let $P'_l$ denote the path of $\mathcal P'$ starting at $w_l$.

\emph{Indices.} For $i\in[n-1]$ and $j\in[x]$, let $q_{i,j}:=\mu'+2(i-1)x+x+j$, so that the vertices $w_{q_{i,j}}$ form the second half of $X_i$. For $j\in[\mu']$, let $i(j):=\lceil j/x\rceil$, let $p_j:=\mu'+2(i(j)-1)x+(j-(i(j)-1)x)$, and let $p'_j:=3\mu'+j$. The vertices $w_{p_j}$ form the first halves of the sets $X_i$. Define
\[
\Pi_j:=\{u_j,u_{p_j},u_{p'_j}\}\ \ (j\in[\mu']),\qquad \Pi_{\mu'+j}:=\{u_{q_{i,j}}:i\in[n-1]\}\cup\{u_{4\mu'+j}\}\ \ (j\in[x]).
\]
The index sets $\{j:j\in[\mu']\}$, $\{p_j:j\in[\mu']\}$, $\{p'_j:j\in[\mu']\}$, $\{q_{i,j}:i\in[n-1],\,j\in[x]\}$ and $\{4\mu'+j:j\in[x]\}$ partition $[3\mu'+\mu]$. So $\Pi_1,\dots,\Pi_\mu$ are pairwise disjoint and contain $3\mu'+\mu$ vertices in total. As illustrated in Figure~\ref{fig:step4}, each branch set $A''_l$ of the model $\mathcal A''=\{A''_1,\dots,A''_\mu\}$ defined below uses some of the paths of $\mathcal P'$ to reach the junction $J_0$. Table~\ref{tab:indices} lists these paths.

\begin{table}[ht]
\centering\small
\begin{tabular}{@{}llll@{}}\toprule
Branch set & Path & First vertex & Lies in\\\midrule
$A''_j$, $j\in[\mu']$ & $P'_j$ & $w_j$ & $W$, the old branch set $A'_j$\\
 & $P'_{p_j}$ & $w_{p_j}$ & first half of $X_{i(j)}$, in $J_{i(j)}$\\
 & $P'_{p'_j}$, $p'_j=3\mu'+j$ & $w_{3\mu'+j}$ & $H_3$\\\midrule
$A''_{\mu'+j}$, $j\in[x]$ & $P'_{q_{i,j}}$, $i\in[n-1]$ & $w_{q_{i,j}}$ & second half of $X_i$, in $J_i$\\
 & $P'_{4\mu'+j}$ & $w_{4\mu'+j}$ & $H_3$\\\bottomrule
\end{tabular}
\caption{The paths of $\mathcal P'$ used by each branch set of $\mathcal A''$ in Step~4 of the proof of Proposition~\ref{prop:DP}. The last vertex of each path lies in $\Pi_j$ or $\Pi_{\mu'+j}$, respectively, and hence in $R_j$ or $R_{\mu'+j}$.}\label{tab:indices}
\end{table}

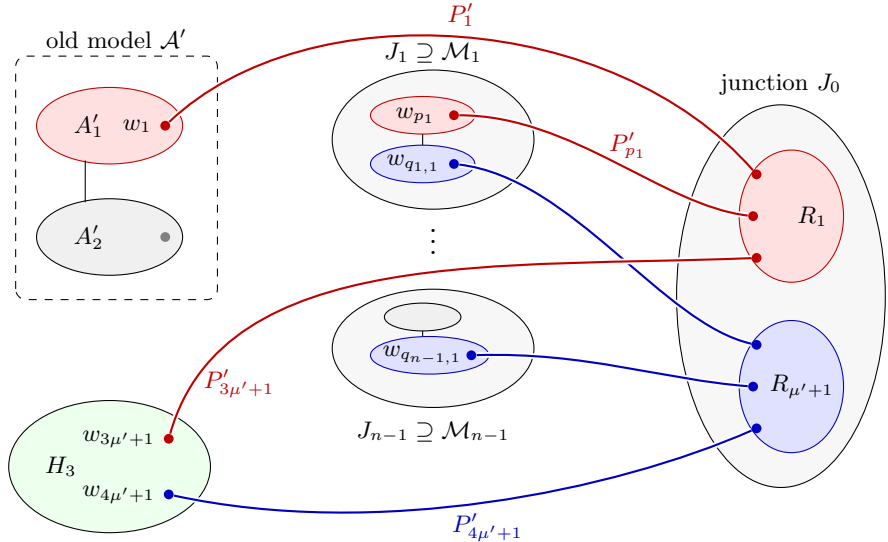
\begin{figure}[ht]
\centering
\begin{tikzpicture}[scale=0.92,
  vtx/.style={circle,fill,inner sep=1.3pt},
  lab/.style={font=\footnotesize},
  rpath/.style={red!75!black,thick,preaction={draw,white,line width=2.2pt}},
  bpath/.style={blue!75!black,thick,preaction={draw,white,line width=2.2pt}}]
% old model
\draw[dashed,rounded corners] (-1.0,1.85) rectangle (1.9,5.35);
\node[lab] at (0.45,5.6) {old model $\mathcal A'$};
\draw[fill=red!12,draw=red!75!black] (0.35,4.35) ellipse (1.05 and 0.55);
\node[lab] at (0.05,4.35) {$A'_1$};
\draw[fill=gray!12] (0.35,2.75) ellipse (1.05 and 0.55);
\node[lab] at (0.05,2.75) {$A'_2$};
\draw (0.0,3.85)--(0.0,3.25);
\node[vtx,red!75!black] (w1) at (1.15,4.35) {};
\node[lab,left] at (1.1,4.35) {$w_1$};
\node[vtx,gray] at (1.15,2.75) {};
% H_3
\draw[fill=green!7] (0.35,-0.55) ellipse (1.45 and 0.95);
\node[lab] at (-0.35,-0.55) {$H_3$};
\node[vtx,red!75!black] (h1) at (1.2,-0.15) {};
\node[lab,left] at (1.15,-0.15) {$w_{3\mu'+1}$};
\node[vtx,blue!75!black] (h2) at (1.2,-0.95) {};
\node[lab,left] at (1.15,-0.95) {$w_{4\mu'+1}$};
% piece J_1
\draw[fill=gray!6] (5.0,4.15) ellipse (1.45 and 1.0);
\node[lab] at (5.0,5.4) {$J_1\supseteq\mathcal M_1$};
\draw[fill=red!12,draw=red!75!black] (4.85,4.5) ellipse (0.75 and 0.27);
\draw[fill=blue!12,draw=blue!75!black] (4.85,3.8) ellipse (0.75 and 0.27);
\draw (4.85,4.23)--(4.85,4.07);
\node[vtx,red!75!black] (p1) at (5.3,4.5) {};
\node[lab] at (4.75,4.5) {$w_{p_1}$};
\node[vtx,blue!75!black] (q1) at (5.3,3.8) {};
\node[lab] at (4.72,3.8) {$w_{q_{1,1}}$};
% dots and piece J_{n-1}
\node at (5.0,2.8) {$\vdots$};
\draw[fill=gray!6] (5.0,1.15) ellipse (1.45 and 0.85);
\node[lab] at (5.0,-0.05) {$J_{n-1}\supseteq\mathcal M_{n-1}$};
\draw[fill=blue!12,draw=blue!75!black] (4.95,1.05) ellipse (0.85 and 0.27);
\draw[fill=gray!12] (4.85,1.6) ellipse (0.55 and 0.2);
\draw (4.85,1.32)--(4.85,1.4);
\node[vtx,blue!75!black] (q2) at (5.55,1.05) {};
\node[lab] at (4.85,1.05) {$w_{q_{n-1,1}}$};
% junction J_0
\draw[fill=gray!6] (10.0,1.9) ellipse (1.5 and 2.75);
\node[lab] at (10.0,4.95) {junction $J_0$};
\draw[fill=red!12,draw=red!75!black] (10.15,3.05) ellipse (0.75 and 0.95);
\node[lab] at (10.45,3.05) {$R_1$};
\draw[fill=blue!12,draw=blue!75!black] (10.15,0.6) ellipse (0.75 and 0.95);
\node[lab] at (10.3,0.6) {$R_{\mu'+1}$};
\node[vtx,red!75!black] (r1) at (9.65,3.65) {};
\node[vtx,red!75!black] (r2) at (9.6,3.05) {};
\node[vtx,red!75!black] (r3) at (9.65,2.45) {};
\node[vtx,blue!75!black] (b1) at (9.65,1.2) {};
\node[vtx,blue!75!black] (b2) at (9.6,0.6) {};
\node[vtx,blue!75!black] (b3) at (9.65,0.0) {};
% paths
\draw[rpath] (w1) .. controls (3.2,6.3) and (7.6,6.0) .. node[lab,above,pos=0.5] {$P'_1$} (r1);
\draw[rpath] (p1) .. controls (7.0,4.5) and (8.3,3.2) .. node[lab,above,pos=0.55] {$P'_{p_1}$} (r2);
\draw[bpath] (q1) .. controls (7.0,3.6) and (8.3,1.5) .. (b1);
\draw[rpath] (h1) .. controls (2.0,2.65) and (6.0,2.3) .. node[lab,right,pos=0.1] {$P'_{3\mu'+1}$} (r3);
\draw[bpath] (q2) .. controls (7.2,1.1) and (8.3,0.7) .. (b2);
\draw[bpath] (h2) .. controls (4.2,-1.6) and (7.6,-0.9) .. node[lab,below,pos=0.5] {$P'_{4\mu'+1}$} (b3);
\end{tikzpicture}
\caption{Step~4 of the proof of Proposition~\ref{prop:DP}, for the extended branch set $A''_1$ (red) and the new branch set $A''_{\mu'+1}$ (blue). The set $A''_1$ is the union of $A'_1$, the branch set $B(w_{p_1})$ of $\mathcal M_1$, the paths $P'_1$, $P'_{p_1}$ and $P'_{3\mu'+1}$, and $R_1$. The set $A''_{\mu'+1}$ is the union of the branch sets $B(w_{q_{i,1}})$ of $\mathcal M_i$ ($i\in[n-1]$), the paths $P'_{q_{i,1}}$ and $P'_{4\mu'+1}$, and $R_{\mu'+1}$. Each meets $H_3$ in exactly one vertex. They are adjacent, since they contain distinct branch sets of $\mathcal M_1$.}\label{fig:step4}
\end{figure}

\emph{Knitting in the junction.} The graph $J_0$ is $(9t+1)$-connected, and $e(J_0)>{\Ck}t\,v(J_0)\geqslant5(3\mu'+\mu)\,v(J_0)$. Moreover, $J_0$ has the connectivity $2(3\mu'+\mu)-\mu$ required by Lemma~\ref{lem:knit}, since $2(3\mu'+\mu)-\mu=7\mu'+x\leqslant7(t+M\sqrt{\log t})+t/(M\sqrt{\log t})+1\leqslant7.001t\leqslant9t+1$. By Lemma~\ref{lem:knit} with $3\mu'+\mu$ and $\mu$ in place of $k$ and $\ell$, the graph $J_0$ is $(3\mu'+\mu,\mu)$-knit, so there are pairwise disjoint sets $R_1,\dots,R_\mu\subseteq V(J_0)$ with $\Pi_j\subseteq R_j$ and each $J_0[R_j]$ connected.

\emph{The model.} For $j\in[\mu']$ let
\[
A''_j:=A'_j\cup V(P'_j)\cup R_j\cup V(P'_{p_j})\cup B(w_{p_j})\cup V(P'_{p'_j}),
\]
and for $j\in[x]$ let
\[
A''_{\mu'+j}:=V(P'_{4\mu'+j})\cup R_{\mu'+j}\cup\bigcup_{i\in[n-1]}\bigl(V(P'_{q_{i,j}})\cup B(w_{q_{i,j}})\bigr).
\]
Each path in these definitions starts in $A'_j$, in some set $B(\cdot)$, or in $H_3$, and ends in the set $R_j$ or $R_{\mu'+j}$ of the same definition. Since the graphs $H[A'_j]$, $H[B(\cdot)]$ and $J_0[R_l]$ are connected, $H[A''_l]$ is connected for every $l\in[\mu]$. The sets $A''_1,\dots,A''_\mu$ are pairwise disjoint since:
\begin{itemize}[nosep]
\item Within each of the following families, the members are pairwise disjoint: the sets $A'_j$, the paths of $\mathcal P'$, the sets $R_j$, and the sets $B(\cdot)$. For the sets $B(\cdot)$, this holds since they are branch sets of a model $\mathcal M_i$ in the pairwise disjoint pieces $J_i$.
\item If a path $P'_l$ of $\mathcal P'$ meets $V(\mathcal A')\cup V(J_0)\cup V(\mathcal M_1)\cup\dots\cup V(\mathcal M_{n-1})$ in a vertex $v$, then $v$ is an end of $P'_l$, and $v$ lies in the set $A''_j$ whose definition uses $P'_l$.
\item $J_0$ and the graphs $J_i$ are disjoint from $V(\mathcal A')$.
\item $J_0$ is disjoint from the graphs $J_i$.
\item Every element of $\Pi_j$ belongs to $R_j$.
\end{itemize}
The sets are pairwise adjacent. Two sets $A''_j,A''_{j'}$ with distinct $j,j'\leqslant\mu'$ are adjacent because $\mathcal A'$ is a model. For $j\leqslant\mu'$ and $j'\in[x]$, the sets $A''_j$ and $A''_{\mu'+j'}$ contain the distinct branch sets $B(w_{p_j})$ and $B(w_{q_{i(j),j'}})$ of $\mathcal M_{i(j)}$. For distinct $j,j'\in[x]$, the sets $A''_{\mu'+j}$ and $A''_{\mu'+j'}$ contain the distinct branch sets $B(w_{q_{1,j}})$ and $B(w_{q_{1,j'}})$ of $\mathcal M_1$. In the last two cases these branch sets are adjacent, since $\mathcal M_{i(j)}$ and $\mathcal M_1$ are $K_{2x}$ models. So $\mathcal A'':=\{A''_1,\dots,A''_\mu\}$ is a $K_\mu$ model.

\emph{Tangency and core.} The sets $A'_j$, $R_j$ and $B(\cdot)$ are disjoint from $H_3$, by Step~3 and since $V(J_0)$ and the sets $V(J_i)$ lie in $\overline J$. Let $S:=S'\cup \overline J\cup\{w_{3\mu'+1},\dots,w_{3\mu'+\mu}\}$. By the intersection properties of $\mathcal P'$, the paths of $\mathcal P'$ meet $V(H_3)$ exactly in the vertices $w_{3\mu'+1},\dots,w_{3\mu'+\mu}$, and $w_{3\mu'+l}$ is the first vertex of $P'_{3\mu'+l}$ for each $l\in[\mu]$. If $l\leqslant\mu'$, then $3\mu'+l=p'_l$. If $l>\mu'$, then $3\mu'+l=4\mu'+(l-\mu')$. In both cases the path $P'_{3\mu'+l}$ is used by exactly one set $A''_j$, namely $A''_l$. So $H_3$ meets each $A''_j$ exactly in $w_{3\mu'+j}$, which lies in $S$. Hence $H_3$ is tangent to $\mathcal A''$. Adjacencies between the sets $A''_j$ with $j\leqslant\mu'$ are witnessed by edges with both ends in $S'$, and all other adjacencies by edges of some $J_i$, whose ends lie in $V(J_i)\subseteq\overline J$. Hence $S$ is a core of $\mathcal A''$, and $V(H_3)\cap V(\mathcal A'')\subseteq S$. Finally,
\[
|S|\ \leqslant\ 2(n-1)^2{\PL}\,t\log t+n{\PL}\,t\log t+3t\ \leqslant\ 2n^2{\PL}\,t\log t ,
\]
because $2(n-1)^2+n+1\leqslant2n^2$ for $n\geqslant1$ and $3t\leqslant {\PL}\,t\log t$.

\medskip\noindent\emph{Step 5. Building a $K_\mu$ model when $n=1$.} Here $\mu'=0$, there are no pieces $J_i$ with $i\geqslant1$, and $\mu=x$. The construction in Step~4 does not apply, because it uses the model $\mathcal M_1$ to make the new branch sets pairwise adjacent. Instead, since $H'$ is $x$-connected and $V(H_3)$ and $V(J_0)$ have more than $k\geqslant x$ and more than $9t+1\geqslant x$ vertices, respectively, Theorem~\ref{thm:menger} gives $x$ pairwise disjoint paths from $V(H_3)$ to $V(J_0)$ in $H'$. Trimming them gives $x$ pairwise disjoint $V(H_3)$--$V(J_0)$ paths $P^\circ_1,\dots,P^\circ_x$ in $H'$, with $P^\circ_j$ joining $y_j\in V(H_3)$ to $z_j\in V(J_0)$. The graph $J_0$ is $(2x,\lfloor4.004t\rfloor)$-woven, by Step~1, and in particular $(2x,0)$-woven. Since $v(J_0)>9t+1\geqslant2x$, Lemma~\ref{lem:pad} with $b=0$ shows that $J_0$ is $(x,0)$-woven. So $J_0$ has a $K_x$ model $\{B_1,\dots,B_x\}$ with $z_j\in B_j\subseteq V(J_0)$. Let $A''_j:=V(P^\circ_j)\cup B_j$ and $S:=V(J_0)\cup\{y_1,\dots,y_x\}$. Since each $P^\circ_j$ meets $V(J_0)$ only in $z_j$ and $V(H_3)$ only in $y_j$, and each $B_j\subseteq V(J_0)$ is disjoint from $H_3$ by Step~3, these sets form a $K_x$ model $\mathcal A''$ to which $H_3$ is tangent, and $V(H_3)\cap V(\mathcal A'')=\{y_1,\dots,y_x\}\subseteq S$. Also $S$ is a core, since all required edges lie in $J_0$, and $|S|\leqslant {\PL}\,t\log t+x\leqslant2{\PL}\,t\log t$.

\medskip\noindent\emph{Step 6. Shrinking the branch sets.} By Steps~4 and~5, for each $j\in[\mu]$, the graph $H_3$ meets $A''_j$ in exactly one vertex $h_j$, and $h_j\in S$. For each $j\in[\mu]$, apply Lemma~\ref{lem:62} to the connected graph $H[A''_j]$ and the set $S\cap A''_j$, which is non-empty since it contains $h_j$. We obtain a set $A_j\subseteq A''_j$ with $H[A_j]$ connected, and a set $T_j$ with $S\cap A''_j\subseteq T_j\subseteq A_j$, $|T_j|\leqslant3|S\cap A''_j|$ and $\chi(H[A_j\setminus T_j])\leqslant2$. Let $\mathcal A:=\{A_1,\dots,A_\mu\}$. The sets $A_j\subseteq A''_j$ are pairwise disjoint, and each $H[A_j]$ is connected. Since $S\cap A''_j\subseteq A_j\subseteq A''_j$, we have $A_j\cap S=A''_j\cap S$. As $S$ is a core of $\mathcal A''$, for all distinct $j,j'$ some edge joins $A''_j\cap S=A_j\cap S$ to $A''_{j'}\cap S=A_{j'}\cap S$. So $\mathcal A$ is a $K_\mu$ model with core $S$. Next, $V(H_3)\cap A''_j=\{h_j\}$ and $h_j\in S\cap A''_j\subseteq A_j$. Hence $|V(H_3)\cap A_j|=1$ for every $j\in[\mu]$, that is, $H_3$ is tangent to $\mathcal A$. Also $V(H_3)\cap V(\mathcal A)\subseteq V(H_3)\cap V(\mathcal A'')\subseteq S$. Let $T:=\bigcup_jT_j$. Then $|T|\leqslant3|S|\leqslant6n^2{\PL}\,t\log t\leqslant N_t$, so $\chi(H[T])\leqslant\Gamma$. Also $V(\mathcal A)\setminus T$ is the disjoint union of the sets $A_j\setminus T_j$, and each $H[A_j\setminus T_j]$ can be coloured with two colours of its own. So $\chi(H[V(\mathcal A)\setminus T])\leqslant2\mu\leqslant6t$. Hence $\chi(H[V(\mathcal A)])\leqslant\chi(H[T])+\chi(H[V(\mathcal A)\setminus T])\leqslant6t+\Gamma$.

\medskip\noindent\emph{Step 7. Restoring the invariant.} Let $H_4:=H-V(\mathcal A)$. Then $H_4$ is an induced subgraph of $H$, and
\[
\chi(H_4)\ \geqslant\ \chi(H)-\chi(H[V(\mathcal A)])\ \geqslant\ \chi(H)-6t-\Gamma\ \geqslant\ \Psi-6t-\Gamma\ \geqslant\ \lceil3.0625(k-1)\rceil .
\]
By Corollary~\ref{cor:nguyen}, $H_4$ has a $k$-connected induced subgraph $H_5$ with $\chi(H_5)\geqslant\chi(H)-6t-\Gamma-2k=\chi(H)-\Theta$. Since $H_4$ is an induced subgraph of $H$, so is $H_5$. Also $V(H_5)\cap V(\mathcal A)=\emptyset$. As illustrated in Figure~\ref{fig:step7}, there are two cases, according to whether $H_5$ and $H_3$ share fewer than $k$ vertices.

\emph{Case 1. $|V(H_5)\cap V(H_3)|<k$.} Let $Y_1:=V(H_5)\setminus V(H_3)$ and $Y_2:=V(H_3)$. These sets are disjoint, so $H[Y_1]$ and $H[Y_2]$ are vertex-disjoint induced subgraphs of $H$. Deleting a vertex lowers the chromatic number by at most $1$, and $H[Y_1]=H_5-(V(H_5)\cap V(H_3))$ since $H_5$ is induced. Hence
\[
\chi(H[Y_1])\ \geqslant\ \chi(H_5)-(k-1)\ \geqslant\ \chi(H)-\Theta-k\ \geqslant\ \chi(H)-\sigma ,
\]
since $\sigma-\Theta=24t+\Gamma+2k\geqslant k$. Moreover $H[Y_2]=H_3$, since $H_3$ is an induced subgraph of $H$, so $\chi(H[Y_2])=\chi(H_3)\geqslant\chi(H)-\sigma$ by Step~3. Thus $H$ is $\sigma$-chromatic-separable, a contradiction.

\emph{Case 2. $|V(H_5)\cap V(H_3)|\geqslant k$.} Let $R:=V(H_3)\cup V(H_5)$ and $J^\ast:=H[R]$. We claim that $J^\ast$ is $k$-connected. Indeed, $v(J^\ast)\geqslant v(H_5)>k$. Let $Z\subseteq R$ with $|Z|<k$. Then $H_3-Z$ and $H_5-Z$ are connected, because $H_3$ and $H_5$ are $k$-connected. They share a vertex, because $|V(H_3)\cap V(H_5)|\geqslant k>|Z|$. Their union is a connected spanning subgraph of $J^\ast-Z$. So $J^\ast-Z$ is connected. This proves the claim.

We now verify $(\ast)_n$ for $\mathcal A$, $S$ and $J^\ast$. Property~(I1) holds by Step~6, and by the bound on $|S|$ in Step~4 or Step~5. Property~(I2) was just shown. For~(I3) and~(I4), recall that $V(H_5)\cap V(\mathcal A)=\emptyset$. So $V(J^\ast)\cap A_j=V(H_3)\cap A_j$ for every $j\in[\mu]$. This set has exactly one vertex, by Step~6. Also $V(J^\ast)\cap V(\mathcal A)=V(H_3)\cap V(\mathcal A)\subseteq S$, by Step~6. Finally, $\chi(J^\ast)\geqslant\chi(H_5)\geqslant\chi(H)-\Theta$, which is~(I5). This proves $(\ast)_n$.
\end{proof}

\begin{figure}[ht]
\centering
\begin{tikzpicture}[scale=0.85,lab/.style={font=\footnotesize}]
% ---------- Case 1 ----------
\draw[fill=red!10,draw=red!70!black] (0,0) ellipse (1.5 and 1.0);
\draw[fill=blue!10,draw=blue!70!black,fill opacity=0.7] (2.55,0) ellipse (1.5 and 1.0);
\begin{scope}
\clip (0,0) ellipse (1.5 and 1.0);
\fill[violet!35] (2.55,0) ellipse (1.5 and 1.0);
\end{scope}
\node[lab] at (-0.35,0) {$H_3=H[Y_2]$};
\node[lab] at (3.3,0) {$H_5$};
\node[lab,align=center] at (1.27,-1.45) {fewer than $k$\\common vertices};
\node[lab] at (3.15,1.25) {$Y_1=V(H_5)\setminus V(H_3)$};
\node at (1.27,-2.25) {Case~1};
% ---------- Case 2 ----------
\begin{scope}[xshift=7.6cm]
\draw[fill=red!10,draw=red!70!black] (0,0) ellipse (1.5 and 1.0);
\draw[fill=blue!10,draw=blue!70!black,fill opacity=0.7] (1.4,0) ellipse (1.5 and 1.0);
\begin{scope}
\clip (0,0) ellipse (1.5 and 1.0);
\fill[violet!35] (1.4,0) ellipse (1.5 and 1.0);
\end{scope}
\node[lab] at (-0.9,0) {$H_3$}; \node[lab] at (2.3,0) {$H_5$};
\node[lab] at (0.7,0) {$\geqslant k$};
\node[lab] at (0.7,1.3) {$J^\ast=H[V(H_3)\cup V(H_5)]$};
\node at (0.7,-2.25) {Case~2};
\end{scope}
\end{tikzpicture}
\caption{Step~7 of the proof of Proposition~\ref{prop:DP}. In Case~1, the $k$-connected graphs $H_3$ and $H_5$ share fewer than $k$ vertices, and $Y_1=V(H_5)\setminus V(H_3)$ and $Y_2=V(H_3)$ give a chromatic separation of $H$. In Case~2, they share at least $k$ vertices, so $J^\ast=H[V(H_3)\cup V(H_5)]$ is $k$-connected.}\label{fig:step7}
\end{figure}
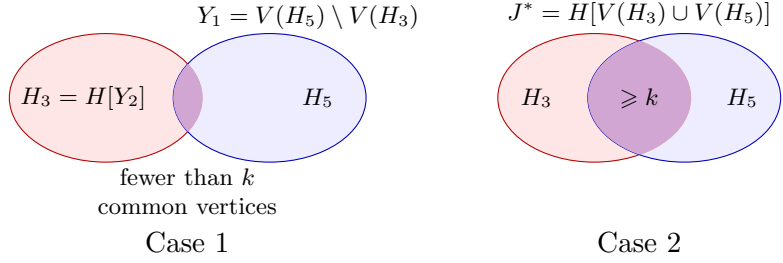

\paragraph{Comparison with the proof of Delcourt and Postle.} The proof of Proposition~\ref{prop:DP} follows the proof of Lemma~6.4 of Delcourt and Postle~\cite{DP}, and $(\ast)_n$ is their inductive statement. The main differences are as follows. We use Corollary~\ref{cor:nguyen}, Theorem~\ref{thm:refinedlin} and Lemmas~\ref{lem:knit} and~\ref{lem:Ck} in place of~\cite[Theorems~5.1, 2.3 and~5.8 and Lemma~5.13]{DP}. We use $M$ times as many stages, each adding $x=\lceil t/(M\sqrt{\log t})\rceil$ branch sets instead of $\lceil t/\sqrt{\log t}\rceil$. We record the chromatic number lost at each step, so that $\Gamma$ enters $\sigma$ and $\Psi$ only with coefficient~$2$, and $\Cdens$ enters only through the term $2d$ of $\Psi$. Step~5 is new, since for $n=1$ their construction gives no adjacency between the new branch sets.

%=====================================================================
\section{Proof of Theorem~\ref{thm:all}}\label{sec:proof}
%=====================================================================

This section combines Proposition~\ref{prop:DP}, Corollary~\ref{cor:allscales} and an adaptation of Lemma~\ref{lem:weave} in a final separation step, and then completes the proof.

\begin{corollary}\label{cor:sepP}
Let $b$ be an integer with $\log b\geqslant1.6\cdot10^6$, and let
\[
\sigma_b:=(78+2{\cst})\,b\qquad\text{and}\qquad\Psi_b:=(32+2\Cdens+2{\cst})\,b+2 .
\]
Every graph $F$ with no $K_b$ minor and $\chi(F)\geqslant\Psi_b$ is $\sigma_b$-chromatic-separable.
\end{corollary}

\begin{proof}
Since $\log b\geqslant1.6\cdot10^6$, we have $b\geqslant3$. Let $U\subseteq V(F)$ with $|U|\leqslant N_b$. Then $F[U]$ has no $K_b$ minor, so Corollary~\ref{cor:allscales} gives $\chi(F[U])\leqslant{\cst} b$. So Proposition~\ref{prop:DP} applies with $b$, $F$ and ${\cst} b$ in place of $t$, $H$ and $\Gamma$, and its first conclusion is the claim.
\end{proof}

In the final application, we use Corollary~\ref{cor:sepP} in Step~7b of the proof of Lemma~\ref{lem:weave}, in place of $(\mathcal H_{\Phi,a_0})$, with ${\LamF}$ and ${\KF}$ as defined in Lemma~\ref{lem:weavefinal} below. The separation cost $\sigma_b=(78+2{\cst})b$ does not involve $\Cdens$. Each separation in Step~7b has cost $\sigma_b$ with $b=\min\{2a,t\}\leqslant2a$, and these costs are covered by the term ${\LamF}a$, as in Step~8. The chromatic threshold $\Psi_b$ does involve $\Cdens$. Since $b\leqslant t$, the requirement $\chi(G_4)\geqslant\Psi_b+\sigma_b$, which allows two successive separations, is met by the term ${\KF}t$ in the hypothesis of Lemma~\ref{lem:weavefinal}. This term is the same at every level of the induction on $a$. So $\Cdens$ does not enter ${\LamF}$, and enters ${\KF}$, and hence the final constant, only through the summand $2\Cdens$. By contrast, using $(\mathcal H_{\Phi,a_0})$ with the second conclusion of Proposition~\ref{prop:DP} would put $2\Cdens$ into $\Phi$, and the coefficient $10$ of $\Phi$ in ${\LamU}$ would give the summand $20\Cdens$ (see Appendix~\ref{app:further}).

\begin{lemma}\label{lem:weavefinal}
Let $T_{\mathrm{fin}}:=18.4$. Let $t$ be an integer with $\log\log t\geqslant T_{\mathrm{fin}}$, and let $a_0:=t/\log t$ and $\eta:=5/(a_0-4)$. Let
\[
{\LamF}:=(1+\eta)\bigl(897.5+10(78+2{\cst})\bigr)\qquad\text{and}\qquad {\KF}:=110+2\Cdens+4{\cst}+\frac2t .
\]
Let $a\geqslant1$ be an integer, and let $G$ be a $21a$-connected graph with no $K_t$ minor and
\[
\chi(G)\ \geqslant\ D\bigl(\tfrac{46}3a_0+4\bigr)+1+{\LamF} a+{\KF}t .
\]
Then $G$ is $(a,\lceil10a/3\rceil)$-woven.
\end{lemma}

\begin{proof}
Since $a_0=t/\log t\geqslant9$ when $\log\log t\geqslant T_{\mathrm{fin}}$, we have $\eta\in(0,1]$ and $a_0=4+5/\eta$. The proof is the proof of Lemma~\ref{lem:weave} for the uneven scheme, with the following changes. Throughout, $t$ is fixed, and we use induction on $a$. Every graph to which the induction hypothesis is applied is an induced subgraph of $G$, and so has no $K_t$ minor. Steps~1--6, 7a and~9 are unchanged, and Step~2 gives $a>a_0$ as before, since $\chi(G)\geqslant D(\frac{46}3a_0+4)+1$. Steps~3 and~5 use only $\chi(G)\geqslant1+897.5a$, which holds since ${\LamF}>897.5$. In Step~7b, the choice of $R'$ and $S'$, the definition of $G_4$ and the bound $\chi(G_4)\geqslant\chi(G_3)-3a$ are unchanged. The rest of Step~7b up to the definition of $J_1,J_2,J_3$ is replaced by the following, and Step~8 is modified as described below.

Let $b:=\min\{2a,t\}$, which is an integer. The graph $G_4$ has no $K_{2a}$ minor by Step~7a, and no $K_t$ minor since $G$ has none. So $G_4$ has no $K_b$ minor. Also $b\geqslant\min\{a,t\}\geqslant a_0$. Let $L:=\log t$, so that $L\geqslant e^{T_{\mathrm{fin}}}\geqslant9.7\cdot10^7$. Then $\log L\leqslant L/2$, since $e^{L/2}\geqslant(L/2)^2/2=L^2/8\geqslant L$. So $\log b\geqslant\log a_0=L-\log L\geqslant L/2\geqslant4.85\cdot10^7\geqslant1.6\cdot10^6$. Since $b\leqslant t$,
\[
\Psi_b+\sigma_b\ =\ (110+2\Cdens+4{\cst})\,b+2\ \leqslant\ {\KF}t .
\]
By Steps~3--5 and the bound $\chi(G_4)\geqslant\chi(G_3)-3a$, we have $\chi(G_4)\geqslant\chi(G)-(137.5a+17)$. Since ${\LamF} a\geqslant137.5a+17$, it follows that $\chi(G_4)\geqslant{\KF}t\geqslant\Psi_b+\sigma_b$. By Corollary~\ref{cor:sepP}, the graph $G_4$ is $\sigma_b$-chromatic-separable. So it has vertex-disjoint induced subgraphs $F_1,F_2$ with $\chi(F_i)\geqslant\chi(G_4)-\sigma_b\geqslant\Psi_b$. Applying Corollary~\ref{cor:sepP} to $F_1$, which also has no $K_b$ minor, gives vertex-disjoint induced subgraphs $F_{1,1},F_{1,2}$ of $F_1$ with $\chi(F_{1,j})\geqslant\chi(F_1)-\sigma_b$. Let $J_1:=F_{1,1}$, $J_2:=F_{1,2}$ and $J_3:=F_2$. Then $\sigma_b=(78+2{\cst})b\leqslant2(78+2{\cst})a$. So the rest of Step~7b, and Step~8, go through as in the proof of Lemma~\ref{lem:weave} with $\sigma_b$ in place of $\sigma$, with $\Phi$ replaced by $78+2{\cst}$, with ${\LamF}$ in place of ${\LamU}$, and with ${\KF}t$ added to the lower bound on $\chi(G)$ in the hypothesis and to the resulting lower bounds on $\chi(J'_i)$ in Step~8. In particular, each $J'_i$ is a $21a'_i$-connected graph with no $K_t$ minor and
\[
\chi(J'_i)\ \geqslant\ D\bigl(\tfrac{46}3a_0+4\bigr)+1+{\LamF} a'_i+{\KF}t ,
\]
so the induction hypothesis applies to it.
\end{proof}

The even scheme would give $(1+\eta)\bigl(517.5+12(78+2{\cst})\bigr)$ in place of ${\LamF}$, which is larger, since $78+2{\cst}>190$.

We now prove the main theorem.

\begin{proof}[Proof of Theorem~\ref{thm:all}]
Let $G$ be a graph, and let $h:=\had(G)$. If $h\leqslant1$, then $G$ has no edge, since an edge gives a $K_2$ minor. So $\chi(G)=0=h$ if $G$ is null, and $\chi(G)=1=h$ otherwise. Now let $h\geqslant2$ and $t:=h+1$, so that $G$ has no $K_t$ minor. Since $t\geqslant3$, Theorem~\ref{thm:kos} gives $\chi(G)<60\,t\sqrt{\log t}+1$.

If $h<10^6$, then $\log t\leqslant t-1=h<10^6\leqslant1001^2$, so $\sqrt{\log t}\leqslant1001$. Since $h\geqslant2$,
\[
\chi(G)\ <\ 60{,}060\,t+1\ =\ 60{,}060\,h+60{,}061\ \leqslant\ 1{,}118{,}430\,h .
\]
Now let $h\geqslant10^6$.

\emph{Case 1. $\log\log t<T_{\mathrm{fin}}$.} Then $\log t<e^{T_{\mathrm{fin}}}$, so $\sqrt{\log t}<e^{T_{\mathrm{fin}}/2}=e^{9.2}<10^4$. Since $h\geqslant10^6$,
\[
\chi(G)\ <\ 600{,}000\,t+1\ =\ 600{,}000\,h+600{,}001\ \leqslant\ 1{,}118{,}430\,h .
\]

\emph{Case 2. $\log\log t\geqslant T_{\mathrm{fin}}$.} Let $a_0$, $\eta$, ${\LamF}$ and ${\KF}$ be as in Lemma~\ref{lem:weavefinal}, and let $L:=\log t$. Then $L\geqslant e^{T_{\mathrm{fin}}}\geqslant9.7\cdot10^7$. Since $t=e^L\geqslant L^3/6$, we have $t\geqslant a_0=t/L\geqslant L^2/6\geqslant2.5\cdot10^9$.

We first show that $D(\frac{46}3a_0+4)+1\leqslant t/10$. Let $s_0:=\lfloor\frac{46}3a_0+4\rfloor$, so that $D(\frac{46}3a_0+4)=D(s_0)$. Since $a_0\geqslant6$, we have $2\leqslant s_0\leqslant16a_0\leqslant16t$, and so $\log s_0\leqslant\log16+L<L+3$. Since $L\geqslant9.7\cdot10^7\geqslant9650^2(1+3/L)$, we have $L+3\leqslant(L/9650)^2$, that is, $\sqrt{L+3}\leqslant L/9650$. By Theorem~\ref{thm:kos}, and since $a_0L=t$,
\[
D(s_0)\ \leqslant\ 60\,s_0\sqrt{\log s_0}+1\ \leqslant\ 60\cdot16a_0\cdot\frac{L}{9650}+1\ =\ \frac{960}{9650}\,t+1\ =\ \frac t{10}-\frac t{1930}+1 .
\]
Since $t\geqslant3860$, it follows that $D(s_0)\leqslant t/10-1$.

Next, every $21t$-connected graph $G'$ with no $K_t$ minor satisfies
\[
\chi(G')\ <\ \frac t{10}+{\LamF}t+{\KF}t .
\]
Otherwise $\chi(G')\geqslant D(\frac{46}3a_0+4)+1+{\LamF}t+{\KF}t$, and Lemma~\ref{lem:weavefinal} with $a=t$ shows that $G'$ is $(t,\lceil10t/3\rceil)$-woven. Since $G'$ has more than $21t$ vertices, it has a $K_t$ minor, a contradiction.

We now claim that
\[
\chi(G)\ <\ \frac t{10}+42t+{\LamF} t+{\KF}t .
\]
Suppose not. Then $\chi(G)\geqslant42t+{\LamF}t\geqslant3.0625(21t-1)$, since ${\LamF}\geqslant64$. As $\chi(G)$ is an integer, $\chi(G)\geqslant\lceil3.0625(21t-1)\rceil$. So by Corollary~\ref{cor:nguyen} with $k=21t$, $G$ has a $21t$-connected induced subgraph $G'$ with $\chi(G')\geqslant\chi(G)-42t\geqslant t/10+{\LamF} t+{\KF}t$. Since $G'$ has no $K_t$ minor, this contradicts the previous paragraph. This proves the claim.

Finally, we bound the right-hand side of the claim. Let $\Lambda_0:=897.5+10(78+2{\cst})=902{,}177.246$, so that ${\LamF}=(1+\eta)\Lambda_0$. Since $a_0-4\geqslant2500\,\Lambda_0$, we have $\eta\Lambda_0=5\Lambda_0/(a_0-4)\leqslant0.002$. Also ${\KF}t=(110+2\Cdens+4{\cst})t+2$. Since $10(78+2{\cst})+4{\cst}=780+24{\cst}$ and $0.1+42+897.5+780+110=1829.6$,
\[
\frac t{10}+42t+\Lambda_0t+(110+2\Cdens+4{\cst})t\ =\ (1829.6+2\Cdens+24{\cst})\,t\ =\ {\cfin}t .
\]
Hence, by the claim,
\[
\chi(G)\ <\ ({\cfin}+0.002)\,t+2\ \leqslant\ 1{,}118{,}430\,(t-1)\ =\ 1{,}118{,}430\,h .
\]
Here the second inequality holds because ${\cfin}=1{,}118{,}429.2952$, so it is equivalent to $0.7028\,t\geqslant1{,}118{,}432$, which holds as $t\geqslant2.5\cdot10^9$.

In all cases $\chi(G)\leqslant1{,}118{,}430\,\had(G)$.
\end{proof}

The margin $1{,}118{,}430-{\cfin}=0.7048$ absorbs the error term $\eta\Lambda_0t\leqslant0.002t$ from ${\LamF}t$, the summand $2$ of ${\KF}t$, and the change from $t$ to $h=t-1$. The value $T_{\mathrm{fin}}=18.4$ is, to one decimal place, the least value for which the argument of Case~2 gives $D(\frac{46}3a_0+4)+1\leqslant t/10$. That argument bounds $D(s_0)$ by $960\,t\sqrt{L+3}/L+1$, which is less than $t/10$ only if $960\sqrt{L+3}<L/10$, and this fails for $L=e^{18.3}$. A larger value of $T_{\mathrm{fin}}$ up to $19.6$ would give the same constant, since $60\,e^{19.6/2}<1.09\cdot10^6$, so Case~1 would still give $\chi(G)\leqslant1{,}118{,}430\,h$ (as $h\geqslant10^6$).

\appendix

%=====================================================================
\section{Postscript}\label{app:further}
%=====================================================================

This paper was first submitted to the arXiv on the morning of 7th October 2026 (Melbourne time). A few hours later, OpenAI released its disproof of Hadwiger's Conjecture~\cite{OpenAI-HC}.\footnote{OpenAI also proved that $\operatorname{ch}(G)\in O(\had(G))$~\cite{OpenAI-list}, where $\operatorname{ch}(G)$ is the list chromatic number of $G$. Since $\chi(G)\leqslant\operatorname{ch}(G)$, this implies the Linear Hadwiger Theorem of Norin and Steiner~\cite{NS}. No explicit constant is given for the bound $\operatorname{ch}(G)\in O(\had(G))$.} OpenAI's construction gives infinitely many graphs $G$ with $\alpha(G)\leqslant2$ and $\had(G)<\frac{26}{75}v(G)+\frac23$, implying $\chi(G)\geqslant \frac12 v(G)>\frac{75}{52}\had(G)-\frac{25}{26}$ (and $\had(G)\to\infty$ along this family). Two natural questions arise. Let $c_1$ be the infimum of the real numbers $c$ such that $\chi(G)\leqslant c\,\had(G)$ for every graph $G$. Let $c_2$ be the infimum of the real numbers $c$ such that $\chi(G)\leqslant(c+o(1))\had(G)$ for every graph $G$, where the $o(1)$ term tends to $0$ as $\had(G)\to\infty$. What are $c_1$ and $c_2$? We know that $\frac{75}{52}\leqslant c_2\leqslant c_1\leqslant1{,}118{,}430$. With the goal of closing this gap, we now identify the parts of the argument that determine the upper bound in Theorem~\ref{thm:all}, and indicate what would be needed to reduce it further.

\paragraph{The dominant term.} The constant ${\cfin}=1829.6+2\Cdens+24{\cst}$ is dominated by $24{\cst}\approx1.081\cdot10^6$, whose origin is explained in Section~\ref{sec:constants}. The term $2\Cdens=36{,}000$ contributes about $3.2\%$. So ${\cfin}$ is determined mainly by the bootstrap constant ${\cst}$ and by the coefficient $10$ of $\Phi$ in ${\LamU}$, which turns $78+2{\cst}$ into the term $20{\cst}$ of $24{\cst}$. In particular, further improvements to the weaving constant ${\Ck}$ or to the proportion $\cL$ affect only $\Cdens$, and so would change ${\cfin}$ by less than $3.2\%$. Using Proposition~\ref{prop:DP} only through its second conclusion, that is, Lemma~\ref{lem:weave} with $(\mathcal H_{\Phi,a_0})$ in place of Lemma~\ref{lem:weavefinal}, gives a constant of about $939.6+10(2\Cdens+33)+20{\cst}\approx1.262\cdot10^6$, which is about $13\%$ more. Several estimates below are values of the formula $1829.6+2\Cdens+24{\cst}$ or of a variant of it. A value of ${\cfin}$ below $60\,e^{T_{\mathrm{fin}}/2}$, which exceeds $593{,}827$, would also require a smaller size threshold $T_{\mathrm{fin}}$, since Case~1 of the proof of Theorem~\ref{thm:all} uses only $\chi(G)\leqslant60\,t\sqrt{\log t}+1$, which is close to $60\,e^{T_{\mathrm{fin}}/2}\,t$ when $\log\log t$ is just below $T_{\mathrm{fin}}$.

\paragraph{The chromatic-inseparable case.} This case affects the main constant mainly through $\Gamma$, which enters each of $\sigma$ and $\Psi$ twice (see the paragraph after the proof of Proposition~\ref{prop:DP}). In each of $\sigma$ and $\Psi$, one copy of $\Gamma$ enters through $\Theta$. This copy of $\Gamma$ bounds the chromatic number of the set $T$ of Step~6 of the proof of Proposition~\ref{prop:DP}, which lies in the new branch sets. The other copy bounds the chromatic number of a union of pieces, namely the set $\overline J$ of Step~2 (used in Step~3) for $\sigma$, and the set $U$ of Step~1 for $\Psi$. In the final separation step (Corollary~\ref{cor:sepP}), $\Gamma={\cst}b$, and the two copies of $\Gamma$ produce the term $24{\cst}$ of ${\cfin}$ (Section~\ref{sec:constants}). An argument that loses $\Gamma$ only once would lower this term to $12{\cst}$, a reduction of $12{\cst}\approx5.4\cdot10^5$, almost half of ${\cfin}$. The other contributions of this case are small. Better connectivity results would change only the terms $2\Cdens=36{,}000$ and $1829.6$, which together are less than $4\%$ of ${\cfin}$. Lemma~\ref{lem:Ck} needs pieces that are $9t$-connected and have density ${\Ck}t=70t$, and Lemma~\ref{lem:mader}, used in the proof of Theorem~\ref{thm:refinedlin}, gives both. By Lemma~\ref{lem:mader}, average degree $4k$ forces a $(k+1)$-connected subgraph with more than $2k$ vertices. Bernshteyn and Kostochka~\cite{BK16} showed that average degree $(3+\frac16)k$ forces a $(k+1)$-connected subgraph, and Carmesin~\cite{Car20} showed that average degree $(3+\frac13)k$ forces a $(k+1)$-connected subgraph with more than $2k$ vertices, and that this is sharp. These results give connectivity without the density, and a result that gives only one of the two properties does not help. Kawarabayashi and Yu~\cite{KY26} prove that every $8\ell$-connected graph is \defn{$\ell$-knitted}, that is, $(\ell,1)$-knit in our terminology, correcting a gap in their earlier argument for $9\ell$~\cite{KY13}. The knitting in the proof of Proposition~\ref{prop:DP} needs connectivity at most $7.001t$, which is below the $9t$ already required by Lemma~\ref{lem:Ck}, so a better knitting result would not help.

\paragraph{The bootstrap.} Within the framework of Section~\ref{sec:bootstrap}, the schedule of Proposition~\ref{prop:explicitbootstrap} is optimal up to the error terms $73\xi$. A dynamic program over all schedules (script \texttt{bootstrap\_dp.py}) takes $\xi=0$ and rounds each exponent reached up to the right end of a short interval, which can only help a schedule. Its value $45{,}024.98$ is therefore a lower bound on the constant given by any schedule, and ${\cst}=45{,}024.98+73\xi$. So a reduction of ${\cst}$ by more than $73\xi$ needs a change to the framework itself. There are three places to look. First, each step doubles the previous constant, because $G_1$ is coloured through its contraction (step~(i) of the proof of Lemma~\ref{lem:level}). A factor smaller than $2$ would reduce each of the first three terms of ${\cst}=2^6\cdot4.01+62{\Sst}+\cS(c_{\mathrm{side}})+\xi$. Second, each step raises the exponent by less than half the exponent $\beta_{\mathrm G}$ at which the Gy\'arf\'as graphs are coloured, since these graphs have order up to $2b'k(k-1)$ with $k=\lceil(\log t)^\zeta\rceil$ (by the condition $(\beta_{\mathrm G}-2\zeta)T\geqslant1$ of Lemma~\ref{lem:level}). Third, the final exponent must exceed $2$, because the order bound $N_t$ of Proposition~\ref{prop:DP} satisfies $N_t\in O(t\log^2t)$ (see the remark after Proposition~\ref{prop:explicitbootstrap}). A larger increment per step, or a smaller order bound, would save steps. For example, one fewer main step would replace ${\cst}$ by $2\bar c_4+\cS(c_{\mathrm{side}})+\xi\approx25{,}453$, and ${\cfin}$ by about $6.49\cdot10^5$. The additive term $\cS(\Phi)$ of each step and the base case are discussed below.

\paragraph{The base case.} The base case enters the main constant through its constant and through its exponent. In ${\cst}=2^6\cdot4.01+62{\Sst}+\cS(c_{\mathrm{side}})+\xi$, the base constant $4.01$ appears with coefficient $2^6=64$ in the first term, with coefficient $12$ in each of the $62$ copies of ${\Sst}=\cS(4.01)+\xi$, and with coefficient $10\cdot14=140$ in $\cS(c_{\mathrm{side}})$, since $c_{\mathrm{side}}=2\cdot4.01+{\Sst}$. So each unit of the base constant contributes $64+744+140=948$ to ${\cst}$, and hence $22{,}752$ to ${\cfin}$. The constant $4$ of the small-graph bound of Norin and Steiner and of Theorem~\ref{thm:NSquant} is the factor $2$ of Reed and Seymour times the factor $2$ from a bipartition step in the proof of~\cite[Lemma~8]{NS}. Removing the second factor would lower ${\cfin}$ by about $45{,}500$. The exponent matters more. The base case stops at exponent $\frac12$ because the additive term $(100/\eps)^{2000/\eps^2}$ of Theorem~\ref{thm:NSquant} is at most $\eps a$ for $\eps$ of order $(\log a)^{-\beta}$ only if $\beta<\frac12$. In the argument of Norin and Steiner, this reflects the error term $(d+1)^{4m}$ of~\cite[Lemma~8]{NS}, with $d\in\exp(O(m\log m))$, which is in $\exp(O(m^2\log m))$~\cite[Section~2.1]{NS}. An error term in $\exp(O(m\log m))$ would extend the base case to exponent $1-o(1)$. Then $\zeta$ could be close to $\frac12$ in Lemma~\ref{lem:level}, and three bootstrapping steps with the base case for the Gy\'arf\'as graphs would suffice. These three steps would give the constant $2^3\cdot4.01+7{\Sst}\approx4285.4$. However, Corollary~\ref{cor:allscales} also uses the bound $(144.4+4T_{\mathrm{boot}})a=4788.4\,a$ of Section~\ref{sec:small}. So ${\cst}$ would be replaced by $4788.4$, and ${\cfin}$ by about $1.53\cdot10^5$. If $T_{\mathrm{boot}}$ were also lowered to at most $1035.25$, so that $144.4+4T_{\mathrm{boot}}\leqslant4285.4$, then ${\cst}$ would be replaced by about $4285.4$, and ${\cfin}$ by about $1.41\cdot10^5$.

\paragraph{The chromatic-separable reduction.} The cost of this reduction is $\cS(\Phi)=42+\min\{517.5+12\Phi,\,897.5+10\Phi\}$. In the uneven-scheme term $897.5+10\Phi$ of $\cS(\Phi)$, the coefficient $10$ of $\Phi$ produces the term $20{\cst}$ of ${\cfin}$, since $78+2{\cst}$ plays the role of $\Phi$ in the final step (see above). The coefficient $12$ enters ${\cst}$ through the side step and the five main steps, and the coefficient $10$ enters ${\cst}$ through the final bootstrapping step. The coefficient $10=5\cdot2$ has two factors. The factor $2$ arises because Lemma~\ref{lem:kawa} gives rooted $K_a$ models only from a $K_{2a}$ minor, so Step~7b of the proof of Lemma~\ref{lem:weave} separates graphs with no $K_{2a}$ minor. A lemma that finds rooted $K_a$ models from a $K_{(1+o(1))a}$ minor in $21a$-connected graphs would halve both coefficients, to $6$ and $5$. With all other terms unchanged, this would give ${\cst}\approx40{,}335$ and ${\cfin}\approx6.0\cdot10^5$. The factor $5$ arises because two chromatic separations give three subgraphs, and each branch set uses two of them (Step~9). Among the ways of assigning two of the three subgraphs to each branch set, the uneven scheme is optimal (see the discussion after Table~\ref{tab:losses}), so a smaller factor needs a different way to combine the subgraphs. The additive terms $517.5=3\cdot172.5$ and $897.5=5\cdot179.5$ come from the losses listed in Table~\ref{tab:losses}, mainly from the set $Z$, the hub and the two extractions of $21a$-connected subgraphs. Lowering each of $517.5$ and $897.5$ by $1$ would lower ${\cst}$ by $73$ and ${\cfin}$ by about $1750$. The connectivity $21a$ is needed because Step~9 needs $G_3-R'$ to be $2a$-linked, and the connectivity form of the Thomas--Wollan theorem gives this when $G_3-R'$ is $20a$-connected. A lower connectivity would lower the losses $42a$ of the two extractions of $21a$-connected subgraphs. Theorem~\ref{thm:TW} might allow lower connectivity, but it needs density and connectivity in the same subgraph, which Corollary~\ref{cor:nguyen} does not provide.

%=====================================================================
\section{Computer verification}\label{app:lean}
%=====================================================================

\begin{sloppypar}
\begin{table}[!t]
\centering\small
\begin{tabular}{@{}>{\raggedright}p{6.8cm}>{\raggedright}p{2.5cm}>{\raggedright\arraybackslash}p{4.3cm}@{}}\toprule
Claim & Where stated & Script\\\midrule
Constants ${\Sst}$, $c_{\mathrm{side}}$, $\bar c_j$, ${\cst}$, ${\cfin}$, and the parts $897.5+10(78+2{\cst})$ of ${\LamF}$ and $110+2\Cdens+4{\cst}$ of ${\KF}$ & Secs.~\ref{sec:constants}, \ref{sec:bootstrap}, \ref{sec:proof} & \texttt{precise\_check.py} (exact)\\
The remarks after Theorem~\ref{thm:kos} and Lemma~\ref{lem:hub} & Secs.~\ref{sec:prelim}, \ref{sec:sep} & \texttt{precise\_check.py} (exact, interval)\\
Linear bounds in Lemmas~\ref{lem:hub}, \ref{lem:weave} and Theorem~\ref{thm:sepexplicit}, Table~\ref{tab:losses}, the crossover $\Phi=190$, the coefficients $10$ and $12$ & Sec.~\ref{sec:sep} & \texttt{precise\_check.py} (exact)\\
The bounds $2.7182818283^{23}\cdot1.22>1.188\cdot10^{10}>578.4K+10$ and $1.188\cdot10^{10}>24.5K$, and the other bounds in Lemmas~\ref{lem:Aeps} and~\ref{lem:base} & Sec.~\ref{sec:base} & \texttt{precise\_check.py} (interval)\\
The bounds for $y\geqslant100$ in the proof of Lemma~\ref{lem:level} & Sec.~\ref{sec:level} & \texttt{precise\_check.py} (exact, interval)\\
The conditions of Lemma~\ref{lem:level} on $T$, $\zeta$, $\beta_{\mathrm G}$ and $\Phi$ at the size thresholds, the bounds $24{\PL}M^2\leqslant2.7^{79}$ and $0.0792\,T_{\mathrm{boot}}\geqslant79$, and the bounds $79/0.0792<998$ and $873.4$ in the remark after Table~\ref{tab:schedule} & Sec.~\ref{sec:schedule} & \texttt{precise\_check.py} (exact, interval)\\
The bounds $24{\PL}M^2\leqslant2.7^{70}$, $\log(24{\PL}M^2)\leqslant70$ and $4788.4$ of Section~\ref{sec:small} & Sec.~\ref{sec:small} & \texttt{precise\_check.py} (interval)\\
Bounds in Lemmas~\ref{lem:massed} and~\ref{lem:Ck} & Sec.~\ref{sec:wollan} & \texttt{precise\_check.py} (exact, interval)\\
Constants in Section~\ref{sec:lin} & Sec.~\ref{sec:lin} & \texttt{precise\_check.py} (exact, interval)\\
Bounds in the proof of Proposition~\ref{prop:DP} ($4.004t$, $7.001t$, $24t$, orders) & Sec.~\ref{sec:insep} & \texttt{precise\_check.py} (exact, interval)\\
The bounds in the proof of Lemma~\ref{lem:weavefinal}, the bounds $t/10$, $9650$, $2.5\cdot10^9$, $0.002$, $600{,}000$ and $60{,}060$ in the proof of Theorem~\ref{thm:all}, and the remarks on $T_{\mathrm{fin}}$ after it & Sec.~\ref{sec:proof} & \texttt{precise\_check.py} (interval)\\
Estimates in Section~\ref{sec:overview} (including the argument of Norin and Steiner) and in Appendix~\ref{app:further} & Sec.~\ref{sec:overview}, App.~\ref{app:further} & \texttt{precise\_check.py} (exact, interval)\\
Optimality of the schedule within the framework & App.~\ref{app:further} & \texttt{bootstrap\_dp.py}\\
Location of the size thresholds (not needed for the proof) & Secs.~\ref{sec:bootstrap}, \ref{sec:proof} & \texttt{explicit\_thresholds.py}\\
Floating-point summary of Sections~\ref{sec:wollan} and~\ref{sec:lin} (not needed) & Secs.~\ref{sec:wollan}, \ref{sec:lin} & \texttt{refined\_lin\_check.py}\\\bottomrule
\end{tabular}
\caption{The numerical claims of the paper and the scripts that check them. Here ``exact'' means exact rational arithmetic, and ``interval'' means interval arithmetic, after the reduction to the size threshold described above.}\label{tab:scripts}
\end{table}

\paragraph{The scripts.} The scripts check the numerical inequalities in three ways. First, rational constants such as ${\Sst}$, ${\cst}$ and ${\cfin}$ are checked in exact rational arithmetic (\texttt{precise\_check.py}). Second, inequalities involving logarithms, square roots and exponentials at specific values are checked in interval arithmetic with outward rounding at $60$ digits (\texttt{precise\_check.py}, using the \texttt{mpmath} library), so these checks are certified despite rounding errors. Third, since numbers such as $\exp(e^{1161})$ are far too large to handle directly, each inequality required for all $t$ with $\log\log t$ at least a size threshold is rewritten in terms of $y:=\log\log t$, so that $\log t=e^y$. It is then reduced to the case where $y$ equals the size threshold. This case is checked in interval arithmetic. This reduction uses that, in each case, the inequality can be written as $f(y)\geqslant0$ with $f$ non-decreasing in $y$, so that the inequality at the size threshold implies it for all larger $y$. For most inequalities this is clear, because the terms that must be small decrease with $y$. Where it is not clear, the text proves it. The script \texttt{explicit\_thresholds.py} locates the size thresholds, and the script \texttt{refined\_lin\_check.py} summarises the computations of Sections~\ref{sec:wollan} and~\ref{sec:lin}. Both use floating-point arithmetic and are not needed for the proof. Table~\ref{tab:scripts} lists where each numerical claim of the paper is checked.
\end{sloppypar}

\paragraph{The Lean formalisation.} The formalisation of Theorem~\ref{thm:all} builds on the Lean library \texttt{LinearHadwiger}, which formalises the proof of Norin and Steiner and its external inputs (\url{https://github.com/snorin239/LinearHadwiger}, commit \texttt{9ef0db5}, Lean and Mathlib v4.32.1). This library accompanies their paper~\cite{NS}, and was produced by OpenAI Codex following guidance from Norin and Steiner. On top of this, Claude Opus 5.5 formalised the arguments of the present paper, and the following external results not in \texttt{LinearHadwiger}:
\begin{itemize}[nosep]
\item the linkage theorem of Thomas and Wollan~\cite[Corollary~1.2]{TW05} (Theorem~\ref{thm:TW}),
\item Theorem~1.6 of Thomas and Wollan~\cite{TW05}, which states that every graph with minimum degree at least $5k$ and at most $10k$ vertices has a $k$-linked subgraph with at least $2k$ vertices, and on which their proof of Theorem~\ref{thm:TW} relies,
\item the degree condition on which Theorem~1.6 of Thomas and Wollan relies, due to Egawa et al.~\cite{EFGISW00} and to Kawarabayashi, Kostochka and Yu~\cite[Theorem~1.1]{KKY06} (formalised in the minimum-degree form stated by Thomas and Wollan~\cite[Theorem~4.1]{TW05}, by adapting the proof of Kawarabayashi, Kostochka and Yu),
\item the sharpened form of Mader's theorem (Lemma~\ref{lem:mader}, proved in Section~\ref{sec:prelim}), since \texttt{LinearHadwiger} contains only the basic form, which gives a $k$-connected induced subgraph,
\item Nguyen's theorem~\cite[Theorem~1.1]{Ngu24}, and
\item the greedy colouring bound of Delcourt and Postle (Lemma~\ref{lem:greedy}).
\end{itemize}
\begin{sloppypar}
The final Lean statement, displayed in Section~\ref{sec:intro}, uses Mathlib's chromatic number \texttt{G.chromaticNumber} and the Hadwiger number \texttt{hadwigerNumber G}, which is defined from Mathlib notions in \texttt{MainMathlib.lean}. The library \texttt{LinearHadwiger} defines \texttt{chromatic G} as \texttt{G.chromaticNumber} converted to a natural number, which loses nothing since $G$ is finite. In \texttt{LinearHadwiger}, \texttt{cliqueMinorNumber G} is the largest $n$ such that $G$ has a $K_n$ model. The file \texttt{BSMain.lean} proves the bound of Theorem~\ref{thm:all} for \texttt{chromatic G} and \texttt{cliqueMinorNumber G}. The Lean files also prove that \texttt{hadwigerNumber G} equals \texttt{cliqueMinorNumber G}, and derive the final statement from this bound. Every numbered result of the paper is formalised, and every statement and proof in the paper follows the Lean files. Results of \texttt{LinearHadwiger} are stated in the paper in the forms formalised there. Published results that the paper cites without proof, such as Wollan's Theorems~2.1--2.3~\cite{Wol08}, are formalised in \texttt{LinearHadwiger} or in the list above.
\end{sloppypar}

\end{document}